\documentclass[11pt, a4paper]{article}

\usepackage[margin=2.4cm]{geometry}

\usepackage[margin=1cm, font=small, labelfont=bf]{caption}

\usepackage{amsfonts}
\usepackage{amsthm}
\usepackage{amsmath}
\usepackage{amssymb}
\usepackage{amscd}
\usepackage{mathrsfs}
\usepackage[mathscr]{eucal}
\usepackage{graphicx}
\usepackage{epstopdf} 
\usepackage{pict2e}
\usepackage{epic}
\usepackage{xcolor}
\usepackage{float}
\usepackage{mathtools}
\usepackage{centernot}

\usepackage{relsize}

\usepackage{cite}
\usepackage{hyperref}
\hypersetup{colorlinks=true, urlcolor=black, linkcolor=black, citecolor=blue}

\usepackage[utf8]{inputenc}
\usepackage[T1]{fontenc}
\usepackage{lmodern}
\usepackage{dsfont}
\usepackage{indentfirst}

\usepackage{hyperref}   
\hypersetup{
    linktoc=page,
    linkcolor=blue,          
    citecolor=blue,        
    filecolor=blue,      
    urlcolor=black,
   colorlinks=true           
}

\DeclareRobustCommand{\SkipTocEntry}[5]{}

\numberwithin{equation}{section}

\theoremstyle{plain}
\newtheorem{Thm}{Theorem}[section]
\newtheorem{Lem}[Thm]{Lemma}

\newtheorem{Coro}[Thm]{Corollary}
\newtheorem{Prop}[Thm]{Proposition}

\theoremstyle{definition}
\newtheorem{Def}[Thm]{Definition}

\newtheorem{Rem}[Thm]{Remark}
\newtheorem{Pb}[Thm]{Problem}

\definecolor{darkgreen}{rgb}{0,0.6,0.05}

\newcommand\n{\mathbf{n}}
\newcommand\m{\mathbf{m}}
\newcommand\sn{\partial\mathbf{n}}

\newcommand{\connect}{\xleftrightarrow}
\newcommand\fg{\mathfrak{g}}

\renewcommand{\epsilon}{\varepsilon}

\newcommand{\Z}{\mathbb{Z}}
\newcommand{\N}{\mathbb{N}}

\newcommand{\Triangle}{\ensuremath{\nabla}}

\def \p {{\partial}}

\def\<#1{\langle #1\rangle}

\def\bi{\begin{itemize}}  
\def\ei{\end{itemize}}
\def\bnum{\begin{enumerate}} 
\def\enum{\end{enumerate}}

\title{One-arm exponents of the high-dimensional Ising model}

\begin{document}
	
	\author{Diederik van Engelenburg\footnotemark[1]\footnote{TU Wien, Wien, Austria, \url{diederik.engelenburg@tuwien.ac.at}}\:, Christophe Garban\footnotemark[2]\footnote{Universit\'e Claude Bernard Lyon 1, CNRS UMR 5208, Institut Camille Jordan, 69622 Villeurbanne, France, and Courant Institute (NYU), New York, USA, \url{christophe.garban@gmail.com}}\:, Romain Panis\footnotemark[3]\footnote{Universit\'e Claude Bernard Lyon 1, Villeurbanne, France, \url{panis@math.univ-lyon1.fr}}\:,\\ Franco Severo\footnotemark[4]\footnote{CNRS and Sorbonne Universit\'e, Paris, France, \url{severo@lpsm.paris}}}

	\maketitle
	\abstract{We study the probability that the origin is connected to the boundary of the box of size $n$ (the one-arm probability) in several percolation models related to the Ising model. We prove that different universality classes emerge at criticality:
	\begin{enumerate}
		\item[$\bullet$] For the FK-Ising measure in a box of size $n$ with wired boundary conditions, we prove that this probability decays as $1/n$ in dimensions $d>4$, and as $1/n^{1+o(1)}$ when $d=4$.
		\item[$\bullet$] For the infinite volume FK-Ising measure, 
		we prove that this probability decays as $1/n^2$ in dimensions $d>6$, and as $1/n^{2+o(1)}$ when $d=6$.
		\item[$\bullet$] For the sourceless double random current measure, we prove that this probability decays as $1/n^{d-2}$ in dimensions $d>4$, and as $1/n^{2+o(1)}$ when $d=4$.
	\end{enumerate}
	Additionally, for the infinite volume FK-Ising measure, we show that the one-arm probability is $1/n^{1+o(1)}$ in dimension $d=4$, and at least $1/n^{3/2}$ in dimension $d=5$. This establishes that the FK-Ising model has upper-critical dimension equal to $6$, in contrast to the Ising model, where it is known to be less or equal to $4$, thus solving a conjecture of Chayes, Coniglio, Machta, and Shtengel \cite{chayes1999mean}.
	}

\setcounter{tocdepth}{2}
\tableofcontents

\section{Introduction}\label{sec:intro} 	

Understanding critical phenomena of lattice models is one of the main challenges of statistical mechanics. In the context of \emph{second-order} (or \emph{continuous}) phase transitions, Widom \cite{Widom1965equation} proposed that, near or at the critical point, thermodynamic quantities of interest should exhibit a power-law behaviour governed by \emph{critical exponents}. Computing these exponents, or even justifying their existence, is in general a very difficult problem.

In the framework of models defined on the hypercubic lattice $\mathbb Z^d$, a significant observation has been made by theoretical physicists in the twentieth century: above the so-called \emph{upper-critical} dimension $d_c$ of the model, the critical exponents simplify, in the sense that they match those derived in the simpler settings of trees or the complete graph. The dimensions $d\geq d_c$ form the \emph{mean-field regime} of the model.

\vspace{6pt}

We are interested in the emergence of mean-field features in the context of percolation models. The simplest and most studied example of a percolation model is given by independent percolation (or \emph{Bernoulli} percolation): given a parameter $p\in [0,1]$, we construct a random subgraph of $\mathbb Z^d$ by independently keeping (resp. deleting) each edge with probability $p$ (resp. $1-p$). The resulting law is denoted by $\mathbb P_p$. In dimensions $d\geq 2$, this model undergoes a phase transition for the existence of an infinite connected component at a parameter $p_c\in (0,1)$. It is widely believed that this phase transition is continuous, in the sense that there is no infinite cluster at criticality: $\theta(p_c):=\mathbb P_{p_c}[0\connect{}\infty]=0$. This qualitative property of the model has been established in dimension $d=2$ \cite{Kesten1980criticalproba}, and in dimensions $d>10$ \cite{HaraSlade1990Perco,HaraDecayOfCorrelationsInVariousModels2008,FitznervdHofstad2017Percod10}. Obtaining the corresponding result in the intermediate dimensions $3\leq d \leq 10$ (and especially $3\leq d\leq 6$) constitutes one of the main open problems in percolation theory. It is expected that the upper-critical dimension $d_c$ of Bernoulli percolation is equal to $6$. This means that several critical exponents cannot simultaneously take their mean-field value in dimensions $2\leq d \leq 5$ \cite{tasaki1987hyperscaling,ChayesChayesUpperCritDimPerco1987}, while they are predicted to do so in dimensions $d\geq 6$ \cite{AizenmanNewmanTreeGraphInequalities1984,HaraSlade1990Perco,BarskyAizenmanCriticalExponentPercoUnderTriangle1991,HaraDecayOfCorrelationsInVariousModels2008,FitznervdHofstad2017Percod10} (see also \cite{hutchcroft2022derivation} for conditional results when $d=6$). 

Among the numerous critical exponents, we focus on one which has been the object of intense study in the last thirty years: the \emph{one-arm exponent}. The vanishing of $\theta(p_c)$ is equivalent to the decay to $0$ as $n$ tends to infinity of the sequence of probabilities $\theta_n(p_c):=\mathbb P_{p_c}[0\connect{}\partial \Lambda_n]$, where $\Lambda_n:=[-n,n]^d\cap \mathbb Z^d$ and $\partial \Lambda_n$ is the vertex boundary of $\Lambda_n$. According to Widom's scaling hypothesis, this decay should be governed by a critical exponent. More precisely, it is conjectured that there exists $\rho=\rho(d)>0$ such that, for every $n\geq 1$,
\begin{equation}
	\mathbb P_{p_c}[0\connect{}\partial \Lambda_n]=\frac{1}{n^{\rho+o(1)}},
\end{equation}
where $o(1)$ tends to $0$ as $n$ tends to infinity. The critical exponent $\rho$ is often referred to as the
{\em one-arm exponent}. In a seminal work, Kozma and Nachmias \cite{KozmaNachmias2011OneArm} showed that $\rho=2$ in dimensions $d>6$ under the assumption that $\mathbb P_{p_c}[0\connect{}x]\asymp |x|^{2-d}$, where $\asymp$ means that the two quantities are within a multiplicative constant of each other. In fact, they showed the stronger result that $\theta_n(p_c)\asymp n^{-2}$. Presently, their assumption has been verified in dimensions $d>10$ using a technique called the \emph{lace expansion} \cite{HaraSlade1990Perco,HaraDecayOfCorrelationsInVariousModels2008,FitznervdHofstad2017Percod10}, see \cite{SladeSaintFlourLaceExpansion2006} for an introduction. It is expected to hold in every dimension $d> 6$. Let us mention that the assumption has also been verified in the context of sufficiently \emph{spread-out} percolation in dimensions $d>6$ \cite{HaraSladevdHofstad2003PercoSO,DumPan24Perco}.

Outside of the mean-field regime, the results are much more restricted. Using Smirnov's conformal invariance proof \cite{SmirnovCardy2001}, it is possible to show that $\rho=\frac{5}{48}$ for site percolation on the (two-dimensional) triangular lattice \cite{lawler2002one}. 

\vspace{6pt}

In this work, we study the one-arm exponent in various percolation models that arise in the study of the \emph{Ising model}, a cornerstone of statistical mechanics and one of its most renowned models. The Ising model has a rich history in its own right: we refer to \cite{DuminilreviewIsing} for a detailed review, and to \cite{FriedliVelenikIntroStatMech2017,DuminilLecturesOnIsingandPottsModels2019} for a mathematical introduction. 

Before moving to the definitions, we need some notation. If $x,y\in \mathbb Z^d$, we write $x\sim y$ if $|x-y|_2=1$ (where $|\cdot|_2$ is the $\ell^2$ norm on $\mathbb R^d$). If $\Lambda\subset \mathbb Z^d$, we write $E(\Lambda):=\{xy: x,y\in \Lambda, \: x\sim y\}$. 

The Ising model is formally defined as follows: given $G=(V,E)$  a finite subgraph of $\mathbb Z^d$, a \emph{boundary condition} $\tau\in\{-1,0,+1\}^{\mathbb Z^d}$, an \emph{interaction} $J=(J_{xy})_{xy\in E(\mathbb Z^d)}\in (\mathbb R^+)^{E(\mathbb Z^d)}$, and an (inhomogeneous) \emph{external magnetic field} $\mathsf{h}=(\mathsf{h}_x)_{x\in V}\in \mathbb R^V$, we construct a probability measure on $\{-1,+1\}^V$ by setting
\begin{equation}\label{eq:defIsing}
    \langle F(\sigma)\rangle_{G,J,\mathsf{h}}^\tau:=\frac{1}{\mathbf{Z}_{G,J,\mathsf{h},\tau}^{\textup{Ising}}}\sum_{\sigma\in \{\pm 1\}^V}F(\sigma)\exp\Big(\sum_{\substack{xy\in E}}J_{xy}\sigma_x\sigma_y+\sum_{\substack{x\in V\\y\notin V\\x\sim y}}J_{xy}\sigma_x\tau_y+\sum_{x\in V} \mathsf{h}_x\sigma_x\Big),
\end{equation}
where ${F:\{-1,+1\}^V\rightarrow \mathbb R}$, and $\mathbf{Z}_{G,J,\mathsf{h},\tau}^{\textup{Ising}}$ is the \emph{partition function} of the model which is such that $\langle 1\rangle_{G,J,\mathsf{h}}^\tau=1$. This defines the Ising model on $G$ with \emph{boundary condition} $\tau$ at parameters $J,\mathsf{h}$. If $J$ (resp. $\mathsf{h}$) is constant equal to $\beta\geq 0$ (resp. $h\in \mathbb R$), we often abuse notations and write $J=\beta$ (resp. $\mathsf{h}=h$). When $\mathsf{h}=0$, we simply write $\langle\cdot\rangle^\tau_{G,J,0}=\langle\cdot\rangle^\tau_{G,J}$. When $\tau\equiv 1$ (resp $\tau\equiv 0$), we write $\langle \cdot \rangle^\tau_{G,J,\mathsf{h}}=\langle \cdot \rangle^+_{G,J,\mathsf{h}}$ (resp. $\langle \cdot\rangle_{G,J,\mathsf{h}}$).

Let $\beta\geq 0$ and $h\in \mathbb R$. It is a well-known fact \cite{GriffithsCorrelationsIsing1-1967,GriffithsCorrelationsIsing2-1967} that $\langle \cdot\rangle_{G,\beta,h}^+$ and $\langle \cdot\rangle_{G,\beta,h}$ admit (weak) limits when $G\nearrow\mathbb Z^d$. We denote them by $\langle \cdot \rangle_{\beta,h}^+$ and $\langle \cdot\rangle_{\beta,h}$ respectively. In dimensions $d\geq 2$, the model undergoes a \emph{phase transition} \cite{PeierlsIsing1936} for the vanishing of the \emph{magnetisation} at a critical parameter $\beta_c\in(0,\infty)$ defined by
\begin{equation}
	\beta_c:=\inf\Big\{\beta\geq 0:  m^*(\beta):=\langle \sigma_0\rangle_\beta^+>0\Big\}.
\end{equation}
The phase transition is continuous in the sense that $m^*(\beta_c)=0$, see \cite{Yang1952spontaneous,AizenmanFernandezCriticalBehaviorMagnetization1986,WernerPercolationEtModeledIsing2009,AizenmanDuminilSidoraviciusContinuityIsing2015}. Moreover, it is expected that $d_c=4$ (observe that it differs from Bernoulli percolation, where one expects $d_c=6$). Indeed, it is known that $d_c>2$ \cite[Proposition~8.2]{AizenmanGeometricAnalysis1982}, and that $d_c\leq 4$ \cite{AizenmanGeometricAnalysis1982,FrohlichTriviality1982,AizenmanDuminilTriviality2021}.
Tasaki \cite{tasaki1987hyperscaling} proved that---conditionally on their existence---the critical exponents of the Ising model cannot all take their mean-field value when $d=3$. This provides a conditional proof of the equality $d_c=4$, but the unconditional result remains open. Let us mention that on the physics side, the so-called $\varepsilon$-expansion of Wilson and Fisher \cite{WilsonFisher1972dcexpansion}, or the conformal bootstrap approach \cite{RychkovNonGaussianity2017}, both support the conjecture that $d_c>3$.

Among the various approaches to studying the Ising model, one of the most successful has been the use of geometric representations. Notably, several percolation representations of the model have been proposed and proved to be powerful tools for its study. We focus on two fundamental examples: the \emph{Fortuin--Kasteleyn representation} (or \emph{FK-Ising model}), and the (double) \emph{random current representation} of the model.

Unlike (independent) Bernoulli percolation, these models are correlated, which significantly complicates their analysis. However, as we are about to see, this is also responsible for a richer critical behaviour, which we now describe. We postpone the precise definitions to the next subsections and give a brief account of the main results of this paper.
\begin{enumerate}
\item[$\bullet$] If we consider the critical finite volume FK-Ising measure with wired boundary conditions $\phi^1_{\Lambda_n,\beta_c}$, we find that $\rho=1$ in dimensions $d\geq 4$. In dimensions $d>4$, we obtain the stronger result that $\phi^1_{\Lambda_n,\beta_c}[0\connect{}\partial \Lambda_n]\asymp n^{-1}$. This improves and completes an earlier work of Handa, Heydenreich, and Sakai \cite{handa2019mean} in which the authors argued that, if $\rho$ exists and $d>4$, then $\rho\leq 1$. 
\item[$\bullet$]  On the other hand, if we consider the critical infinite volume FK-Ising measure $\phi_{\beta_c}=\phi_{\mathbb Z^d,\beta_c}$, we find that $\rho=2$ in dimensions $d\geq 6$. In dimensions $d>6$, we obtain the stronger result that $\phi_{\beta_c}[0\connect{}\partial \Lambda_n]\asymp n^{-2}$. This matches the results obtained in the context of Bernoulli percolation \cite{KozmaNachmias2011OneArm}.
\item[$\bullet$] Still for the measure $\phi_{\beta_c}$, we find that $\rho= 1 $ in dimension $d=4$, and that $\rho\leq \tfrac{3}{2}$ in dimension $d=5$. This proves that---unlike the Ising model---the upper-critical dimension of the FK-Ising model satisfies $d_c^{\textup{FK-Ising}}=6>4\geq d_c^{\textup{Ising}}$, and solves a conjecture of Chayes, Coniglio, Machta, and Shtengel \cite{chayes1999mean}.
\item[$\bullet$] In the case of the critical \emph{sourceless} double random current measure $\mathbf P^{\emptyset,\emptyset}_{\beta_c}=\mathbf P^{\emptyset,\emptyset}_{\mathbb Z^d,\mathbb Z^d,\beta_c}$, we find that $\rho(d)=d-2$ in dimensions $d\geq 4$. In dimensions $d>4$, we obtain the stronger result that $\mathbf P^{\emptyset,\emptyset}_{\beta_c}[0\connect{}\partial \Lambda_n]\asymp n^{2-d}$, which matches results previously obtained in the context of \emph{loop} percolation \cite{chang2016phase}.
\end{enumerate}

The methods developed in this paper are robust. In a companion paper \cite{vEGPSperco}, we apply them to spread-out Bernoulli percolation and obtain an alternative proof of the results of \cite{KozmaNachmias2011OneArm,ChatterjeeHanson2020Halfspace}.

Let us also mention that we expect the techniques introduced in this paper to be useful in the study of the one-arm exponent for other Ising-type systems---including the $\varphi^4$ model---for which similar FK and random current type representations have recently been constructed \cite{GunaratnamPanagiotisPanisSeveroPhi42022,gunaratnam2025supercritical}.

\paragraph{Notations.} If $x=(x_1,\ldots,x_d)\in \mathbb R^d$, we let $|x|:=\max_{1\leq i\leq d}|x_i|$ denote the infinite norm of $x$. For every $k\geq 1$, we let $\Lambda_k:=[-k,k]^d\cap \mathbb Z^d$. Let $\mathbf{e}_i$ be the basis vector of $i$-th coordinate equal to $1$.
Throughout the paper various dimensional-dependent constant appear without playing any particular role. We will therefore use the following compact notation: we write $f\lesssim g$ (or $g\gtrsim f$) if there exists a constant $C\in (0,\infty)$ which only depends on the dimension $d$ such that $f\leq Cg$. If $f\lesssim g$ and $g\lesssim f$, we write $f\asymp g$.

\subsection{One-arm exponents of the FK-Ising model}\label{sec:introFK} 
	
The Ising model is related to a dependent percolation model: the FK-Ising model. We now define it and refer to the monograph \cite{Grimmett2006RCM} or the lecture notes \cite{DuminilLecturesOnIsingandPottsModels2019} for more information.

Let $G=(V,E)$ be a finite subgraph of $\mathbb Z^d$. Fix $\xi\in \{0,1\}^{E(\mathbb Z^d)\setminus E}$. If $\omega \in \{0,1\}^{E}$ is a fixed percolation configuration, we let $k^\xi(\omega)$ be the number of connected components of the graph with vertex set $\mathbb Z^d$ and edge set $\omega \vee \xi$ that intersect $V$. The FK-Ising model at inverse temperature $\beta\geq 0$ with boundary condition $\xi$ is the measure on $\{0,1\}^{E}$ given by
\begin{equation}
	\phi^\xi_{G,\beta}[\omega]:=\frac{2^{k^\xi(\omega)}\prod_{xy\in E}(e^{2\beta}-1)^{\omega_{xy}}}{\mathbf{Z}_{G,\beta,\xi}^{{\textup{FK}}}},
\end{equation}
where $\mathbf{Z}_{G,\beta,\xi}^{{\textup{FK}}}$ is the partition function of the model, which guarantees that $\phi^\xi_{G,\beta}$ is a probability measure. We let $\phi^1_{G,\beta}$ and $\phi_{G,\beta}^0$ denote respectively the cases $\xi\equiv 1$ and $\xi\equiv 0$. It is classical that the Ising model and the FK-Ising model can be coupled together through the Edward--Sokal coupling \cite{edwards1988generalization} (see also \cite{DuminilLecturesOnIsingandPottsModels2019}). This provides the following fundamental relations: for $x,y\in V$, 
\begin{equation}\label{eq: consequences of ESC}
	\phi_{G,\beta}^0[x\connect{}y]=\langle \sigma_x\sigma_y\rangle_{G,\beta}, \quad \phi_{G,\beta}^1[x\connect{}y]=\langle \sigma_x\sigma_y\rangle_{G,\beta}^+, \quad \phi^1_{G,\beta}[x\connect{}\partial G]=\langle \sigma_x\rangle_{G,\beta}^+,
\end{equation}
where we recall that $\partial G$ is the vertex boundary of $G$ defined by $\partial G=\{x\in V: \exists y \notin V, \: x\sim y\}$.
Again, one can construct the (weak) limits of these measures as $G\nearrow\mathbb Z^d$. We denote them by $\phi^1_\beta$ and $\phi^0_\beta$ respectively. Observe that $m^*(\beta)=\phi^1_{\beta}[0\connect{}\infty]$. The FK-Ising model undergoes a phase transition for the existence of an infinite cluster at the parameter $\beta_c$ introduced above. That is, we may alternatively define $\beta_c$ as follows:
\begin{equation}
	\beta_c=\inf\Big\{\beta \geq 0 : \phi^1_\beta[0\connect{}\infty]>0\Big\}.
\end{equation}
It is well-known (see \cite{BodineauTranslationGibbsIsing2006,AizenmanDuminilSidoraviciusContinuityIsing2015,RaoufiGibbsMeasures}) that $\phi_{\beta}^1=\phi_{\beta}^0$ for all $\beta>0$. We denote the unique (translation invariant) infinite volume measure by $\phi_\beta$. We are now in a position to state our results, starting with the case of wired boundary conditions. Below, we identify a set $\Lambda\subset \mathbb Z^d$ with the graph $(\Lambda,E(\Lambda))$.

\subsubsection{Wired boundary conditions}\label{sec:introwired}

Recall from the Edwards--Sokal coupling \eqref{eq: consequences of ESC} that 
\begin{equation}\label{eq:ESC intro mag}
\phi^1_{\Lambda_n,\beta_c}[0\connect{}\partial \Lambda_n]=\langle \sigma_0\rangle_{\Lambda_n,\beta_c}^+,
\end{equation}
so that all the results below can also be formulated in terms of the Ising model directly. Our first result is a mean-field lower bound on the one-arm probability that is valid in every dimensions $d\geq 2$. It improves on a previous result of Handa, Heydenreich, and Sakai \cite{handa2019mean} (which relies on the existence of the exponent $\rho$ and requires $d>4$).

\begin{Thm}\label{thm:FKwired meanfieldLB} Let $d\geq 2$. There exists $c>0$ such that, for every $n\geq 1$,
\begin{equation}\label{eq:FKwired meanfieldLB}
	\phi^1_{\Lambda_n,\beta_c}[0\connect{}\partial \Lambda_n]\geq \frac{c}{n}.
\end{equation}	
\end{Thm}
The above bound is not very useful in $d=2$ where a combination of integrability \cite{Wu1966theory,McCoyWu1973two} and Russo--Seymour--Welsh theory in its non-strong form \cite{duminil2011connection} gives $\rho=\frac{1}{8}$. However, it provides a non-trivial bound in $d=3$, that was already derived in \cite{tasaki1987hyperscaling} thanks to the inequalities $\phi^1_{\Lambda_n,\beta_c}[0\connect{}\partial \Lambda_n]\geq \sqrt{\phi_{\beta_c}[0\connect{}2n\mathbf{e}_1]}$ and $\phi_{\beta_c}[0\connect{}n\mathbf{e}_1]\gtrsim n^{-2}$ (see e.g.\cite{SimonInequalityIsing1980}). If one additionally assumes the existence of the critical exponent $\eta$ associated with the decay of the critical two-point function (see \eqref{eq:def eta} below), 
\cite[Theorem~1.5]{DuminilPanis2024newLB} gives the following improved lower bound in $d=3$: 
\begin{equation}
\phi^1_{\Lambda_n,\beta_c}[0\connect{}\partial\Lambda_n]\gtrsim n^{-3/4+o(1)},
\end{equation}
where $o(1)$ tends to $0$ as $n$ tends to infinity. 

As it turns out, \eqref{eq:FKwired meanfieldLB} is sharp (up to a constant multiplicative factor) in dimensions $d>4$.

\begin{Thm}\label{thm:FKwired d>4} Let $d>4$. There exist $c,C>0$ such that, for every $n\geq 1$,
\begin{equation}
	\frac{c}{n}\leq \phi^1_{\Lambda_n,\beta_c}[0\connect{}\partial \Lambda_n]\leq \frac{C}{n}.
\end{equation}
\end{Thm}
\begin{Rem} For every $M>1$, one has
\begin{equation}
	\phi^1_{\Lambda_n,\beta_c}[0\connect{}\partial \Lambda_n]\geq \phi^1_{\Lambda_{Mn},\beta_c}[0\connect{}\partial \Lambda_n]\geq \phi^1_{\Lambda_{Mn},\beta_c}[0\connect{}\partial \Lambda_{Mn}],
\end{equation}
where the first inequality follows from Markov's property and monotonicity in boundary conditions (see \cite{DuminilLecturesOnIsingandPottsModels2019}), and the second from inclusion of events. Using Theorem \ref{thm:FKwired d>4}, this gives that $\phi^1_{\Lambda_{Mn},\beta_c}[0\connect{}\partial \Lambda_n]\asymp_M 1/n$ in dimensions $d>4$.
\end{Rem}

The computation of the wired one-arm exponent was previously performed in the context of trees. Before stating the result obtained there, we introduce some notations. If $d\geq 1$, we let $\mathbb T^d$ denote the $(d+1)$-regular tree. For every $x,y\in \mathbb T^d$, let $\textup{d}(x,y)$ denote the length of the shortest past connecting $x$ and $y$. We root $\mathbb T^d$ at a vertex $o$ and let $B_n:=\{x\in \mathbb T^d: \textup{d}(o,x)\leq n\}$.  It was shown in \cite{heydenreich2018critical} that one has $\phi^1_{B_n,\beta_c}[o\connect{}\partial B_n]\asymp n^{-1/2}$ (see also the recent \cite{ventura2024non} for an exact asymptotic estimate, as well as for an extension to quenched supercritical Galton--Watson trees). The mismatch of exponents between the case of trees and the high-dimensional setting is not in contradiction with the discussion on the mean-field regime above. It can be solved by modifying the notion of distance in the tree. The proper metric to consider is \textbf{d}$(o,x):=\sqrt{\textup{d}(o,x)}$. We refer to \cite{KozmaNachmias2009AlexanderOrbach} for a discussion.

As explained above, it is conjectured that the upper-critical dimension of the Ising model satisfies $d_c^{\textup{Ising}}=4$. The dimension $d=4$ is often referred to as \emph{marginal} and various quantities present logarithmic divergences. This reflects the fact that the model is ``barely'' mean-field at $d_c=4$.  
\begin{Thm}\label{thm:FKwired d=4} Let $d=4$. There exist $c,C>0$ such that, for every $n\geq 2$,
\begin{equation}
	\frac{c}{n}\leq \phi^1_{\Lambda_n,\beta_c}[0\connect{}\partial \Lambda_n]\leq \frac{C (\log n)^{5/2}}{n}.
\end{equation}
\end{Thm}
\begin{Rem} When $d=4$, we conjecture the existence of $\hat{\rho}\geq 0$ such that
\begin{equation}
\phi^{1}_{\Lambda_n,\beta_c}[0\connect{}\partial \Lambda_n]=\frac{(\log n)^{\hat{\rho}+o(1)}}{n},
\end{equation}
where $o(1)$ tends to $0$ as $n$ tends to infinity. We expect neither of the bounds in Theorem \ref{thm:FKwired d=4} to be of the sharp order, i.e. we conjecture that $\hat{\rho}\in (0,5/2)$.
\end{Rem}

The upper bounds in Theorems~\ref{thm:FKwired d>4} and \ref{thm:FKwired d=4} are obtained through on a renormalisation argument, which relies on the bounds on the magnetisation with external magnetic field obtained in \cite{AizenmanFernandezCriticalBehaviorMagnetization1986}. The successful implementation of the renormalisation argument relies on \cite{DuminilPanis2024newLB}. Theorem \ref{thm:FKwired meanfieldLB} relies on the mean-field lower bound on the magnetisation obtained in \cite{AizenmanBarskyFernandezSharpnessIsing1987} (see also \cite{DuminilTassionNewProofSharpness2016}), together with a simple integration argument. We present a more detailed overview of the proofs in Section \ref{sec:introStrategy}.

Theorems \ref{thm:FKwired d>4} and \ref{thm:FKwired d=4} have implications in the study of the relaxation time of critical Glauber dynamics of the Ising model. When $d>4$, Hu \cite{hu2025polynomial} recently established---building on \cite{bauerschmidt2024log} and conditionally on a bound of the form $\phi^1_{\Lambda_n,\beta_c}[0\connect{}\partial \Lambda_n]\lesssim n^{-\varepsilon}$ for some $\varepsilon>0$---that the temporal correlation between the spins $\sigma_x(t=0)$ and $\sigma_x(t=T)$ under an infinite volume Glauber dynamics at criticality is algebraically decaying in $T$. We present another interesting consequence of these results which has to do with the computation of the \emph{mixing rate} exponent $\iota$ introduced by Duminil-Copin and Manolescu in \cite{DuminilMano2022Scaling}. In the same paper, they conjecture the existence of $\iota=\iota(d)>0$ such that, for every $n\geq 1$,
\begin{equation}
	\sup_{e\in E(\Lambda_n)}\Big(\phi^1_{\Lambda_{2n},\beta_c}[\omega_e]-\phi^0_{\Lambda_{2n},\beta_c}[\omega_e]\Big)=\frac{1}{n^{\iota +o(1)}},
\end{equation}
where $o(1)$ tends to $0$ as $n$ tends to infinity. Roughly, the above quantity measures how the state of a given edge is influenced by boundary conditions. While $\iota$ plays a crucial role in the derivation of the \emph{scaling relations} of the planar random cluster model in \cite{DuminilMano2022Scaling}, it is not yet clear to us how relevant it is in the study of the high-dimensional Ising (or FK-Ising) model. Nevertheless, we obtain the following result.

\begin{Coro}\label{coro:iota d>4} Let $d\geq 4$. There exists $C>0$ such that, for every $n\geq 2$,
\begin{align}\label{}
\sup_{e\in E(\Lambda_n)}\Big(\phi^1_{\Lambda_{2n},\beta_c}[\omega_e]-\phi^0_{\Lambda_{2n},\beta_c}[\omega_e]\Big)&\leq \frac{C}{n^2} \,,  &\text{if } d>4, \\
\sup_{e\in E(\Lambda_n)}\Big(\phi^1_{\Lambda_{2n},\beta_c}[\omega_e]-\phi^0_{\Lambda_{2n},\beta_c}[\omega_e]\Big)&\leq \frac{C(\log n)^5}{n^{2}} \,, ~~~~~ &\text{if } d=4.
\end{align}
\end{Coro}
Let us mention that these bound are expected to be \emph{sharp} (up to a logarithmic factor in $d=4$), in the sense that $\iota$ should be equal to $2$ in every dimensions $d\geq 4$. A brief justification to this fact can be obtained by the following scaling relation (which is derived by extrapolating on \cite[(\textbf{R7})]{DuminilMano2022Scaling}): 
\begin{equation}
\nu(d\wedge d_c^{\textup{Ising}}-\iota)=1,
\end{equation}
where $\nu$ is the correlation length exponent. Using one of the results of \cite{DuminilPanis2024newLB} which gives $\nu=\frac{1}{2}$, together with the expected value $d_c^{\textup{Ising}}=4$, this gives the conjectural value $\iota=2$ in dimensions $d\geq 4$.
	
\subsubsection{Free boundary conditions and infinite volume}
We now turn to results regarding the infinite volume measure $\phi_{\beta_c}$.
Our first result can be compared with \cite{KozmaNachmias2011OneArm}. It suggests 
that dimension $6$ plays a particular role for the FK-Ising model.  
\begin{Thm}\label{thm:FKfree d>6} Let $d>6$. There exist $c,C,M>0$ such that, for every $n\geq 1$,
\begin{equation}\label{eq:thmFKfree d>6}
	\frac{c}{n^2}\leq \phi^0_{\Lambda_{Mn},\beta_c}[0\connect{}\partial \Lambda_n]\leq \phi_{\beta_c}[0\connect{}\partial \Lambda_n]\leq \frac{C}{n^2}.
\end{equation}
\end{Thm}
The middle inequality in \eqref{eq:thmFKfree d>6} follows from the fact that $\phi_{\beta_c}$ stochastically dominates $\phi^0_{\Lambda_{Mn},\beta_c}$. Let us add that \eqref{eq:thmFKfree d>6} with $M=1$ would follow with the same method if we knew a sharp upper bound on the two-point function restricted to a half-space.

\smallskip
It is conjectured (and proved below) that---just as for Bernoulli percolation---the upper-critical dimension of the (free) FK-Ising model is equal to $6$ \cite{chayes1999mean,fang2022geometric,fang2023geometric,wiese2024two}. It is therefore not surprising to observe logarithmic corrections in the (conjecturally marginal) dimension $d=6$. 

\begin{Thm}\label{thm:FKfree d=6} Let $d=6$. There exist $c,C,M>0$ such that, for every $n\geq 2$,
\begin{equation}
	\frac{c}{n^2}\leq \phi^0_{\Lambda_{M n},\beta_c}[0\connect{}\partial \Lambda_n]\leq \phi_{\beta_c}[0 \leftrightarrow \partial \Lambda_n]  \leq \frac{C\log n}{n^2}.
\end{equation}
\end{Thm}

The proofs of the upper bounds in Theorems \ref{thm:FKfree d>6} and \ref{thm:FKfree d=6} are based on the entropic method introduced by Dewan and Muirhead \cite{DewanMuirhead} in the context of Gaussian percolation (which includes Bernoulli percolation). 
We also (again) heavily rely on the correlation length computations of \cite{DuminilPanis2024newLB}. The lower bounds are based on second-moment methods. We provide more details on the proofs in Section \ref{sec:introStrategy}.

Theorems \ref{thm:FKfree d>6} and \ref{thm:FKfree d=6}, together with estimates on the critical two-point function $\phi_{\beta_c}[0\connect{}x]$, allow us to identify the critical exponent $\delta$ which describes the volume of a critical cluster. More precisely, $\delta$ is defined by 
\begin{equation}
	\phi_{\beta_c}[|\mathcal C(0)|\geq n]=\frac{1}{n^{1/\delta +o(1)}},
\end{equation}
where $\mathcal C(0)$ denotes the cluster of the origin, and $o(1)$ tends to $0$ as $n$ tends to infinity. We find that $\delta=2$, which matches previous results \cite{AizenmanBarsky1987sharpnessPerco,BarskyAizenmanCriticalExponentPercoUnderTriangle1991} (see also \cite{hutchcroft2022derivation}) obtained for Bernoulli percolation under the so-called \emph{triangle condition} (which is expected to hold in dimensions $d>6$).
\begin{Thm}\label{thm:volume_exp}
Let $d\geq6$. There exist $c,C>0$ such that, for every $n\geq 2$,  
\begin{align}\label{}
\frac{c}{n^{1/2}} \leq\, &\phi_{\beta_c}[|\mathcal C(0)| \geq n]  \leq \frac{C}{n^{1/2}} \,,  &\text{if } d>6, \\
\frac{c(\log n)^{-1/2}}{n^{1/2}} \leq\, &\phi_{\beta_c}[|\mathcal C(0)| \geq n]  \leq \frac{C(\log n)^{1/2}}{n^{1/2}} \,, ~~~~~ &\text{if } d=6.
\end{align}

\end{Thm}

The lower bounds above can also be obtained for the appropriate finite volume measures (similarly to Theorems \ref{thm:FKfree d>6} and \ref{thm:FKfree d=6}), see Remark~\ref{rem:vol_lower-bound_finite-vol}.

\vspace{6pt}

Our next results concern the value of $\rho$ in dimensions $d<6$. We obtain that $\rho<2$ in dimensions $d=4,5$ (with a conditional result in $d=3$). This confirms the aforementioned prediction that $d_c^{\textup{FK-Ising}}=6$. We discuss this new result more extensively below.

\begin{Thm}\label{thm: FKfree d=5} Let $d=5$. There exists $c>0$ such that, for every $n\geq 1$,
\begin{equation}
	\phi_{\beta_c}[0\connect{}\partial \Lambda_n]\geq \frac{c}{n^{3/2}}.
\end{equation}	
\end{Thm}

 Perhaps surprisingly, in the case $d=4$, we can rely on Theorem~\ref{thm:FKwired d=4} to prove that the one-arm exponent for the infinite volume measure takes the same value $\rho=1$ as for the wired measure. 
 \begin{Thm} \label{thm: FKfree d=4} Let $d=4$. There exist $c,C>0$ such that, for every $n\geq 2$,
\begin{align}
\frac{c (\log n)^{-3/2}}{n} \leq 
\phi_{\beta_c}[0 \leftrightarrow \partial \Lambda_n]  \leq \frac{C(\log n)^{5/2}}{n}.
\end{align}
\end{Thm}
Our result in $d=3$ is conditional and based on the existence of the critical exponent $\eta$ defined as follows:
\begin{equation}\label{eq:def eta}
	\langle \sigma_0\sigma_x\rangle_{\beta_c}=\phi_{\beta_c}[0\connect{}x]=\frac{1}{|x|^{d-2+\eta+o(1)}},
\end{equation}
where $o(1)$ tends to $0$ as $|x|$ tends to infinity.
The critical exponent $\eta$ has recently been shown to be equal to $0$ in dimensions $d\geq 4$, see \cite{DuminilPanis2024newLB}. In dimension $d=3$---conditionally on its existence---\cite{DuminilPanis2024newLB} shows that $\eta\leq \tfrac{1}{2}$. Using this input, we can prove the following result.

\begin{Thm}\label{thm: FKfree d=3} Let $d=3$. Assume that the critical exponent $\eta$ defined in \eqref{eq:def eta} exists. Then, for every $n\geq 1$,
\begin{equation}
	\phi_{\beta_c}[0\connect{}\partial\Lambda_n]\geq \frac{1}{n^{5/4+o(1)}},
\end{equation}
where $o(1)$ tends to $0$ as $n$ tends to infinity.
\end{Thm}
Obtaining a polynomial upper bound on $\phi_{\beta_c}[0\connect{}\partial \Lambda_n]$ (or $\phi^1_{\Lambda_n,\beta_c}[0\connect{}\partial \Lambda_n]$) for $d=3$ remains an open problem. 

As mentioned above, when $d=2$, the one-arm exponent can be computed by relying on integrability \cite{Wu1966theory,McCoyWu1973two} and RSW theory \cite{duminil2011connection}: one has $\rho = \frac 1 8$, regardless the boundary conditions. Here, RSW theory justifies the fact that, in $d=2$, one has
\begin{equation}\label{eq:what hyperscaling should give}
	\phi_{\beta_c}[0\connect{}2n\mathbf{e}_1]\asymp \phi_{\beta_c}[0\connect{}\partial \Lambda_n]^2,
\end{equation}
where $\mathbf{e}_1=(1,0,\ldots,0)\in \mathbb R^d$. In words, the above relation says that ``conditionally on reaching distance $n$, two points at distance $\asymp n$ of one another have a positive probability of lying in the same cluster''. A version of such a statement---which is related to \emph{hyperscaling}---is expected to hold for all dimensions $d<6$, and up to a logarithmic factor in $d=6$. In dimension $d=4$, we know from \cite{DuminilPanis2024newLB} that $\phi_{\beta_c}[0\connect{}2n\mathbf{e}_1]=\langle \sigma_0\sigma_{2n\mathbf{e}_1}\rangle_{\beta_c}\asymp n^{-2-o(1)}$. Hence, Theorem \ref{thm: FKfree d=4} can be seen as a (weak) evidence that \eqref{eq:what hyperscaling should give} is indeed true (up to a logarithmic correction) in dimension $d=4$. On the basis of \eqref{eq:what hyperscaling should give}, we conjecture that $\rho=3/2$ in dimension $d=5$. In dimensions $d>6$, \eqref{eq:what hyperscaling should give} does not hold anymore, and a polynomial correction is necessary. Indeed, combining \cite{DuminilPanis2024newLB} and Theorem \ref{thm:FKfree d>6} gives
\begin{equation}\label{eq:no hyperscaling}
	\phi_{\beta_c}[0\connect{}2n\mathbf{e}_1]\asymp n^{6-d}\phi_{\beta_c}[0\connect{}\partial \Lambda_n]^2.
\end{equation}

The FK-Ising model is one of the very few examples of a percolation model for which it is possible to compute a critical exponent \emph{at} (Theorem \ref{thm:FKfree d=6}) and \emph{below} (Theorems \ref{thm: FKfree d=4} and \ref{thm: FKfree d=5}) the upper-critical dimension. Relatedly, significant progress has been made regarding the percolation induced by the excursion sets of the Gaussian free field on metric graphs. Among the main successes is the computation, in the intermediate dimensions $3\leq d \leq 5$, of the one-arm exponent (up to a multiplicative constant) \cite{ding2020percolation,drewitz2023arm,cai2024onelow,drewitz2025critical} (see also \cite{cai2025high} for $d>6$) and the volume of the critical cluster \cite{cai2024incipient,drewitz2024cluster}. See also \cite{cai2024quasi} for a proof of a stronger version of \eqref{eq:no hyperscaling} in this setting. In the context of long-range Bernoulli percolation, let us also mention the impressive recent series of works of Hutchcroft \cite{hutchcroft2024pointwise,hutchcroft2025criticalI,hutchcroft2025criticalII,hutchcroft2025criticalIII}, who (in particular) manages to exactly identify several logarithmic corrections at the upper-critical dimension.

\vspace{6pt}

\paragraph{On the upper-critical dimensions of Ising and FK-Ising.} 

We conclude this subsection by discussing an important consequence of the above results: the upper-critical dimensions of the Ising model and its FK representation are distinct. Indeed, our results prove that $d_c^{\textup{FK-Ising}}=6$, in the sense that the one-arm exponent is mean-field in dimensions $d\geq 6$, but not in dimensions $d<6$. 
A (perhaps) non-intuitive consequence of this result is that, in dimensions $d=4,5$, the observables of the spin model \emph{are} mean-field, while some observables of the FK-Ising model \emph{are not}. For the case of the one-arm probability $\phi_{\beta_c}[0\connect{}\partial \Lambda_n]$, this result is justified by the observation that---unlike $\phi^1_{\Lambda_n,\beta_c}[0\connect{}\partial \Lambda_n]$---it cannot be easily represented in terms of correlation functions of the model.

To the best of our knowledge, this surprising discrepancy in upper-critical dimensions was first conjectured by Chayes, Coniglio, Machta, and Shtengel \cite{chayes1999mean}, building upon earlier observations \cite{coniglio1982solvent}. In their work, they consider the FK-Ising model on the Bethe lattice. They construct a notion of supercritical (i.e.\ $\beta>\beta_c$) correlation length and prove that the associated critical exponent $\tilde{\nu}'$ is equal to $\tfrac{1}{4}$. This value does not match the predicted one for the truncated correlations of the Ising model: $\nu'=\tfrac{1}{2}$. Relying on a \emph{hyperscaling} relation, they then observe that the value of $\tilde{\nu}'$ is consistent with an upper-critical dimension $d_c=6$, thereby justifying their conjecture. In our paper, we do not address the (interesting) question of the computation of $\tilde{\nu}'$ in dimensions $d>6$. However, we confirm that the upper-critical dimension of the FK-Ising model is indeed equal to $6$.

On the physics side, recent works have also been interested in establishing that the upper-critical dimension of the FK-Ising model differs from the Ising one, see for instance \cite{fang2022geometric,fang2023geometric,wiese2024two}. Such a discrepancy in upper-critical dimensions between a spin model and its geometric representation is not a specificity of the Ising model, and is also expected to hold for Potts models \cite{wiese2024two}. Interestingly, it is also believed to be the case for the so-called {\em arboreal gas} (which can be viewed as the ``$0-$Potts model'') whose associated spin model is the hyperbolic $\mathbb{H}^{0|2}$ model, see \cite{bauerschmidt2024percolation}.

\subsection{One-arm exponent of the sourceless double random current}\label{sec:introDRC}

Our last object of study is the sourceless double random current measure, which we quickly introduce here---see Section~\ref{sec:RCR} for more details. Let $G=(V,E)$ be a subgraph of $\Z^d$. The single (sourceless) random current measure is the probability measure on current configurations $\n \in \N^E=\{0,1,\ldots\}^E$ defined by 
\begin{equation*}
	\mathbf{P}^\emptyset_{G,\beta}[\n]=\frac{1}{Z^\emptyset_{G,\beta}}\mathds{1}_{\partial \n =\emptyset}\prod_{\substack{e\in E}}\dfrac{\beta^{\n_{e}}}{\n_{e}!},
\end{equation*}
where $\partial \n :=\{x\in V:~ \sum_{xy \in E} \n_{xy} \text{ is odd}\}$ is the set of \emph{sources} of $\n$, and $Z^{\emptyset}_{G,\beta}$ is the partition function of the model. The double random current is simply the product measure $\mathbf{P}^{\emptyset,\emptyset}_{G,\beta} := 	\mathbf{P}^\emptyset_{G,\beta} \otimes 	\mathbf{P}^\emptyset_{G,\beta}$. We regard $\mathbf{P}^{\emptyset,\emptyset}_{G,\beta}$ as a percolation model by declaring each edge $e$ open in a double current configuration $(\n,\m)$ if and only if $\n_e+\m_e>0$.

Similarly to the FK-Ising model introduced in the previous section, the connectivity properties of the double random current are related to correlations of the Ising model. However, for the double random current this relation is not provided by a coupling, but rather by a combinatorial identity known as the \emph{switching lemma}---see Lemma~\ref{lem:switching}. 
In particular, the point-to-point connection probability of the double random current exactly corresponds to the \emph{square} of the two-point function of the Ising model, that is
	\begin{equation}\label{eq: connection 2pt current 2pt Ising}
		\mathbf P^{\emptyset,\emptyset}_{G,\beta}[x\connect{}y]=\langle \sigma_x\sigma_y\rangle_{G,\beta}^2.
	\end{equation}
Once again, it is possible to construct the weak limit of the above measure as $G\nearrow \mathbb Z^d$ \cite{AizenmanDuminilSidoraviciusContinuityIsing2015}, we denote it by $\mathbf P^{\emptyset,\emptyset}_{\beta}$. One may also pass \eqref{eq: connection 2pt current 2pt Ising} to the infinite volume limit: for every $x,y\in \mathbb Z^d$, $\mathbf P^{\emptyset,\emptyset}_{\beta}[x\connect{}y]=\langle\sigma_x\sigma_y\rangle_\beta^2$. This allows to prove that $\mathbf P^{\emptyset,\emptyset}_{\beta}$ and the Ising model undergo a phase transition at the same parameter $\beta_c$. Moreover, still by \cite{AizenmanDuminilSidoraviciusContinuityIsing2015}, the phase transition is continuous in the sense that $\mathbf P^{\emptyset,\emptyset}_{\beta_c}[0\connect{}\infty]=0$.
		
	We are interested in the study of the measure $\mathbf P^{\emptyset,\emptyset}_{\beta}$ in dimensions $d\geq 4$. As in Section \ref{sec:introwired}, logarithmic corrections emerge at the upper-critical dimension $d_c=4$.
	\begin{Thm}\label{thm:DRC} Let $d>4$. There exist $c,C>0$ such that, for every $n\geq 1$,
	\begin{equation}
		\frac{c}{n^{d-2}}\leq \mathbf P^{\emptyset,\emptyset}_{\beta_c}[0\connect{}\partial \Lambda_n]\leq \frac{C}{n^{d-2}}.
	\end{equation}
	\end{Thm}

	\begin{Thm}\label{thm:DRC2} Let $d=4$. There exist $c,C>0$ such that, for every $n\geq 1$,
	\begin{equation}
		\frac{c}{n^{2}(\log n)^4}\leq \mathbf P^{\emptyset,\emptyset}_{\beta_c}[0\connect{}\partial \Lambda_n]\leq \frac{C\log n}{n^{2}}.
	\end{equation}
	\end{Thm}
	
		These results may seem very surprising in view of the one-arm exponent computation performed for Bernoulli percolation and the FK-Ising model. However, they are consistent with the interpretation of the double random current percolation model as a \emph{loop} model. Indeed, in the context of the loop percolation model introduced by Le Jan and Lemaire \cite{le2013markovian}, it has been shown \cite{chang2016phase} that the one-arm exponent is of order $n^{2-d}$ in dimensions $d>4$ (see also the recent \cite{vogel2025} for an exact asymptotic estimate). In \cite{chang2016phase}, the authors also conjectured that the one-arm exponent is of order $(\log n)^{\hat{\rho}}/n^2$ for some $\hat{\rho}>0$ when $d=4$. We conjecture a similar behaviour here.
		
To illustrate how atypical the double random current model is compared with the percolation models introduced above, we make the following observation. Let us start by recalling a fundamental consequence of reflection positivity: the \emph{infrared bound}. It classically follows from \cite{FrohlichSimonSpencerIRBounds1976,MessagerMiracleSoleInequalityIsing} (see \cite[Section~3]{PanisTriviality2023} for a detailed proof): if $d\geq 3$, there exists $C>0$ such that, for every $\beta\leq \beta_c$, and every $x\in \mathbb Z^d\setminus \{0\}$,
	\begin{equation}\label{eq: irb intro drc}
		\langle \sigma_0\sigma_x\rangle_{\beta}\leq \frac{C}{|x|^{d-2}}.
	\end{equation}
	When restricting to dimensions $d>4$, the infinite volume version of \eqref{eq: connection 2pt current 2pt Ising} and \eqref{eq: irb intro drc} give that
		\begin{equation}\label{eq: finite susc at criticality}
		\mathbf E^{\emptyset,\emptyset}_{\beta_c}[|\mathbf C(0)|]=\sum_{x\in \mathbb Z^d}\mathbf P^{\emptyset,\emptyset}_{\beta_c}[0\connect{}x]=\sum_{x\in \mathbb Z^d}\langle \sigma_0\sigma_x\rangle_{\beta_c}^2<\infty,
	\end{equation}
	where $\mathbf E^{\emptyset,\emptyset}_\beta$ denotes the expectation with respect to $\mathbf P^{\emptyset,\emptyset}_\beta$, and where $\mathbf C(0)$ denotes the cluster of the origin. The expected size of the cluster of the origin is often called the \emph{susceptibility} of the model. In \eqref{eq: finite susc at criticality}, we see that the sourceless double random current measure has a finite susceptibility at its critical point. This contrasts with Bernoulli percolation or the FK-Ising model for which the susceptibility always blows up at criticality. Let us mention that, in dimensions $d\in \{2,3,4\}$, it is known \cite{Wu1966theory,DuminilPanis2024newLB} that $\mathbf E_{\beta_c}^{\emptyset,\emptyset}[|\mathbf C(0)|]=\infty$.
	
	The lower bounds in Theorems \ref{thm:DRC} and \ref{thm:DRC2} are derived using a second moment method. The upper bounds rely on a new way to explore an open path from $0$ to $\partial \Lambda_n$ in a sourceless double random current (see Lemma \ref{lem:one arm backbone}).

\subsection{Strategy of proof}\label{sec:introStrategy}

We now describe the strategy of proof for the results of Section \ref{sec:introFK}. In all our proof, the following key quantity---introduced in \cite{DuminilTassionNewProofSharpness2016}---plays a fundamental role: for every $S\subset \mathbb Z^d$, every $\beta\geq 0$, define
\begin{equation}\label{eq:def phi beta}
	\varphi_\beta(S):=\beta\sum_{\substack{u\in  S\\v\notin S\\ u\sim v}}\langle \sigma_0\sigma_u\rangle_{S,\beta}.
\end{equation}
In \cite{PanisTriviality2023}, a natural notion of correlation length associated to $\varphi_\beta(S)$ was introduced (see also \cite[Section~4]{hutchcroft2022derivation} for a closely related quantity). The \emph{sharp length} $L(\beta)$ is defined for $\beta\leq\beta_c$ by
\begin{equation}\label{eq:def l(beta)}    L(\beta):=\inf\left\lbrace k\geq 1: \:\exists S\subset \mathbb Z^d,\: 0\in S, \: \textup{diam}(S)\leq 2k, \: \varphi_{\beta}(S)< \tfrac{1}{10}\right\rbrace.
\end{equation}
A crucial input for our proofs is the following upper bound on $L(\beta)$, recently derived in \cite{DuminilPanis2024newLB}: there exists $C_{\ell}>0$ such that, for every $\beta\leq \beta_c$,
\begin{equation}\label{eq:cor.length_upperbound_intro}
	L(\beta) \leq C_{\ell} 
	\begin{cases}
		(\beta_c-\beta)^{-1/2}, & \text{for } d>4, \\
		(\beta_c-\beta)^{-1/2} |\log (\beta_c-\beta)|, & \text{for } d=4,
	\end{cases}
\end{equation}

\paragraph{One-arm for FK-Ising with wired boundary conditions.}

We describe here the main ideas in the proofs of Theorems~\ref{thm:FKwired meanfieldLB}, \ref{thm:FKwired d>4}, and \ref{thm:FKwired d=4}, concerning the one-arm probability of the wired FK-Ising model. Recalling \eqref{eq:ESC intro mag}, we disregard the FK-Ising measure $\phi_{\Lambda_n,\beta}^1$ and always work directly with the Ising measures of the form $\langle \cdot \rangle^\tau_{\Lambda,\beta,h}$, as defined in \eqref{eq:defIsing}.

\vspace{0.3cm}
\emph{The lower bound.} The proof of Theorem~\ref{thm:FKwired meanfieldLB} is very simple. The main input is a mean-field lower bound on $\langle \sigma_0\rangle_{\beta_c,h}$ ($h>0$), which is valid in all dimensions $d\geq2$. Precisely, it was proved in \cite{AizenmanBarskyFernandezSharpnessIsing1987} (see also \cite{DuminilTassionNewProofSharpness2016}) that for every $h\geq 0$ small enough,
\begin{equation}\label{eq:mean-field_mag_intro}
	\langle \sigma_0 \rangle_{\beta_c,h}\geq c_{\textup{mag}} h^{1/3},
\end{equation}
where $c_{\textup{mag}}>0$.
By differentiating on the magnetic field $h_x$ at each vertex $x\in\Lambda_n$ and using well-known correlation inequalities, we can upper bound the contribution to $\langle \sigma_0 \rangle_{\beta_c,h}$ of the magnetic field within $\Lambda_n$ by $Cn^2 h$.
Using this fact together with  \eqref{eq:mean-field_mag_intro} and taking $h=h_n:=(\varepsilon/n)^3$ for $\varepsilon>0$ small enough, gives
\begin{equation*}
	\langle \sigma_0 \rangle_{\beta_c,h\mathds{1}_{\Lambda_n^c}}\geq c_{\textup{mag}} h_n^{1/3} - Cn^2 h_n \asymp 1/n.
\end{equation*}
Since---by Markov's property and comparison of boundary conditions---one has $\langle \sigma_0 \rangle_{\Lambda_n,\beta_c}^+\geq \langle \sigma_0 \rangle_{\beta_c,h\mathds{1}_{\Lambda_n^c}}$, this implies the desired bound.

\vspace{0.3cm}
\emph{The upper bounds.} 
The proofs of the upper bounds of Theorems~\ref{thm:FKwired d>4} and \ref{thm:FKwired d=4} are much more involved. Our main input is again a bound on the infinite volume magnetisation with positive magnetic field which complements \eqref{eq:mean-field_mag_intro}. Precisely, it was proved in \cite{AizenmanFernandezCriticalBehaviorMagnetization1986} that
\begin{equation}\label{eq:mean-field_mag_upperbound_intro}
	\langle \sigma_0 \rangle_{\beta_c,h}\leq C_{\textup{mag}}
	\begin{cases}
		h^{1/3}, & \text{for } d>4, \\
		h^{1/3} |\log h|, & \text{for } d=4.
	\end{cases}
\end{equation}
Another crucial input in the proof is the upper bound on the sharp length $L(\beta)$ given above in \eqref{eq:cor.length_upperbound_intro}.
Using these two ingredients, we establish a \emph{renormalisation inequality}, i.e.~a relation between quantities in different scales, which ultimately implies the desired upper bounds by induction. 

We now explain our strategy in more details. 
We focus here on the case $d>4$, the case $d=4$ follows very similar lines, with parameters chosen slightly differently in order to accommodate the worse bounds \eqref{eq:cor.length_upperbound_intro} and \eqref{eq:mean-field_mag_upperbound_intro}.  We prove, by induction, that the following stronger inequality holds
\begin{equation}\label{eq:induction_intro}
	\langle \sigma_0 \rangle_{\Lambda_k,\beta_c,h_k}^+ \leq \frac{C_\star}{k}, \tag{$\star$}
\end{equation}
where $h_k:=(H/k)^3$ for some large enough constant $H>0$. 
Assuming that \eqref{eq:induction_intro} is true for $k= n/2$, we proceed to show that it also holds for $k=n$. 
Precisely, we prove the aforementioned renormalisation inequality:
\begin{align}\label{eq:renorm-wired_intro}
	\langle \sigma_0 \rangle_{\Lambda_n,\beta_c,h_n}^+ &\leq  
	\frac{1}{4} \langle \sigma_0 \rangle_{\Lambda_{n/2},\beta_c,h_{n/2}}^+ +  \langle \sigma_0 \rangle_{\beta_c,h_{n/2}},
\end{align}
which then implies \eqref{eq:induction_intro} at $k = n$ by the induction hypothesis and the mean-field bound on the magnetisation \eqref{eq:mean-field_mag_upperbound_intro}. 
In order to prove \eqref{eq:renorm-wired_intro}, we write
\begin{equation} 	\langle \sigma_0 \rangle_{\Lambda_n, \beta_c, h_n}^+ 
	= 
	\underbrace{ \langle \sigma_0 \rangle_{\Lambda_n, \beta_c, h_n}^+ -  \langle \sigma_0 \rangle_{\Lambda_{n}, J, h_{n/2}}^+ }_{(\mathrm{I})}
	+  
	\underbrace{ \langle \sigma_0 \rangle_{\Lambda_{n}, J, h_{n / 2}}^+ - \langle \sigma_0 \rangle_{\Lambda_{n}, J, h_{n/2}} }_{(\mathrm{II})} 
	+
	\underbrace{ \langle \sigma_0 \rangle_{\Lambda_{n}, J, h_{n/2}} }_{(\mathrm{III})},	
	\end{equation}
where $J=\beta_n \mathds{1}_{\Lambda_{n/2}} + \beta_c \mathds{1}_{\Lambda_{n/2}^c}$, with $\beta_n := \beta_c - (B/n)^2$, for some constant $B>0$. Clearly $(\mathrm{III})\leq \langle \sigma_0 \rangle_{\beta_c,h_{n/2}}$. By \eqref{eq:cor.length_upperbound_intro}, for $B$ sufficiently large we may find $S\subset \Lambda_{n/2}$ such that $\varphi_{\beta}(S)\leq \tfrac{1}{10}\leq \tfrac{1}{4}$ (by definition of $L(\beta)$). Relying on a \emph{Simon--Lieb type inequality} (see Lemma \ref{lem: random current proposition}), we then obtain
\begin{equation}\label{eq:proof strat wired II}
(\mathrm{II})\leq \varphi_{\beta}(S) \langle \sigma_0 \rangle_{\Lambda_{n/2}, \beta_c, h_{n / 2}}^+\leq \frac{1}{4}\langle \sigma_0 \rangle_{\Lambda_{n/2}, \beta_c, h_{n / 2}}^+.
\end{equation}

Finally, we show that $(\mathrm{I})\leq 0$, which is a more delicate task and relies itself on the induction hypothesis. Roughly speaking, the idea is that for $H$ large enough (depending on $B$, which has already been fixed), we can compensate the effect of increasing the inverse temperature from $\beta_n$ to $\beta_c$ by the effect of decreasing the magnetic field from $h_{n/2}$ to $h_n$. This is proved by comparing the derivatives in $J$ and $h$. 
In fact, it is easy to prove that the analogous quantity in infinite volume $(\mathrm{I}'):=\langle \sigma_0 \rangle_{\beta_c, h_n} -  \langle \sigma_0 \rangle_{\beta_n, h_{n/2}}$ is negative (for $H$ large enough). Indeed, one can consider the linear interpolation given by
$$\beta(t):=\beta_n + t, ~~~~~~h(t):=h_{n/2} - t\frac{h_{n/2}-h_n}{\beta_c-\beta_n},$$ 
and prove that $\partial_t \langle \sigma_0 \rangle_{\beta(t), h(t)}\leq 0$ for all $t\in[0,\beta_c-\beta_n]$. The latter follows readily from a differential inequality proved in \cite{AizenmanBarskyFernandezSharpnessIsing1987} (see also \cite{DuminilTassionNewProofSharpness2016}):
\begin{equation}\label{eq:diff-ineq_beta-vs-h_intro}
	\partial_\beta \langle \sigma_0 \rangle_{\beta, h} \leq 2d \langle \sigma_0 \rangle_{\beta, h} \cdot \partial_h \langle \sigma_0 \rangle_{\beta, h},
\end{equation}
together with the fact that $\langle \sigma_0 \rangle_{\beta, h}\leq C_{\textup{mag}} h^{1/3}$ for $\beta\leq\beta_c$ by \eqref{eq:mean-field_mag_upperbound_intro}. We follow a similar strategy in order to upper bound the finite volume derivative $\partial_t \langle \sigma_0 \rangle_{\Lambda_n,J(t), h(t)}^+$, where $J(t)= J + t\mathds{1}_{\Lambda_{n/2}}$. However, the analogue of \eqref{eq:diff-ineq_beta-vs-h_intro} in finite volume is more complicated. In particular, the right-hand side involves quantities of the form  $\langle \sigma_x \rangle_{\Lambda_n,J(t), h(t)}^+$ for $x\in\Lambda_n$, replacing the role of $\langle \sigma_0 \rangle_{\beta, h}$ (see \eqref{eq:mag-derived-t-wired-d>4 new proof}). In order to bound these quantities, instead of using \eqref{eq:mean-field_mag_upperbound_intro} as above, we observe that for $x\in\Lambda_{n/2}$ we have $\langle \sigma_x \rangle_{\Lambda_n,J(t), h(t)}^+ \leq \langle \sigma_0 \rangle_{\Lambda_{n/2},J(t), h(t)}^+ \leq \langle \sigma_0 \rangle_{\Lambda_{n/2},\beta_c, h_{n/2}}^+$ and use the induction hypothesis \eqref{eq:induction_intro}.

\paragraph{One-arm for FK-Ising with free boundary conditions.}

We now describe the main ideas in the proofs of Theorems \ref{thm:FKfree d>6} and \ref{thm:FKfree d=6}. 
In both cases, the most interesting part of the analysis is in the derivation of the upper bound (the lower bounds are less sensitive to $d\geq 6$ and are obtained using a second moment method which relies on both \cite{CamiaJiangNewman21,DuminilPanis2024newLB}). 

Again, the proof goes through the derivation of a renormalisation inequality on $u_n(\beta):=\phi_\beta[0\connect{}\partial \Lambda_n]$. Our input is very different from the wired case. The main idea is inspired from the entropic method introduced by Dewan and Muirhead \cite{DewanMuirhead} in the context of Gaussian percolation. Let us briefly describe it. We want to compare $u_{n}(\beta_c)$ to $u_{n/2}(\beta_c)$. As in \cite{DewanMuirhead}, if we allow to work at a slightly subcritical and ``well-tuned'' parameter $\beta_n$, then we can  obtain the ``contracting estimate'' 
\begin{equation}
u_{n}(\beta_n) \leq \frac 1 {10} u_{n/2}(\beta_c)
\end{equation}
The challenge in \cite{DewanMuirhead} is to tune $\beta_c-\beta_n$ 
as small as possible while still guaranteeing such a contraction bound. There, the authors rely on a ``classical'' notion of correlation length: the inverse of the exponential decay rate of the two-point function. However, this choice is not optimal and is responsible for the emergence of logarithmic divergences in their argument (see \cite{vEGPSperco} for an explanation). Here, we rely on the more geometric notion of correlation length given by the sharp length $L(\beta)$ defined in \eqref{eq:def l(beta)}.

We use \eqref{eq:cor.length_upperbound_intro} to obtain that---for $\beta_n:= \beta_c - \frac A {n^2}$ (with $A$ large enough)---there exists $S\subset \Lambda_{n/2}$ with $\varphi_{\beta_n}(S)\leq \tfrac{1}{10}$. As in \eqref{eq:proof strat wired II}, we rely on a Simon--Lieb type inequality (see Proposition \ref{prop:BK type for FK}) to conclude that $u_n(\beta_n) \leq \varphi_{\beta_n}(S) u_{n/2}(\beta_c)\leq \tfrac{1}{10}u_{n/2}(\beta_c)$. Let us mention that Proposition \ref{prop:BK type for FK} can also be seen as a (very limited) form of \emph{van den Berg--Kesten-type inequality} for the FK-Ising model. We believe this result is interesting on its own and may be relevant in the future to study the FK-Ising model.

Once such a contracting bound is obtained, we  need to estimate the difference $u_n(\beta_c)- u_n(\beta_n)$. The key idea of Dewan and Muirhead is to formulate a bound on this quantity in terms of the relative entropy between the law of the cluster of the origin at $\beta_c$ and at $\beta_n$. This is what they call the \emph{entropic bound}, see \cite[Proposition~1.21]{DewanMuirhead}. In the context of Bernoulli percolation, this bound turns out to be an easy consequence of Pinsker's inequality.
 
 In this paper, we derive the new analogue of their entropic bound for the FK-Ising model, see Theorem \ref{Thm: entropic bound pinsker}. However, in our case, an additional quantity emerges in the bound: the so-called {\em triangle diagram} defined by
\begin{equation} \Triangle(\beta_c) :=  \sum_{y, z \in \Z^d}  \langle \sigma_0 \sigma_y \rangle_{\beta_c} \langle \sigma_y \sigma_z \rangle_{\beta_c} \langle \sigma_z \sigma_0 \rangle_{\beta_c}. 
\end{equation}
It is a consequence of \cite{DuminilPanis2024newLB} (see Section \ref{sec:Prelim}) that this quantity is finite if and only if $d>6$. The presence of this infinite quantity in  $d<6$ makes our entropic bound inefficient in these dimensions, and provides a small justification of why the one-arm exponent changes around $d=6$.

When $d>6$, the combination of the steps above (i.e. the contraction inequality and the entropic bound) leads to the following renormalisation inequality: there exists $K\geq 1$ such that for every $n\geq 1$,
\begin{equation}
	u_n(\beta_c) \leq \frac 1 {10} u_{n/2}(\beta_c) + \frac K n \sqrt{u_n(\beta_c)}.
\end{equation}
This result is sufficient to conclude the desired upper bound of order $n^{-2}$. When $d=6$, we conclude similarly by approximating the law of the $\phi_{\beta_c}$ in $\Lambda_n$ with that of $\phi_{\Lambda_{n^\alpha},\beta_c}$ (with a suitable exponent $\alpha>1$), and by using the fact the corresponding truncated triangle diagram $\nabla_{\Lambda_{n^\alpha}}(\beta_c)$ only diverges logarithmically in $n$.

\begin{Rem}\label{r.perconote}
The use of the sharp length instead of the standard correlation length in the argument of \cite{DewanMuirhead} lead us to revisit the corresponding result for high-dimensional Bernoulli percolation in our companion paper \cite{vEGPSperco}. There, we provide a new short proof of the main result of \cite{KozmaNachmias2011OneArm} for sufficiently spread-out percolation in dimension $d>6$, which only relies on \cite{DumPan24Perco}.
Interestingly, we may also wish to revisit our entirely different proof of the wired-FK one arm exponent (Theorem \ref{thm:FKwired d>4}) in the setting of Bernoulli percolation. It turns out to give yet another interesting proof of \cite{KozmaNachmias2011OneArm}. As that approach is a bit longer than the entropic bound technique, we did not include it in \cite{vEGPSperco}. 
\end{Rem}

Let us now say a word about the proofs of our results in dimensions $d<6$. To obtain the lower bounds of Theorems \ref{thm: FKfree d=5}--\ref{thm: FKfree d=3}, a natural route is to use the Paley--Zygmund inequality and write
\begin{equation}
	\phi_{\beta_c}[0\connect{}\partial \Lambda_n]=\phi_{\beta_c}[\mathcal Z>0]\geq \frac{\phi_{\beta_c}[\mathcal Z]^2}{\phi_{\beta_c}[\mathcal Z^2]},
\end{equation}
where $\mathcal Z:=\sum_{z\in \Lambda_{2n}\setminus \Lambda_{n-1}}\mathds{1}_{0\connect{}z}$. To derive an upper bound on $\phi_{\beta_c}[\mathcal Z^2]$, we need to bound the three-point connectivity function $\phi_{\beta_c}[x,y,z \text{ lie in the same cluster}]$. This can be done using a variant of the \emph{tree-graph inequality} for the FK-Ising model (see 
Proposition \ref{prop:tree graph}). However, driven by the analogy with Bernoulli percolation, we only expect the tree-graph inequality to be
relevant (i.e. sharp) in dimensions $d\geq 6$. In the context of Bernoulli percolation in $d<6$, a (conjecturally sharp) replacement of this inequality was recently obtained by
  Gladkov \cite{gladkov2024percolation} (it played a key role in \cite{hutchcroft2025criticalII}). We did not manage to easily extend his inequality to the setting of the FK-Ising model. Nevertheless, we can derive a (conjecturally sharp in $d<6$) bound on the four-point connectivity function, see Proposition \ref{prop:four point}. It is well-suited to the following variation of the Paley--Zygmund inequality:
  \begin{equation}
  	\phi_{\beta_c}[\mathcal Z>0]\geq \frac{\phi_{\beta_c}[\mathcal Z]^{3/2}}{\phi_{\beta_c}[\mathcal Z^3]^{1/2}}.
  \end{equation}
  
  Finally, in Theorem \ref{thm: FKfree d=4}, the upper bound follows from a combination of the inequality $\phi_{\beta_c}[0\connect{}\partial \Lambda_n]\leq \phi_{\Lambda_n,\beta_c}^1[0\connect{}\partial \Lambda_n]$ (which is a consequence of Markov's property and monotonicity in boundary conditions) and Theorem \ref{thm:FKwired d=4}. 
%

	\paragraph{Organisation of the paper.} In Section \ref{sec:Prelim}, we recall useful properties of the Ising model. In Section \ref{sec:FKwired}, we prove the results presented in Section \ref{sec:introFK} regarding the wired FK-Ising measure. In Section \ref{sec:FKfree}, we prove the results on the free/infinite volume FK-Ising measure. In Section \ref{sec:RCR}, we introduce the random current representation, and provide a proof to some bounds used in Sections \ref{sec:FKwired} and \ref{sec:FKfree}. Finally, in Section \ref{sec:DRC}, we prove Theorems \ref{thm:DRC} and \ref{thm:DRC2} regarding the sourceless double random current measure.
	
	\paragraph{Acknowledgements.} We thank Gady Kozma for suggesting us to look at \cite{DewanMuirhead}, and Lorca Heeney-Brockett for noticing that our computation of the one-arm exponent for the double random current measure could be extended to $d=4$. Additionally, we thank Roland Bauerschmidt, Hugo Duminil-Copin, Nikita Gladkov, Gr\'egory Miermont, Stephen Muirhead, Akira Sakai, and Zijie Zhuang for useful discussions; and Aman Markar and Florian Schweiger for comments on a preliminary version of the paper.  DvE, CG and FS acknowledge support from the ERC grant VORTEX 101043450. RP acknowledges support from the SNSF through a Postdoc.Mobility grant.

\section{Preliminaries}\label{sec:Prelim}
In this section, we collect several results and definitions that will be used throughout the paper. We first recall some classical correlation inequalities. Then, we state new inequalities for the FK-Ising model, which we prove in Section \ref{sec:RCR} using the random current representation. Finally, we describe key properties of the high-dimensional Ising model. 
\subsection{Correlation inequalities}

We do not seek the widest generality and refer to \cite{FriedliVelenikIntroStatMech2017,DuminilLecturesOnIsingandPottsModels2019} and references therein for more information. Before moving to the statements, we quickly revisit the definition of the Ising model from Section \ref{sec:intro}.

\begin{Rem}\label{rem:more general ising models} We make the following useful observations regarding \eqref{eq:defIsing}. 
\begin{enumerate}
\item[$(i)$] In \eqref{eq:defIsing}, the boundary condition $\tau$ can always be \emph{absorbed} in the external magnetic field $\mathsf{h}$ in the following sense: if for $x\in V$ we let $\mathsf{h}^\tau_x:=\sum_{\substack{y\sim x}}J_{xy}\tau_y\mathds{1}_{y\notin V(G)}$, then $\langle \cdot\rangle^\tau_{G,J,\mathsf{h}}=\langle \cdot\rangle_{G,J,\mathsf{h}+\mathsf{h}^\tau}$. As a consequence, statements involving the general measure $\langle \cdot \rangle_{G,J,\mathsf{h}}$ for $\mathsf{h}\in (\mathbb R^+)^V$ also apply (for example) to $\langle \cdot \rangle_{G,J}^+$.
\end{enumerate}
We will use the following generalisation of \eqref{eq:defIsing}.
\begin{enumerate}
\item[$(ii)$] First, we may define $\langle \cdot\rangle_{G,\beta,\mathsf{h}}$ on \emph{any} finite graph $G$ (and thus on any graph $G$ by standard approximation arguments). An interesting example for us will be to choose $G$ to be the multigraph such that $V(G)=\mathbb Z^d$ and $E(G)$ consists of two copies of $E(\mathbb Z^d)$ (i.e.\ each edge of $\mathbb Z^d$ is duplicated).
\end{enumerate}
\end{Rem}
For a set $A\subset \mathbb Z^d$, we let $\sigma_A:=\prod_{x\in A}\sigma_x$. We begin with one of the most fundamental correlation inequality in the study of the Ising model. 

\begin{Prop}[Griffiths' inequality {\cite{GriffithsCorrelationsIsing1-1967,kelly1968general}}]\label{prop:secondineqGriffiths} Let $d\geq 2$ and $G=(V,E)$ be a subgraph of $\mathbb Z^d$. Then, for every $J\in (\mathbb R^+)^{E}$, every $\mathsf{h}\in (\mathbb R^+)^{V}$, and every $A,B\subset V$,
\begin{equation}
	\langle \sigma_A;\sigma_B\rangle_{G,J,\mathsf{h}}:=\langle \sigma_A\sigma_B\rangle_{G,J,\mathsf{h}}-\langle \sigma_A\rangle_{G,J,\mathsf{h}}\langle \sigma_B\rangle_{G,J,\mathsf{h}}
\geq 0.
\end{equation}
\end{Prop}

The next result is a classical consequence of Proposition \ref{prop:secondineqGriffiths}. Its proof can be found in \cite[Chapter~3]{FriedliVelenikIntroStatMech2017}.

\begin{Prop}\label{prop:consequencegriffiths} Let $d\geq 2$ and $G=(V,E)$, $H=(V',E')$ be two subgraphs of $\mathbb Z^d$ such that $G\subset H$. Then, for every $A\subset V$, every $J\in (\mathbb R^+)^E$, $J'\in (\mathbb R^+)^{E'}$ satisfying $J_e\leq J_e'$ for $e\in E$, and every $\mathsf{h},\mathsf{h}'\in (\mathbb R^+)^{V(H)}$ satisfying $\mathsf{h}_x\leq \mathsf{h}'_x$ for $x\in V$, one has
\begin{equation}\label{eq:monot Griffiths}
	\langle \sigma_A\rangle_{G,J,\mathsf{h}}\leq \langle \sigma_A\rangle_{H,J',\mathsf{h}'}.
\end{equation}
Moreover, if we further assume that the restriction of $J'$ to $E$ is equal to $J$, then for every $v\in V$, one has
\begin{equation}\label{eq:monot Mag}
	\langle\sigma_v\rangle^+_{H,J'}\leq \langle \sigma_v\rangle_{G,J}^+.
\end{equation}
\end{Prop}

Before moving to the next propositions, we recall the classical expressions of the derivatives of $\langle \sigma_A\rangle_{G,J,\mathsf{h}}$ with respect to the parameters $J$ and $\mathsf{h}$.

\begin{Lem}\label{lem:derivatives} Let $d\geq 2$ and $G=(V,E)$ be a subgraph of $\mathbb Z^d$. Let $J\in (\mathbb R^+)^E$, $\mathsf{h}\in \mathbb R^V$, and $A\subset V$. Then, for every $x\in V$,
\begin{equation}
	\partial_{\mathsf{h}_x}\langle \sigma_A\rangle_{G,J,\mathsf{h}}=\langle \sigma_A;\sigma_x\rangle_{G,J,\mathsf{h}},
\end{equation}
and for every $uv\in E$,
\begin{equation}
	\partial_{J_{uv}}\langle \sigma_A\rangle_{G,J,\mathsf{h}}=\langle \sigma_A;\sigma_u\sigma_v\rangle_{G,J,\mathsf{h}}.
\end{equation}
\end{Lem}

The following proposition initially appeared in \cite[Proposition~4.7]{AizenmanFernandezCriticalBehaviorMagnetization1986} and can be derived using the random current representation introduced in Section \ref{sec:RCR}. A stronger version of it--- allowing $\mathsf{h}$ to take negative values---was recently derived in \cite{Ding2023new}.

\begin{Prop}\label{pr.Ding} Let $d\geq 2$ and $G$ be a subgraph of $\mathbb Z^d$. Then, for every $\beta\geq 0$, every $\mathsf{h}\in (\mathbb R^+)^{V(G)}$, and every $x,y\in V(G)$, one has
\begin{equation}
	\langle \sigma_x;\sigma_y\rangle_{G,\beta,\mathsf{h}}\leq \langle \sigma_x\sigma_y\rangle_{G,\beta}.
\end{equation}
\end{Prop}
We now turn to a more intricate and less classical correlation inequality. The way we state it may seem slightly unnatural but it is well suited for later applications. If $J=(J_{xy})_{xy \in E(G)}$ and $x,y,z,w,u,v\in V(G)$, write
\begin{multline}
\langle \sigma_x\sigma_y ; \sigma_z \sigma_w ; \sigma_{u}\sigma_v \rangle_{G,J}
		= \langle \sigma_x \sigma_y \sigma_z \sigma_w ; \sigma_u\sigma_v \rangle_{G,J} - \langle \sigma_x\sigma_y\rangle_{G,J} \langle \sigma_z \sigma_w; \sigma_u\sigma_v \rangle_{G,J} \\- \langle \sigma_z \sigma_w \rangle_{G,J}\langle \sigma_x \sigma_y ; \sigma_u \sigma_v \rangle_{G,J}
\end{multline}
\begin{Rem}\label{Rem:second derivative} If $zw\in E(G)$ and $uv\in E(G)$, we have
\begin{equation}
	\partial_{J_{uv}}\partial_{J_{zw}}\langle \sigma_x\sigma_y\rangle_{G,J}=\langle \sigma_x\sigma_y ; \sigma_z \sigma_w ; \sigma_{u}\sigma_v \rangle_{G,J}.
\end{equation}
\end{Rem}
\begin{Lem}\label{lem:double truncated} Let $d\geq 2$. There exists $C>0$ such that the following holds. For every $G$ subgraph of $\mathbb Z^d$, every $\beta\in [\beta_c/2,\beta_c]$, every $J\in (\mathbb R^+)^{E(G)}$ which satisfies $J_{st}\leq \beta$ for every $st\in E(G)$, every $x,y\in V(G)$, and every $zw,uv\in E(G)$,
\begin{equation}
	\langle \sigma_x\sigma_y ; \sigma_z \sigma_w ; \sigma_{u}\sigma_v \rangle_{G,J}\leq C\Big(\langle \sigma_x\sigma_z\rangle_{G,\beta} \langle \sigma_w\sigma_u\rangle_{G,\beta}\langle \sigma_v\sigma_y\rangle_{G,\beta}+\langle \sigma_y\sigma_z\rangle_{G,\beta} \langle \sigma_w\sigma_u\rangle_{G,\beta}\langle \sigma_v\sigma_x\rangle_{G,\beta}\Big).
\end{equation}
\end{Lem}
\begin{proof} We fix $\beta,G,J$ as in the statement of the proposition. To formulate what is below, it is very convenient to relabel $z=z_1$, $w=z_2$, $u=u_1$, and $v=u_2$. We also let $f(1)=2$ and $f(2)=1$. Thanks to \cite[equation B.34]{LiuPanisSladePlateau}, one has 
\begin{align}
	\langle \sigma_x\sigma_y ; \sigma_{z_1} \sigma_{z_2} ; \sigma_{u_1}\sigma_{u_2} \rangle_{G,J}&\leq \sum_{i,j=1}^2 2\langle\sigma_x\sigma_{z_i}\rangle_{G,J}\langle \sigma_{z_{f(i)}}\sigma_{u_j}\rangle_{G,J}\langle\sigma_{u_{f(j)}}\sigma_y\rangle_{G,J}+(x\Leftrightarrow y) 
	\notag\\&\leq \sum_{i,j=1}^2 2\langle\sigma_x\sigma_{z_i}\rangle_{G,\beta}\langle \sigma_{z_{f(i)}}\sigma_{u_j}\rangle_{G,\beta}\langle\sigma_{u_{f(j)}}\sigma_y\rangle_{G,\beta}+(x\Leftrightarrow y),
	\end{align}
	where $(x\Leftrightarrow y)$ denotes the same sum with the roles of $x$ and $y$ interchanged, and where we used \eqref{eq:monot Griffiths} in the second inequality. The proof follows from Lemma \ref{lem:trivial comparison neighbors}.
\end{proof}

\begin{Lem}\label{lem:trivial comparison neighbors} Let $d \geq 2$.  
	There exists $C>0$ such that the following holds. For every $G$ subgraph of $\mathbb Z^d$, every $\beta\geq \beta_c/2$, every $x\in V(G)$ and $zw\in E(G)$, 
\begin{equation}
	\frac{1}{C}\langle \sigma_x \sigma_w \rangle_{G, \beta} \leq \langle \sigma_x \sigma_z \rangle_{G, \beta} \leq  C \langle \sigma_x \sigma_w \rangle_{G, \beta}.
\end{equation}
Similarly, for any $\mathsf{h}\in (\mathbb{R}^+)^V$ and $zw \in E(G)$,
	\begin{equation}\label{eq:comp mag}
		\langle \sigma_z \rangle_{G, \beta, \mathsf{h}} \leq C \langle \sigma_w \rangle_{G, \beta, \mathsf{h}}
	\end{equation}
\end{Lem}

\begin{proof}
	By Proposition \ref{prop:secondineqGriffiths}, 
\begin{equation}
	\langle \sigma_x \sigma_z \rangle_{G,\beta} \geq \langle \sigma_x \sigma_w \rangle_{G, \beta} \langle \sigma_z \sigma_w \rangle_{G, \beta}. 
\end{equation}
Write $\langle \cdot \rangle_{zw, \beta}$ for the measure on the single edge graph $G=(\{z,w\},\{zw\})$. 
	By Proposition \ref{prop:consequencegriffiths}, 
\begin{equation}
	\langle \sigma_z \sigma_w\rangle_{G, \beta} \geq \langle \sigma_z\sigma_w \rangle_{zw, \beta}=\tanh(\beta) \geq \tanh(\beta_c/2),
\end{equation}
where the last inequality follows from the assumption on $\beta$. Observing that $z$ and $w$ play symmetric roles yields
\begin{equation}
	\langle \sigma_x\sigma_w\rangle_{G,\beta}\geq \langle \sigma_x\sigma_z\rangle_{G,\beta}\tanh(\beta_c/2).
\end{equation} 
This concludes the proof of the first equation by setting $C:=\tanh(\beta_c/2)^{-1}$.

The second part of the statement is very similar to the first one. Using Proposition \ref{prop:secondineqGriffiths} again,
\begin{equation}
	\langle \sigma_w\rangle_{G,\beta,\mathsf{h}}
	\geq \langle \sigma_z\sigma_w\rangle_{G,\beta,\mathsf{h}}
	\langle \sigma_z\rangle_{G,\beta,\mathsf{h}}\geq \langle \sigma_z\sigma_w\rangle_{zw,\beta}\langle \sigma_z\rangle_{G,\beta,\mathsf{h}}\geq \frac{1}{C}\langle \sigma_z\rangle_{G,\beta,\mathsf{h}},
\end{equation}
here in the second inequality we used \eqref{eq:monot Griffiths}. This concludes the proof.
\end{proof}

We conclude this first subsection with two additional correlation inequalities. Their proofs rely on the random current representation and are postponed to Section \ref{sec:RCRcorineq}. The first result already appeared in \cite{AizenmanFernandezCriticalBehaviorMagnetization1986}, and is a rephrasing of the Simon--Lieb inequality \cite{SimonInequalityIsing1980,LiebImprovementSimonInequality}.
\begin{Lem} \label{Lem: bubble bound with C} Let $d\geq 2$ and $G=(V,E)\subset H=(V',E')$ be two subgraphs of $\mathbb Z^d$. For every $\beta\geq 0$, and every $x,y \in V'$,
\begin{equation}
	\langle \sigma_x\sigma_y\rangle_{H,\beta}-\langle \sigma_x\sigma_y\rangle_{G,\beta}\leq \sum_{u\in \partial V}\langle \sigma_x\sigma_u\rangle_{G,\beta}\langle \sigma_u\sigma_y\rangle_{H,\beta},
\end{equation}
where $\p V:= \{ u\in V : \exists v\in V' \setminus V \text{ with } uv \in E' \}$.
\end{Lem}
To state the second result, recall from \eqref{eq:def phi beta} that, for every $S\subset \mathbb Z^d$ finite containing $0$, and every $\beta\geq 0$, 
\begin{equation}
	\varphi_\beta(S)=\beta\sum_{\substack{u\in  S\\v\notin S\\ u\sim v}}\langle \sigma_0\sigma_u\rangle_{S,\beta}.
\end{equation}
Additionally, define the external vertex boundary of $S$ by $\partial^{\textup{ext}} S:=\{y\in \mathbb Z^d\setminus S:\: \exists x\in S, \: x\sim y\}$.
The following result is useful to measure the effect of boundary conditions on the magnetisation. It is very similar to Lemma \ref{Lem: bubble bound with C}.
\begin{Lem}\label{lem: random current proposition} Let $d\geq 2$, $\beta>0$, $h\geq 0$, and $n\geq 1$. Let $J\in (\mathbb R^+)^{E(\mathbb Z^d)}$ which satisfies $J_e=\beta$, for every $e\in E(\Lambda_n)$. Let $\Lambda$ be a finite subset of $\mathbb Z^d$ containing $\Lambda_{2n}$. Then, for every $S\subset \Lambda_n$, 
\begin{equation}
	\langle \sigma_0\rangle^+_{\Lambda,J,h}-\langle \sigma_0\rangle_{\Lambda,J,h}\leq \varphi_\beta(S)\cdot \max_{x\in \partial^{\textup{ext}} S}\langle \sigma_x\rangle^+_{\Lambda,J,h}.
\end{equation}
\end{Lem}

\subsection{Properties of the FK-Ising model}

We begin by recalling two very classical properties of the FK-Ising model and refer to \cite{Grimmett2006RCM,DuminilLecturesOnIsingandPottsModels2019} for proofs.
\begin{Prop}[Domain Markov property]\label{prop:DMP} Let $d\geq 2$ and $\beta\geq 0$. Let $G=(V,E),H=(V',E')$ be subgraphs of $\mathbb Z^d$ with $G\subset H$. Let $\xi\in \{0,1\}^{E'\setminus E}$. Then, for every event $\mathcal A\subset \{0,1\}^E$, one has
 \begin{equation}
 	\phi_{H,\beta}^0[\mathcal A \: |\: \omega_{|E'\setminus E}=\xi]=\phi_{G,\beta}^\xi[\mathcal A].
 \end{equation}
\end{Prop}
Let $E\subset E(\mathbb Z^d)$. A set $\mathcal A\subset \{0,1\}^E$ is said to be \emph{increasing} if the following holds: for every $\omega,\omega'\in \{0,1\}^E$, if $\omega\leq \omega'$ (in the sense that $\omega_e\leq \omega_e'$ for every $e\in E$) and $\omega\in \mathcal A$, then $\omega'\in \mathcal A$.
\begin{Prop}[Stochastic domination]\label{prop:stoch dom FK} Let $d\geq 2$ and $0\leq \beta'\leq \beta$. Let $G=(V,E)$ be a subgraph of $\mathbb Z^d$. Then, for every $\xi\in \{0,1\}^{E(\mathbb Z^d)\setminus E}$, and for every increasing event $\mathcal A\subset \{0,1\}^E$, the following inequalities hold:
\begin{equation}
	\phi_{G,\beta'}^{\xi}[\mathcal A]\leq \phi^1_{G,\beta}[\mathcal A],
\end{equation}
and
\begin{equation}
	\phi_{G,\beta'}^{0}[\mathcal A]\leq \phi_{\beta}[\mathcal A].
\end{equation}
\end{Prop}

We now present two new inequalities for the FK-Ising model. Their proofs also rely on the random current representation and are postponed to Section \ref{sec:RCRFK}. These inequalities are a well-established part of the mathematical study of Bernoulli percolation, where they follow from either independence or the van den Berg--Kesten (BK) inequality (see \cite{GrimmettPercolation1999}). These two properties are not available at the level of the FK-Ising model. Nevertheless, building on a coupling of the FK-Ising model with random currents (see Section \ref{sec:RCRFK}), we manage to obtain the following useful results. The first inequality can be seen as a generalisation of the Simon--Lieb inequality \cite{SimonInequalityIsing1980,LiebImprovementSimonInequality} (which corresponds to $|B|=1$ below).

\begin{Prop}\label{prop:BK type for FK} Let $d\geq 2$ and $\beta>0$. Let $G=(V,E)$ be a subgraph of $\mathbb Z^d$ with $0\in V$. Let $S\cap B=\emptyset$ be finite subsets of $V$, with $S$ containing $0$. Assume that $\partial S$ is separating $0$ and $B$ in the sense that every path from $0$ to $B$ has to visit $\partial S$. Then,
\begin{equation}
	\phi_{G, \beta}^0[0 \connect{} B] \leq \sum_{\substack{u\in S\\ v\notin S\\ u\sim v}} \phi_{S, \beta}^0[0 \connect{} u]\beta\phi_{G, \beta}^0[v \connect{} B]. 
\end{equation}
\end{Prop}

The tree graph inequality was initially obtained by Aizenman and Newman \cite{AizenmanNewmanTreeGraphInequalities1984} in the context of Bernoulli percolation. Roughly put, it states that higher order correlation functions of percolation can be upper bounded by a ``tree-like'' structure involving the two-point function. This inequality plays a key role to get the lower bound on the one-arm probability for Bernoulli percolation in $d>6$ \cite{sakai2004mean,KozmaNachmias2011OneArm}. It is expected to be sharp (up to a multiplicative constant) in dimensions $d>6$ (see conjecture below \cite[Proposition~4.1]{AizenmanNewmanTreeGraphInequalities1984}). We extend it to the FK-Ising model. To the best of our knowledge, this non-trivial extension is new. We let $\{x\connect{}y,z\}$ denote the event that $x$ is connected to both $y$ and $z$.
\begin{Prop}[Tree-graph inequality]\label{prop:tree graph} Let $G=(V,E)$ be a subgraph of $\mathbb Z^d$ and $\beta\geq 0$. Then, for every $x,y,z\in V$,
\begin{equation}
	\phi_{G,\beta}^0[x\connect{}y,z]\leq \sum_{u,u'\in V: \: u\sim u'}\phi_{G,\beta}^0[x\connect{}u']\phi_{G,\beta}^0[u'\connect{} z]\phi^0_{G,\beta}[u\connect{}y].
\end{equation}
\end{Prop}

\subsection{Results in high dimensions} 
We now turn to results regarding the high-dimensional Ising model.
We start by collecting useful bounds on the two-point function of the model at criticality. 

One of the key properties of the Ising model is \emph{reflection positivity} \cite{FrohlichSimonSpencerIRBounds1976,FrohlichIsraelLiebSimon1978}. 
As a consequence of it, 
one may derive two fundamental results: 
the infrared bound \cite{FrohlichSimonSpencerIRBounds1976}, and the Messager--Miracle-Sol\'e inequalities \cite{MessagerMiracleSoleInequalityIsing}. 
When combined, these tools provide the following estimate on the full-space two-point function at criticality: 
if $d\geq 3$, there exists $C>0$ such that, for every $x\in \mathbb Z^d$,
\begin{equation}\label{eq:IRB}
	\langle \sigma_0\sigma_x\rangle_{\beta_c}\leq \frac{C}{(1\vee|x|)^{d-2}}.
\end{equation}
By \eqref{eq:monot Griffiths}, this bound holds in the entire regime $\beta\leq \beta_c$. 
The upper bound in \eqref{eq:IRB} is expected to be of the good order when $d\geq 4$ (see for instance \cite{SladeTombergRGWeakPhi4in2016}).
A matching order near-critical lower bound was derived in dimensions $d>4$ in \cite{DuminilPanis2024newLB}: there exists $c>0$ such that, for every $\beta\leq \beta_c$, and every $|x|\leq c(\beta_c-\beta)^{-1/2}$, one has
\begin{equation}\label{eq: lower bound full-space sec 2}
	\langle \sigma_0\sigma_x\rangle_{\beta}\geq \frac{c}{(1\vee |x|)^{d-2}}.
\end{equation}
In $d=4$, the result obtained in \cite{DuminilPanis2024newLB} is a bit weaker: there exists $c>0$ such that, for every $\beta\leq \beta_c$, and every $|x|\leq c(\beta_c-\beta)^{-1/2}|\log(\beta_c-\beta)|^{-1}$,
\begin{equation}\label{eq: lb d=4 sec 2}
	\langle \sigma_0\sigma_x\rangle_{\beta}\geq \frac{c}{(1\vee|x|^{2}\log |x|)}.
\end{equation}
In our renormalisations, we will need to introduce a proper notion of correlation length. As mentioned in Section \ref{sec:introStrategy}, the most convenient one for us is the so-called sharp length, which was analysed in \cite{DuminilPanis2024newLB} (see also \cite{DumPan24Perco,DumPan24WSAW} for a study in other contexts of statistical mechanics). Recall the definition of $\varphi_{\beta}(S)$ from \eqref{eq:def phi beta}. The sharp length $L(\beta)$ is defined by 
\begin{equation}
    L(\beta)=\inf\left\lbrace k\geq 1: \:\exists S\subset \mathbb Z^d,\: 0\in S, \: \textup{diam}(S)\leq 2k, \: \varphi_{\beta}(S)< \tfrac{1}{10}\right\rbrace.
\end{equation}
We now import a fundamental result of \cite{DuminilPanis2024newLB} (see Remark 1.7): if $d>4$, there exist $c_0,C_0>0$ such that, for every $\beta\leq \beta_c$,
\begin{equation}\label{eq:proof d>4 sharp length estimate}
	c_0(\beta_c-\beta)^{-1/2}\leq L(\beta)\leq C_0(\beta_c-\beta)^{-1/2}.
\end{equation}
(In fact in \cite{DuminilPanis2024newLB}, $L(\beta)$ was defined with $\tfrac{1}{2}$ instead of $\tfrac{1}{10}$ but this replacement does not affect the proof of \eqref{eq:proof d>4 sharp length estimate}). We state a straightforward consequence of \eqref{eq:proof d>4 sharp length estimate} which plays a central role in our renormalisation schemes.
\begin{Lem}\label{lem:existence of S which drops d>4} Let $d>4$. There exists $C_1>0$ such that the following holds. For every $\beta<\beta_c$, and every $n\geq C_1(\beta_c-\beta)^{-1/2}$, there exists $S\subset \Lambda_{n/2}$ containing $0$ such that
\begin{equation}
	\varphi_{\beta}(S)\leq \frac{1}{10}.
\end{equation}
\end{Lem}
\begin{proof}Set $C_1:=4C_0+1$ and let $n\geq C_1(\beta_c-\beta)^{-1/2}$. If $\beta<\beta_c$, \eqref{eq:proof d>4 sharp length estimate} gives that $L(\beta)\leq C_0(\beta_c-\beta)^{-1/2}\leq \tfrac{n}{4}$. In particular, it means that there is $S\subset \Lambda_{2L(\beta)}\subset \Lambda_{n/2}$ containing $0$ such that $\varphi_\beta(S)\leq \tfrac{1}{10}$. This concludes the proof. \end{proof}
In dimensions $d=4$, \cite{DuminilPanis2024newLB} proves the following weaker estimate: for every $\beta<\beta_c$, one has $L(\beta)=(\beta_c-\beta)^{-1/2+o(1)}$ where $o(1)$ tends to $0$ as $\beta$ tends to $\beta_c$.
 In fact, one can show\footnote{Indeed, going through the proof of \cite[Theorem~1.6]{DuminilPanis2024newLB}, we find that for every $\beta<\beta_c$, 
	\begin{equation*}
		\frac{L(\beta)^2}{\log L(\beta)}\lesssim \chi(\beta)\lesssim L(\beta)^2.
	\end{equation*}
	Combined with the main result of \cite{AizenmanGrahamRenormalizedCouplingSusceptibility4d1983}, this gives \eqref{eq:proof d=4 sharp length estimate}.
} that there exist $c_0,C_0>0$ such that, for every $\beta<\beta_c$,
\begin{equation}\label{eq:proof d=4 sharp length estimate}
	c_0(\beta_c-\beta)^{-1/2}\leq L(\beta)\leq C_0|\log(\beta_c-\beta)|(\beta_c-\beta)^{-1/2}.
\end{equation}
Thanks to \eqref{eq:proof d=4 sharp length estimate}, we obtain the following analogue of Lemma \ref{lem:existence of S which drops d>4}.
\begin{Lem}\label{lem:existence of S which drops d=4}Let $d=4$. There exists $C_1>0$ such that the following holds. For every $\beta<\beta_c$, and every $n\geq C_1(\beta_c-\beta)^{-1/2}|\log (\beta_c-\beta)|$, there exists $S\subset \Lambda_{n/2}$ containing $0$ such that
\begin{equation}
	\varphi_{\beta}(S)\leq \frac{1}{10}.
\end{equation}
\end{Lem}

We collect a final result which will be useful in Section \ref{sec:DRC} to prove Theorems \ref{thm:DRC} and \ref{thm:DRC2}. Its proof is more technical and relies on various computations from \cite{DuminilPanis2024newLB}. It is given at the end of Section \ref{sec:RCR}. 
\begin{Lem}\label{lem:existence of S for drc} Let $d\geq 4$. There exists $C>0$ such that the following holds. For every $n\geq 1$, there exists $S\subset \Lambda_n$ containing $0$ which satisfies the following properties:
\begin{enumerate}
	\item[$(i)$] $\varphi_{\beta_c}(S)\leq C(\log n)^{\mathds{1}_{d=4}}$;
	\item[$(ii)$] one has
	\begin{equation}
		\beta_c \sum_{\substack{x\in S\cap \Lambda_{n/2}\\y\notin S\\x\sim y}}\langle \sigma_0\sigma_x\rangle_{S,\beta_c}\langle\sigma_0\sigma_y\rangle_{\beta_c}\leq C\frac{(\log n)^{\mathds{1}_{d=4}}}{n^{d-2}}.
	\end{equation} 
\end{enumerate}
\end{Lem}

\begin{Rem}\label{rem:about phi(S)} It would be convenient to choose $S=\Lambda_n$ in the above lemma. When $d>4$, one would have to prove the existence of $C>0$ such that, for every $n\geq 1$, $\varphi_{\beta_c}(\Lambda_n)\leq C$. This property is not known, but we conjecture it holds. In the context of high-dimensional Bernoulli percolation, the uniform boundedness of $\varphi_{\beta_c}(\Lambda_n)$ was proved in \cite[Theorem~1.5(a)]{vdHofstadSapoS14} (in the setting where the lace expansion applies) relying on the regularity of clusters developed by Kozma and Nachmias \cite{KozmaNachmias2011OneArm}. In other perturbative settings, the same conclusion holds for the weakly self-avoiding walk model in $d>4$ \cite{DumPan24WSAW} and sufficiently spread-out percolation in $d>6$ \cite{DumPan24Perco}. When $d=4$, at the (expected) upper-critical dimension of the Ising model, we conjecture the existence of $C>0$ such that, for every $n\geq 1$, $\varphi_{\beta_c}(\Lambda_n)\lesssim (\log n)^C$.
\end{Rem}

%
%
%

Finally, we state a last classical result regarding the Ising model in dimensions $d>4$. The susceptibility $\chi(\beta)$ is defined, for $\beta<\beta_c$, by
\begin{equation}
	\chi(\beta)=\sum_{x\in \mathbb Z^d}\langle \sigma_0\sigma_x\rangle_{\beta_c}.
\end{equation}
The following result first appeared in \cite{AizenmanGeometricAnalysis1982} (see also \cite{AizenmanGrahamRenormalizedCouplingSusceptibility4d1983}), a stronger result---capturing the exact asymptotic as $\beta$ approaches $\beta_c$---recently appeared in \cite{PanisIIC2024}.
\begin{Prop}\label{prop:susceptibility} Let $d>4$. There exist $c,C>0$ such that, for every $\beta<\beta_c$,
\begin{equation}
	\frac{c}{\beta_c-\beta}\leq \chi(\beta)\leq \frac{C}{\beta_c-\beta}.
\end{equation}
\end{Prop}

\subsection{Mean-field bounds on the full-space magnetisation}
We now describe the main input for the proofs of Theorems \ref{thm:FKwired meanfieldLB}, \ref{thm:FKwired d>4}, and \ref{thm:FKwired d=4}. Below, we let
\begin{equation}
	m(\beta,h):=\langle \sigma_0\rangle_{\beta,h}
\end{equation}
denote the full-space magnetisation at parameters $\beta,h\geq 0$. The following result provides a lower bound on $m(\beta_c,h)$ that is valid in every dimensions. This bound was originally derived in \cite{AizenmanBarskyFernandezSharpnessIsing1987}, and subsequently reproved through an alternative approach in \cite{DuminilTassionNewProofSharpness2016}.

\begin{Thm}\label{thm:MF lower bound} Let $d\geq 2$. There exists $c_{\textup{mag}}=c_{\textup{mag}}(d)> 0$ such that, for every $h\in [0,1]$,
		\begin{equation}
			m(\beta_c,h)\geq c_{\textup{mag}}h^{1/3}.
		\end{equation}
\end{Thm}
The above bound is often called a ``mean-field lower bound'' as it is expected to be sharp (up to a logarithmic correction in $d=4$) in the mean-field regime. This prediction was confirmed with the following complementary upper bound obtained in \cite{AizenmanFernandezCriticalBehaviorMagnetization1986}.
\begin{Thm}\label{thm:MF upper bound} Let $d\geq 4$. There exists $C_{\textup{mag}}=C_{\textup{mag}}(d)\geq 1$ such that, for every $h\geq0$,
		\begin{equation}
			m(\beta_c,h)\leq C_{\textup{mag}}\left\{
			\begin{array}{ll}
				h^{1/3} & \mbox{if } d>4, \\
				h^{1/3}|\log h| & \mbox{if }d=4.
			\end{array}
			\right.
		\end{equation}
	\end{Thm}

\section{One-arm exponent of the wired FK-Ising measure}\label{sec:FKwired}

Recall that by the Edwards--Sokal coupling \eqref{eq: consequences of ESC} we have
$$\phi_{\Lambda_n,\beta_c}^1[0\connect{}\partial \Lambda_n]=\langle \sigma_0 \rangle^+_{\Lambda_n,\beta_c}.$$
We therefore disregard the FK-Ising measure $\phi_{\Lambda_n,\beta}^1$ throughout this section and always work directly with the Ising measures of the form $\langle \cdot \rangle^\tau_{\Lambda,J,h}$, as defined in \eqref{eq:defIsing}. We begin with a proof of Theorem \ref{thm:FKwired meanfieldLB}.

\subsection{The lower bound}

In this section, we fix $d\geq 2$ and prove a mean-fied lower bound on $\langle \sigma_0\rangle^+_{\Lambda_n,\beta_c}$.

\begin{proof}[Proof of Theorem~\textup{\ref{thm:FKwired meanfieldLB}}] 
	Let $\varepsilon>0$ to be chosen small enough, and set for $n\geq 1$, $h_n:=(\varepsilon/ n)^{3}$. We consider the inhomogeneous magnetic fields $\mathsf{h}(t)$, $t\in[0,1]$, given by
	\begin{equation}
		\mathsf{h}(t):=th_n \cdot \mathds{1}_{\Lambda_n} + h_n \cdot \mathds{1}_{\Lambda_n^c}.
	\end{equation}
	By Theorem \ref{thm:MF lower bound} and the fundamental theorem of calculus,
	\begin{equation}\label{eq: lb proof 1}
		\frac{c_{\textup{mag}}\varepsilon}{n}\leq m(\beta_c,h_n) = \langle \sigma_0\rangle_{\beta_c,{\mathsf{h}(1)}} = \langle \sigma_0\rangle_{\beta_c,{\mathsf{h}(0)}} + \int_0^{1}\partial_t \langle \sigma_0\rangle_{\beta_c,{\mathsf{h}(t)}}\mathrm{d}t.
	\end{equation}
	Now, combining Lemma \ref{lem:derivatives}, Proposition \ref{pr.Ding}, and the infrared bound \eqref{eq:IRB} yields
	\begin{align}
		\begin{split}\label{eq: lb proof 2}
			\int_0^{1}\partial_t \langle \sigma_0\rangle_{\beta_c,{\mathsf{h}(t)}}\mathrm{d}t&=h_n\sum_{x\in \Lambda_n} \int_0^1\langle \sigma_0;\sigma_x\rangle_{\beta_c,\mathsf{h}(t)}\mathrm{d}t
			\\&\leq h_n\sum_{x\in \Lambda_n} \langle \sigma_0\sigma_x\rangle_{\beta_c}
			\leq C n^2 \varepsilon^3 n^{-3}=C\varepsilon^3 n^{-1},
		\end{split}
	\end{align}
	for some constant $C=C(d)>0$. Plugging \eqref{eq: lb proof 2} in \eqref{eq: lb proof 1}, and choosing $\varepsilon$ small enough so that $C\varepsilon^3 < \frac12 c_{\textup{mag}}\varepsilon$ gives that, for every $n$ large enough,
	\begin{equation}
		\frac{c_0}{n}\leq \langle \sigma_0\rangle_{\beta_c,{\mathsf{h}(0)}} \leq \langle \sigma_0\rangle_{\Lambda_n,\beta_c}^+, 
	\end{equation}
	where $c_0:=\tfrac{1}{2}c_{\textup{mag}}\varepsilon$, and in the last inequality we used Markov's property (see for instance \cite[Chapter~3]{FriedliVelenikIntroStatMech2017}) and monotonicity in boundary conditions (see \cite[Lemma~3.23]{FriedliVelenikIntroStatMech2017}) to argue that
	\begin{equation}
	 \langle \sigma_0\rangle_{\beta_c,{\mathsf{h}(0)}}=\langle \langle \sigma_0\rangle_{\Lambda_n,\beta_c}^{\sigma_{|\Lambda_n^c}}\rangle_{\beta_c,\mathsf{h}(0)}\leq \langle \sigma_0\rangle_{\Lambda_n,\beta_c}^+.
	 \end{equation}
	  To conclude the proof, it remains to potentially decrease the value of $c_0$ to accommodate the small values of $n$.
\end{proof}

\subsection{The upper bound}

We now turn to the proof of the upper bounds in Theorems \ref{thm:FKwired d>4} and \ref{thm:FKwired d=4}. As a warm-up, we show how one can easily obtain that $\rho\geq \tfrac{1}{6}$ in dimensions $d\geq 4$, i.e. for every $n\geq 1$,
\begin{equation} \label{eq:warm-up-bound-wired}
	\langle \sigma_0\rangle^+_{\Lambda_n,\beta_c}\lesssim \frac{1}{n^{1/6+o(1)}},
\end{equation}
where $o(1)$ tends to $0$ as $n$ tends to infinity.

\subsubsection{Warm-up: an easy upper bound}\label{section: easy lower bound on rho}\label{ss.EasyUB}

Let $d\geq 4$. The constants $c,C$ only depend on $d$ and may change from line to line. Let $n\geq 1$. We introduce a ``boundary'' magnetic field $\mathsf{h}\in (\mathbb R^+)^{\Lambda_n}$ on $\Lambda_n$, which we define as follows: $\mathsf{h}_x=\sum_{y\notin \Lambda_n}\beta\mathds{1}_{x\sim y}$ if $x\in\partial \Lambda_n$, and $\mathsf{h}_x=0$ otherwise. Recalling Remark \ref{rem:more general ising models}, one has, for every $h\geq 0$,
\begin{equation}
	\langle \sigma_0\rangle_{\Lambda_n,\beta,h}^+=\langle \sigma_0\rangle_{\Lambda_n,\beta,h+\mathsf{h}},
\end{equation}
where we used the convention that $h$ denotes the magnetic field on $\Lambda_n$ which is constant equal to $h$. Using the fundamental theorem of calculus, we obtain 
\begin{equation}\label{eq: easy upper bound 1}
	\langle \sigma_0\rangle_{\Lambda_n,\beta_c,h}^+-\langle \sigma_0\rangle_{\Lambda_n,\beta_c,h}=\int_0^1 \partial_t \langle \sigma_0\rangle_{\Lambda_n,\beta_c,h+t\mathsf{h}}{\mathrm{d}}t\leq 2d \sum_{x\in \partial \Lambda_n}\max_{0\leq t \leq 1}\langle \sigma_0;\sigma_x\rangle_{\Lambda_n,\beta_c,h+t\mathsf{h}}.
\end{equation}
Using \cite[Theorem~1.2]{ott2020sharp}, we get
\begin{equation}\label{eq: easy upper bound 1.5}
	\max_{0\leq t \leq 1}\langle \sigma_0;\sigma_x\rangle_{\Lambda_n,\beta_c,h+t\mathsf{h}}\leq \langle \sigma_0\sigma_x\rangle_{\Lambda_n,\beta_c}\cosh(h)^{-n/2}.
\end{equation}

	\begin{Rem}
		The above bound is stated in \cite{ott2020sharp} for homogeneous positive fields but the proof easily adapts to our present setting. 
	\end{Rem}

Now, the idea is to choose $h$ in terms of $n$ in such a way that the right-hand side in \eqref{eq: easy upper bound 1.5} decays super-polynomially fast. Set $h=n^{-1/2}\log n$. As a result, there exist $c_1,C_1>0$ (which only depend on $d$) such that, for every $n\geq 1$,
\begin{equation}\label{eq: easy upper bound 2}
	\max_{0\leq t \leq 1}\langle \sigma_0;\sigma_x\rangle_{\Lambda_n,\beta_c,h+t\mathsf{h}}\leq C_1e^{-c_1(\log n)^2}.
\end{equation}
Plugging \eqref{eq: easy upper bound 2} in \eqref{eq: easy upper bound 1} yields
\begin{equation}
	\langle \sigma_0\rangle_{\Lambda_n,\beta_c,h}^+\leq \langle \sigma_0\rangle_{\Lambda_n,\beta_c,h}+C_1e^{-c_1(\log n)^2}.
\end{equation}
Using Theorem \ref{thm:MF upper bound} then yields, for every $n\geq 2$,
\begin{equation}
	\langle \sigma_0\rangle^+_{\Lambda_n,\beta_c}\leq \langle \sigma_0\rangle_{\Lambda_n,\beta_c,h}^+\leq \frac{C_2(\log n)^{4/3}}{n^{1/6}},
\end{equation}
where $C_2=C_2(d)>0$.
This concludes the proof that $\rho\geq 1/6$. 

\begin{Rem} In \cite{hu2025polynomial}, the bound $\rho>0$ is an essential input for the study of the Glauber dynamics of the critical Ising model in dimensions $d>4$. The short argument presented above is therefore sufficient for this purpose, and also simpler than the proof that follows.
\end{Rem}

\subsubsection{The upper bound in dimensions $d>4$}\label{sec:wired_d>4}

In this section, we fix $d>4$ and prove the upper bound in Theorem \ref{thm:FKwired d>4}. 
The argument is based on the following ingredients. We use Theorem \ref{thm:MF upper bound} to get $C_{\textup{mag}}=C_{\textup{mag}}(d)\geq 1$ such that for every $h\geq 0$,
\begin{equation} \label{eq:pwd>4 1}
	\langle \sigma_0 \rangle_{\beta_c, h} \leq C_{\textup{mag}} h^{1/3},
\end{equation}
Recall that for $S\subset \mathbb Z^d$ containing $0$, 
\begin{equation} 
	\varphi_{\beta}(S)=\beta\sum_{\substack{u\in S\\ v\notin S\\ u\sim v}}\langle \sigma_0\sigma_u\rangle_{S,\beta}.
\end{equation}
By Lemma \ref{lem:existence of S which drops d>4}, there exists $C_1=C_1(d)>0$ such that the following holds: for every $\beta <\beta_c$, and every $n \geq C_1(\beta_c - \beta)^{-1/2}$, there exists $S \subset \Lambda_{n/2}$ which satisfies
\begin{equation}\label{eq:pwd>4 2}
	\varphi_{\beta}(S)\leq \frac{1}{10}.
\end{equation}

\begin{proof}[Proof of the upper bound in Theorem \textup{\ref{thm:FKwired d>4}}]
	For sake of compactness, we will abuse notation and write $n/2$ to mean the integer part of $n/2$. 
	Let $H \in (1, \infty)$ be an appropriate constant (fixed later on in the proof), and set
	\begin{equation}
	h_n = \left( \frac{H}{n} \right)^3. 
	\end{equation}
	We will show by induction that, for $H$ chosen appropriately, and every $i\geq 1$, one has
	\begin{equation}\label{eq: diadic upper bound wired d>4}
		\langle \sigma_0 \rangle^+_{\Lambda_{2^i}, \beta_c, h_{2^i}} \leq \frac{C_\star}{2^i}, \tag{$\star$}
	\end{equation}
where $C_\star=C_\star(H):=4C_{\textup{mag}}H$. This immediately gives the desired upper bound by setting $C:=2C_\star$. Indeed, by monotonicity of the magnetisation in the magnetic field \eqref{eq:monot Griffiths} 
	and in the volume \eqref{eq:monot Mag}: for every $i\geq 0$, and every $2^{i} < m \leq 2^{i+1}$ 
	\begin{equation}
		\langle \sigma_0 \rangle_{\Lambda_m, \beta_c}^+ \leq \langle \sigma_0 \rangle^+_{\Lambda_{2^i}, \beta_c, h_{2^i}} \stackrel{\eqref{eq: diadic upper bound wired d>4}}\leq \frac{C_\star}{2^i}\leq \frac{2C_\star}{m}.
	\end{equation}

	We now prove \eqref{eq: diadic upper bound wired d>4}.  Observe that the bound is trivial when $2^i \leq C_\star$, providing the initial step. 
	Assume that \eqref{eq: diadic upper bound wired d>4} holds for $i = k$, and let us show that it also holds for $i=k+1$. Write $n:=2^{k+1}$. We will prove that---provided $H$ is large enough---the following renormalisation inequality holds:
	\begin{align}\label{eq:pwd>4 3}
		\langle \sigma_0 \rangle_{\Lambda_n,\beta_c,h_n}^+ &\leq  
		\frac{1}{4} \langle \sigma_0 \rangle_{\Lambda_{n/2},\beta_c,h_{n/2}}^+ +  \langle \sigma_0 \rangle_{\beta_c,h_{n/2}}.
	\end{align}
	By the induction assumption \eqref{eq: diadic upper bound wired d>4} applied to $i=k$ and the bound  \eqref{eq:pwd>4 1}, it follows from \eqref{eq:pwd>4 3} that 
	\begin{equation}
	\langle \sigma_0 \rangle_{\Lambda_n,\beta_c,h_n}^+ \leq \frac{1}{4} \frac{C_\star}{(n/2)} + C_{\textup{mag}}(h_{n/2})^{1/3} = \frac{C_\star}{n},
	\end{equation}
	which is exactly \eqref{eq: diadic upper bound wired d>4} for $i=k+1$. 
	 
	We turn to the proof of \eqref{eq:pwd>4 3}. By \eqref{eq:pwd>4 2}, there exists $B=B(d)>0$ such that, for every $\ell\geq 1$, there exists $S_\ell\subset \Lambda_{\ell/2-1}$ containing $0$ and satisfying 
	\begin{equation}\label{eq:pwd>4 4}
		\varphi_{\beta_\ell}(S_\ell) \leq \frac{1}{10}\leq \frac{1}{4},
	\end{equation}
	where $\beta_\ell$ is defined as
	\begin{equation}
	\beta_\ell = \beta_c - \left(\frac{B}{\ell}\right)^2. 
	\end{equation}

	The constant $B$ is now fixed once and for all and only depends on $d$. Define an interaction $J=J(n)$ on $E(\mathbb Z^d)$ as follows: $J_e=\beta_n$ for every $e\in E(\Lambda_{n/2})$, and $J_e=\beta_c$ for every $e\in E(\mathbb Z^d)\setminus E(\Lambda_{n/2})$. Write
	\begin{equation} \label{eq:pwd>4 5}
		\begin{split}
			\langle \sigma_0 \rangle_{\Lambda_n, \beta_c, h_n}^+ &= \big( \langle \sigma_0 \rangle_{\Lambda_n, \beta_c, h_n}^+ -  \langle \sigma_0 \rangle_{\Lambda_{n}, J, h_{n/2}}^+ \big) 
			\\
			&\qquad\qquad + \big( \langle \sigma_0 \rangle_{\Lambda_{n}, J, h_{n / 2}}^+ - \langle \sigma_0 \rangle_{\Lambda_{n}, J, h_{n/2}} \big) \\
			&\qquad\qquad + \langle \sigma_0 \rangle_{\Lambda_{n}, J, h_{n/2}} \\
			&=: (\mathrm{I}) + (\mathrm{II}) + (\mathrm{III}). 
		\end{split}
	\end{equation}
	Notice immediately that for $(\mathrm{III})$, we can apply \eqref{eq:monot Griffiths}  
	to get
	\begin{equation} \label{eq:pwd>4 6}
		(\mathrm{III})=\langle \sigma_0 \rangle_{\Lambda_{n}, J, h_{n/2}} \leq  \langle \sigma_0 \rangle_{\beta_c, h_{n/2}}. 
	\end{equation}
	For $(\mathrm{II})$, we apply Lemma \ref{lem: random current proposition} to $S_n\subset \Lambda_{n/2-1}$---given by \eqref{eq:pwd>4 4}---to obtain
	\begin{equation} \label{eq:pwd>4 7}
		\begin{split}
			(\mathrm{II})=\langle \sigma_0 \rangle_{\Lambda_{n}, J, h_{n/2}}^+ - \langle \sigma_0 \rangle_{\Lambda_{n}, J, h_{n/2}} &\leq \varphi_{\beta_n}(S_n) \max_{x \in \partial^{\mathrm{ext}} S} \langle \sigma_x \rangle_{\Lambda_n, J, h_{n/2}}^+ \\
			&\leq \frac{1}{4} \langle \sigma_0 \rangle_{\Lambda_{n/2}, \beta_c, h_{n/2}}^+ 
					\end{split}
	\end{equation}
	where we used \eqref{eq:monot Griffiths} and \eqref{eq:monot Mag} in the second line. The renormalisation inequality \eqref{eq:pwd>4 3} follows from plugging \eqref{eq:pwd>4 6} and \eqref{eq:pwd>4 7} in \eqref{eq:pwd>4 5} once we know that $(\mathrm{I})$ is negative. We now prove this last fact.
	Here, the point is that both $\beta$ and $h$ vary simultaneously. 
	Since the change in both parameters is related to the other via classical differential inequalities, 
	this allows us to control the increments. 
	Let us be more precise. 
	For $t \in [0, \beta_c - \beta_n]$ and $e\in E(\Lambda_{n/2})$ define the continuous change in parameters: 
	\begin{equation}
	J_{e}(t) = \beta_n + t, \hspace{1cm} \text{and} \hspace{1cm} h(t) = h_{n/2} - t\frac{{h_{n/2} - h_n}}{\beta_c - \beta_n}. 
	\end{equation}
	We also set $J_{e}(t)=\beta_c$ for $e\in E(\Lambda_n)\setminus E(\Lambda_{n/2})$. By the fundamental theorem of calculus, we can write		
	\begin{equation} \label{eq:pwd>4 8}
		(\mathrm{I}) = \langle \sigma_0 \rangle^+_{\Lambda_n, \beta_c, h_n} - \langle \sigma_0 \rangle^+_{\Lambda_n, J, h_{n/2}} = \int_0^{\beta_c - \beta_n} \partial_t \langle \sigma_0 \rangle_{\Lambda_n, J(t), h(t)}^+ \textup{d}t.
	\end{equation}
	We will show that for an appropriate choice of $H$, one has $\partial_t \langle \sigma_0 \rangle_{\Lambda_n, J(t), h(t)}^+\leq 0$, thereby showing that $(\mathrm{I})\leq 0$.
	Using Lemma \ref{lem:derivatives} and the chain rule gives 
	\begin{align}
	\begin{split}\label{eq:pwd>4 9}
		\partial_t \langle \sigma_0 \rangle_{\Lambda_n, J(t), h(t)}^+ = \sum_{xy\in E(\Lambda_{n/2})} & \langle \sigma_0 ; \sigma_{x}\sigma_y \rangle_{\Lambda_n, J(t), h(t)}^+ \\ 
		&- \left(\frac{h_{n/2} - h_n}{\beta_c - \beta_n}\right) \sum_{x \in \Lambda_n} \langle \sigma_0; \sigma_x \rangle_{\Lambda_n, J(t), h(t)}^+.
	\end{split}
	\end{align}
	By the Griffiths--Hurst--Sherman (GHS) inequality \cite{GriffithsHurstShermanConcavity1970}, we have 
	\begin{align}\notag
		\langle \sigma_0 ; &\sigma_{x}\sigma_y \rangle_{\Lambda_n, J(t), h(t)}^+ \\&\leq \langle \sigma_0 ; \sigma_x \rangle^+_{\Lambda_n, J(t), h(t)}\langle \sigma_y \rangle^+_{\Lambda_n, J(t), h(t)} + \langle \sigma_0 ; \sigma_y \rangle^+_{\Lambda_n,J(t), h(t)}\langle \sigma_x \rangle^+_{\Lambda_n, J(t), h(t)}  \notag\\
		&\leq C_1\Big(\langle \sigma_0 ; \sigma_x \rangle^+_{\Lambda_n, J(t), h(t)}\langle \sigma_x \rangle^+_{\Lambda_n,J(t), h(t)} + \langle \sigma_0 ; \sigma_y \rangle^+_{\Lambda_n, J(t), h(t)}\langle \sigma_y \rangle^+_{\Lambda_n, J(t), h(t)}\Big),
	\end{align}
	for some constant $C_1=C_1(d)>0$, where the second inequality follows from the (straightforward) comparison of the magnetisation at neighbors $x$ and $y$ in Lemma \ref{lem:trivial comparison neighbors} (and the fact that the truncated two-point functions are non-negative by Proposition \ref{prop:secondineqGriffiths}). 
	Recall that $\langle \sigma_0 \rangle_{\Lambda,J, \mathsf{h}}^+$ is decreasing in $\Lambda$ \eqref{eq:monot Mag} and increasing in $J,\mathsf{h}$ \eqref{eq:monot Griffiths}. Hence,
	\begin{equation}
		\langle \sigma_x \rangle_{\Lambda_n, J(t), h(t)}^+ \leq \langle \sigma_0 \rangle_{\Lambda_{n/2}, \beta_c, h_{n/2}}^+. 
		\end{equation}
	Plugging the two previously displayed equations in \eqref{eq:pwd>4 9} yields
	\begin{equation} \label{eq:mag-derived-t-wired-d>4 new proof}
		\partial_t \langle \sigma_0 \rangle_{\Lambda_n, \beta(t), h(t)}^+ \leq \left[ 2C_1 \langle \sigma_0 \rangle_{\Lambda_{n/2}, \beta_c, h_{n/2}}^+  - \left(\frac{h_{n/2} - h_n}{\beta_c - \beta_n}\right)  \right] \sum_{x \in \Lambda_{n/2}} \langle \sigma_0 ; \sigma_{x} \rangle_{\Lambda_n, J(t), h(t)}^+. 
	\end{equation}
	By the induction hypothesis and the definition of $\beta_n$ and $h_n$, one has
	\begin{align}\notag
	2C_1 \langle \sigma_0 \rangle_{\Lambda_{n/2}, \beta_c, h_{n/2}}^+  - \left(\frac{h_{n/2} - h_n}{\beta_c - \beta_n}\right) 
	&\leq \frac{2C_1 4C_{\textup{mag}}H}{n/2} - \frac{H^3}{n^3} \frac{n^2}{B^2} \\
	&= \frac{H}{B^2n} \, (16 C_1 C_{\textup{mag}} B^2 - H^2).
 	\end{align}
	Hence, by fixing $H:=4B\sqrt{C_1C_{\textup{mag}}}$, 
	the right-hand side of \eqref{eq:mag-derived-t-wired-d>4 new proof} is negative, and by \eqref{eq:pwd>4 8} 
	\begin{equation}
	(\mathrm{I})\leq 0.
	\end{equation}
	This concludes the proof.
\end{proof}

\subsubsection{The upper bound in dimension $d=4$}
In this section, we let $d=4$ and prove the upper bound in Theorem \ref{thm:FKwired d=4}. The proof follows the strategy used in dimensions $d>4$, so we only present the main modifications in the argument. The main inputs of the proof are slightly different than before. By Theorem \ref{thm:MF upper bound}, there exists $C_{\textup{mag}}\geq 1$ such that, for every $h\geq 0$,
\begin{equation} \label{eq:pwd=4 1}
	\langle \sigma_0 \rangle_{\beta_c, h} \leq C_{\textup{mag}} h^{1/3}|\log h|.
\end{equation}
Moreover, by Lemma \ref{lem:existence of S which drops d=4}, there exists $C_1>0$ such that the following holds: for every $\beta <\beta_c$, and every $n \geq C_1(\beta_c - \beta)^{-1/2}|\log \beta_c-\beta|$, there exists $S \subset \Lambda_{n/2}$ which satisfies
\begin{equation}\label{eq:pwd=4 2}
	\varphi_{\beta}(S)\leq \frac{1}{10}.
\end{equation}

\begin{proof}[Proof of the upper bound in Theorem \textup{\ref{thm:FKwired d=4}}] The proof follows the same lines as the case $d>4$ from Section~\ref{sec:wired_d>4}, so we only explain the main modifications.
	This time, we set
	\begin{equation}
	h_n := \left( \frac{H (\log n)^{3/2}}{n} \right)^3,
	\end{equation}
	and prove that, for $H\in[1,\infty)$ chosen appropriately, we have
	\begin{equation}\label{eq:pwd=4 3}
		\langle \sigma_0 \rangle^+_{\Lambda_{2^i}, \beta_c, h_{2^i}} \leq \frac{C_\star (\log 2^i)^{5/2}}{2^i} \tag{$\star$}
	\end{equation}
	for each $i \geq 1$, where $C_\star=C_\star(H):=12C_{\textup{mag}}H$. As before, this immediately gives the desired upper bound by setting $C:=2C_\star$. We proceed by induction. The bound is trivial for $\tfrac{2^i}{(\log 2^i)^{5/2}}\leq C_\star$. We assume that \eqref{eq:pwd=4 3} holds for $i=k$ and show that it holds with $i=k+1$.
	Again, \eqref{eq:pwd=4 3} with $i=k+1$ will follow easily once we prove the following inequality
	\begin{align}\label{eq:pwd=4 4}
		\langle \sigma_0 \rangle_{\Lambda_n,\beta_c,h_n}^+ &\leq  
		\frac{1}{4} \langle \sigma_0 \rangle_{\Lambda_{n/2},\beta_c,h_{n/2}}^+ +  \langle \sigma_0 \rangle_{\beta_c,h_{n/2}},
	\end{align}
	where $n=2^{k+1}$. Indeed, plugging the induction assumption \eqref{eq:pwd=4 3} for $i=k$ and the bound  \eqref{eq:pwd=4 1} into \eqref{eq:pwd=4 4} gives 
	\begin{align}\notag
	\langle \sigma_0 \rangle_{\Lambda_n,\beta_c,h_n}^+ &\leq \frac{1}{4} \frac{C_\star \log(n/2)^{5/2}}{n/2} + C_{\textup{mag}} (h_{n/2})^{1/3} |\log h_{n/2}| \\
	&\leq \frac{(C_\star/2) (\log n)^{5/2}}{n} + C_{\textup{mag}} \frac{2H (\log n)^{3/2}}{n} 3\log n  = C_\star \frac{(\log n)^{5/2}}{n}.
	\end{align}
	
	Now, in order to prove \eqref{eq:pwd=4 4}, we use the same decomposition \eqref{eq:pwd>4 5} as before, except that here we take 
	\begin{equation}
		\beta_n=\beta_c-\left(\frac{B\log n}{n}\right)^2.
	\end{equation}
	Picking $B$ large enough above implies again---by \eqref{eq:pwd=4 2}---that there exists $S_n \subset \Lambda_{n/2-1}$ such that 
	\begin{equation}
	\varphi_{\beta_n}(S_n)\leq \frac{1}{10} \leq\frac{1}{4}.
	\end{equation} 
	The bounds $(\mathrm{III})\leq \langle \sigma_0 \rangle_{\beta_c,h_{n/2}}$ and $(\mathrm{II})\leq \frac{1}{4} \langle \sigma_0 \rangle_{\Lambda_{n/2},\beta_c,h_{n/2}}^+$ follow for the same reasons as before. We only need to explain how to prove that $(\mathrm{I})\leq 0$. Using the same interpolation $J(t)$ as before, the differential inequality \eqref{eq:mag-derived-t-wired-d>4 new proof} remains true. By the induction hypothesis and the definitions of $h_n$ and $\beta_n$, we can bound
		\begin{align*}
		2C_1 \langle \sigma_0 \rangle_{\Lambda_{n/2}, \beta_c, h_{n/2}}^+  - \left(\frac{h_{n/2} - h_n}{\beta_c - \beta_n}\right) 
		&\leq 2C_1\frac{12 C_{\textup{mag}}H (\log n)^{5/2}}{n/2} - \frac{H^3 (\log n)^{9/2} n^2}{n^3 B^2 (\log n)^2} \\
		& = \frac{H(\log n)^{5/2}}{B^2 n} \,( 48C_1 C_{\textup{mag}}B^2 - H^2 ).
	\end{align*}
	By taking $H:= 4B\sqrt{3C_1 C_{\textup{mag}}}$, we deduce that $(\mathrm{I})\leq 0$ as before. This concludes the proof.
\end{proof}
 
\subsection{The mixing exponent $\iota$}
The goal of this section is to prove Corollary \ref{coro:iota d>4}. It will be an easy consequence of the following result.

\begin{Lem}\label{lem:interm iota} Let $d\geq 2$, $\beta\geq 0$, and $n\geq 1$. Then, for every $e\in E(\Lambda_n)$, one has
\begin{equation}
	\phi^1_{\Lambda_{2n},\beta}[\omega_e]-\phi^0_{\Lambda_{2n},\beta}[\omega_e]\leq \big(\langle \sigma_0\rangle_{\Lambda_{n},\beta}^+\big)^2.
\end{equation}
\end{Lem}
\begin{proof} Let $e=xy\in E(\Lambda_n)$. By the Edwards--Sokal coupling (see for instance \cite{DuminilLecturesOnIsingandPottsModels2019}), one has
\begin{equation}
	\phi^1_{\Lambda_{2n},\beta}[\omega_e=1]=(1-e^{-2\beta})\langle \mathds{1}_{\sigma_x=\sigma_y}\rangle^+_{\Lambda_{2n},\beta}=(1-e^{-2\beta})\frac{1+\langle \sigma_x\sigma_y\rangle^+_{\Lambda_{2n},\beta}}{2}.
\end{equation}
As similar statement holds with the pair $(\phi^0_{\Lambda_{2n},\beta},\langle \cdot\rangle_{\Lambda_{2n},\beta})$. As a result,
\begin{align}\notag
	\phi^1_{\Lambda_{2n},\beta}[\omega_e]-\phi^0_{\Lambda_{2n},\beta}[\omega_e]&=\frac{1-e^{-2\beta}}{2}\Big(\langle \sigma_x\sigma_y\rangle_{\Lambda_{2n},\beta}^+-\langle \sigma_x\sigma_y\rangle_{\Lambda_{2n},\beta}\Big)
	\\&\leq \langle \sigma_x\rangle^+_{\Lambda_{2n},\beta}\langle \sigma_y\rangle_{\Lambda_{2n},\beta}^+\leq \big(\langle \sigma_0\rangle_{\Lambda_{n},\beta}^+\big)^2,
\end{align}
where we used Proposition \ref{pr.Ding} in the first inequality and \eqref{eq:monot Mag} (and the fact that $x,y\in \Lambda_n$) in the second.
\end{proof}

\begin{proof}[Proof of Corollary \textup{\ref{coro:iota d>4}}] Applying Lemma \ref{lem:interm iota} to $\beta=\beta_c$, $e\in E(\Lambda_n)$, and taking the supremum over $e$ yields
\begin{equation}
	\sup_{e\in E(\Lambda_n)}\Big(\phi^1_{\Lambda_{2n},\beta_c}[\omega_e]-\phi^0_{\Lambda_{2n},\beta_c}[\omega_e]\Big)\leq \big(\langle \sigma_0\rangle_{\Lambda_{n},\beta_c}^+\big)^2.
\end{equation}
The proof follows readily from an application of Theorem \ref{thm:FKwired d>4} (in dimensions $d>4$) and \ref{thm:FKwired d=4} (in dimension $d=4$).
\end{proof}

\section{One-arm exponent of the free FK-Ising measure}\label{sec:FKfree}

We now turn to the computation of the one-arm exponent for the measure $\phi_{\beta_c}$. This section is organised as follows. In Section \ref{sec:FKentropy}, we prove an entropic bound that is very similar to the one of Dewan and Muirhead for Bernoulli percolation \cite{DewanMuirhead} (see also \cite{hutchcroft2022derivation}). In Sections \ref{sec:FKd>6UB} and \ref{sec:FKd=6UB}, we apply this entropic bound to prove the upper bounds in Theorems \ref{thm:FKfree d>6} and \ref{thm:FKfree d=6}. In Section \ref{sec:FKsecond moment}, we use a second moment method to obtain the lower bounds in Theorems \ref{thm:FKfree d>6} and \ref{thm:FKfree d=6}. In Section \ref{sec:volume_bounds}, we compute the volume exponents as stated in Theorem \ref{thm:volume_exp}.
Finally, in Section \ref{sec:FKnonMF}, we consider the FK-Ising model below its upper-critical dimension: we prove that $\rho\leq \tfrac{3}{2}$ in dimension $d=5$ (Theorem \ref{thm: FKfree d=5}), and $\rho=1$ in dimension $d=4$ (Theorem \ref{thm: FKfree d=4}).
\subsection{Entropic bound.}\label{sec:FKentropy}

If $\Lambda\subset \mathbb Z^d$ and $\beta\geq 0$, we let
\begin{equation}\label{eq:def triangle and bubble}
	B_{\Lambda}(\beta) := \sup_{x \in \Lambda} \sum_{y \in \Lambda} \langle \sigma_x \sigma_y \rangle_{\Lambda, \beta}^2, \qquad \Triangle_{\Lambda}(\beta) := \sup_{x \in \Lambda} \sum_{y, z \in \Lambda}  \langle \sigma_x \sigma_y \rangle_{\Lambda, \beta} \langle \sigma_y \sigma_z \rangle_{\Lambda, \beta} \langle \sigma_z \sigma_x \rangle_{\Lambda, \beta}.
\end{equation}
Observe that $B_{\Lambda}(\beta)\leq \nabla_\Lambda(\beta)$.
When $\Lambda=\mathbb Z^d$, these quantities are respectively the \emph{bubble} diagram and the \emph{triangle} diagram. Their diagrammatic representations are presented in Figure \ref{fig:bubtriang}. They play a fundamental role in the study of the mean-field regime in statistical mechanics. Indeed, for the Ising model, it was shown that the \emph{bubble condition} $B(\beta_c)=B_{\mathbb Z^d}(\beta_c)<\infty$ implies that several critical exponents exist and take their mean-field value \cite{AizenmanGeometricAnalysis1982,AizenmanGrahamRenormalizedCouplingSusceptibility4d1983,AizenmanFernandezCriticalBehaviorMagnetization1986,AizenmanBarskyFernandezSharpnessIsing1987} (it also implies Gaussianity of critical scaling limits \cite[Appendix~C]{PanisTriviality2023}). Similarly, in the context of Bernoulli percolation, the \emph{triangle condition} $\nabla(p_c)=\nabla_{\mathbb Z^d}(p_c)<\infty$ is an indicator of mean-field behaviour \cite{AizenmanNewmanTreeGraphInequalities1984,BarskyAizenmanCriticalExponentPercoUnderTriangle1991,hutchcroft2022derivation}. Our entropic bound for the FK-Ising model involves the triangle diagram. By \eqref{eq:IRB}, this quantity is finite at criticality when $d>6$. This is an indicator that the dimension $d=6$ plays a pivotal role in the study of this model.

\begin{figure}[htb]
		\centering
		\includegraphics{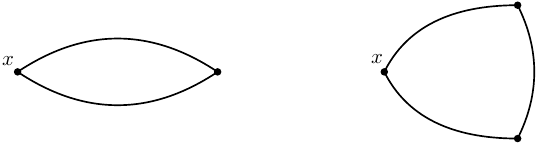}
		\caption{Diagrammatic representations of the \emph{bubble} (on the left) and \emph{triangle} (on the right) diagrams defined in \eqref{eq:def triangle and bubble}. Each line represents a two-point function. The diagrams are \emph{rooted} at some vertex $x\in \Lambda$, where $\Lambda\subset \mathbb Z^d$. Unlabelled vertices are summed over $\Lambda$.}
		\label{fig:bubtriang}
\end{figure}

Recall that $\mathcal C(0)$ denotes the vertex cluster of $0$. We let $\mathcal C_e(0)$ denote the edge cluster of $0$. 

\begin{Thm}[Entropic bound] \label{Thm: entropic bound pinsker} Let $d\geq 2$. There exists $C>0$ such that the following holds. For every $\Lambda\subset \mathbb Z^d$ finite containing $0$, every event $\mathcal A$ measurable with respect to $\mathcal C_e(0)$, and every $\beta_c/2\leq \beta'<\beta\leq \beta_c$, 
	\begin{equation}
		\big|\phi_{\Lambda, \beta}^0[\mathcal{A}] - \phi_{\Lambda, \beta'}^0[\mathcal{A}]\big| \leq C |\beta - \beta'| \sqrt{\Triangle_{\Lambda}(\beta_c)\max\{\phi_{\Lambda, \beta}^0[\mathcal{A}], \phi_{\Lambda, \beta'}^0[\mathcal{A}]\}\phi_{\Lambda, \beta'}^0\big[|\mathcal C(0)|\big]}. 
	\end{equation}
\end{Thm}

\begin{Rem} (1) To the best of our knowledge, this inequality is new for the FK-Ising model.

(2) In dimensions $d>6$, \eqref{eq:IRB} implies that $\nabla(\beta_c)<\infty$. Moreover, for every $\Lambda\subset \mathbb Z^d$, one has $\nabla_\Lambda(\beta_c)\leq \nabla(\beta_c)$. As a consequence, Theorem \ref{Thm: entropic bound pinsker} extends to infinite subsets of $\mathbb Z^d$.

(3) The entropic bound of Dewan and Muirhead \cite[Proposition~1.21]{DewanMuirhead} does not involve the triangle diagram. Hence, it holds in the infinite volume limit in every dimension $d\geq 2$.  
\end{Rem}

In the course of the proof of Theorem \ref{Thm: entropic bound pinsker}, it will be slightly more practical to work with a different parametrisation of the measure $\phi^0_{\Lambda,\beta}$. We abuse notations and define the measure $\phi^0_{\Lambda,p}$ for $p\in [0,1]$ as follows: for $\omega\in \{0,1\}^{E(\Lambda)}$,
\begin{equation}\label{eq:def FK measure with p}
	\phi^0_{\Lambda,p}[\omega]:=\frac{1}{\mathbf Z^{\textup{FK}}_{\Lambda,p}}2^{k^0(\omega)}\prod_{e\in E(\Lambda)}p^{\omega_e}(1-p)^{1-\omega_e}.
\end{equation}
One can check that for every $\beta\geq 0$, if $p_\beta=1-e^{-2\beta}$, then $\phi^0_{\Lambda,p_\beta}=\phi^0_{\Lambda,\beta}$. Moreover, one has (see for instance \cite{DuminilLecturesOnIsingandPottsModels2019})  
\begin{equation}\label{eq:connection pbeta partition function to Ising}
	\mathbf{Z}^{\textup{FK}}_{\Lambda,p_\beta}=(1-p_\beta)^{|E(\Lambda)|/2}\mathbf Z^{\textup{Ising}}_{\Lambda,\beta}.
\end{equation}
In this section, to ease notations slightly, we write $\phi_{\Lambda,p}=\phi^0_{\Lambda,p}$ to denote the law of the cluster of the origin: it's density is hence defined for sets $C\subset E(\Lambda)$ containing $0$ as $\phi_{\Lambda,p}(C):=\phi^0_{\Lambda,p}[\mathcal C_e(0)=C]$. 
We also let $p_c:=1-e^{-2\beta_c}$. Our aim is to study the measure $\phi_{\Lambda,p}$ as $p\leq p_c$ varies. A useful quantity to consider is the so-called \emph{relative entropy}: for $p'<p$ in $[0,1]$, we let
\begin{equation}\label{e.HH}
	\mathrm{H}(\phi_{\Lambda, p'} \| \phi_{\Lambda, p}) := \sum_{C \owns 0} \phi_{\Lambda, p'}(C) \log\bigg(\frac{\phi_{\Lambda, p'}(C)}{\phi_{\Lambda, p}(C)}\bigg), 
\end{equation}
where  the sum runs over all connected sets $C \subset E(\Lambda)$ containing $0$. Let us stress again that this is the relative entropy between the law of the \emph{cluster of the origin} under the FK-Ising measures at parameters $p'$ and $p$ respectively, not between the full FK-Ising measures.
Since $\Lambda$ is finite, the relative entropy is well-defined and finite for each $p',p$ fixed. This quantity is directly connected to our problem thanks to Pinsker's inequality (see e.g. \cite[Lemma 2.3]{DewanMuirhead}): for every event $\mathcal A$ which is measurable in terms of the edge cluster of the origin $\mathcal C_e(0)$, one has
\begin{equation}\label{eq:pinsker}
	\big|\phi_{\Lambda, p}[\mathcal{A}] - \phi_{\Lambda, p'}[\mathcal{A}]\big|\leq 
		\sqrt{2\max\{\phi_{\Lambda, p}[\mathcal{A}], \phi_{\Lambda, p'}[\mathcal{A}]\} \mathrm{H}(\phi_{\Lambda, p'} \| \phi_{\Lambda, p})}.
\end{equation}
Therefore, to prove Theorem \ref{Thm: entropic bound pinsker}, it suffices to bound $\mathrm{H}(\phi_{\Lambda, p'} \| \phi_{\Lambda, p})$. We perform this by controlling the derivatives in $p$ of this function and applying Taylor's theorem. Observe that since $\Lambda$ is finite, 
$
	p \mapsto \mathrm{H}(\phi_{\Lambda, p'} \| \phi_{\Lambda, p})
$
is smooth, and its derivatives are given by
\begin{equation}\label{eq:expression derivates H}
	\frac{\partial^k}{\partial p^k} \mathrm{H}(\phi_{\Lambda, p'} \| \phi_{\Lambda, p}) = -\sum_{C \owns 0} \phi_{\Lambda, p'}(C) \frac{\partial^k}{\partial p^k} \log\left(\phi_{\Lambda, p}(C)\right)
\end{equation}
The first derivative vanishes at $p=p'$ (this is a classical fact, we provide a proof below). 
Our goal is to bound the second derivative of the relative entropy,
a quantity known in statistics as \emph{Fisher's information}. 
For compactness, define the function
\begin{equation}
	F_{\Lambda}(p', p) := \mathrm{H}(\phi_{\Lambda, p'} \| \phi_{\Lambda, p}). 
\end{equation}
To avoid confusion, dimension-dependent constants will be denoted by $K_i$ in the remainder of this subsection. Theorem \ref{Thm: entropic bound pinsker} will follow from the next proposition. 

\begin{Prop} \label{Prop: entropic bound} Let $d\geq 2$.
	There exists a  constant $K=K(d)>0$ such that the following holds. For every  $\Lambda\subset \mathbb Z^d$ finite containing $0$, and every $1-e^{-\beta_c}\leq p'<p \leq p_c$,
	\begin{equation}
		\frac{\partial^2}{\partial p^2} F_{\Lambda}(p', p) \leq K \Triangle_{\Lambda}(\beta) \phi_{\Lambda, p'} \big[ |\mathcal C(0)| \big], 
	\end{equation}
	\end{Prop}

It is crucial that the expected cluster size is taken at $p'$ and not at $p$. 

\begin{proof}[Proof of Theorem \textup{\ref{Thm: entropic bound pinsker}}] Fix $1-e^{-\beta_c} \leq p' < p \leq p_c$. 
	Since $p \mapsto F_{\Lambda}(p', p)$ is a smooth function, Taylor's theorem with Lagrange remainder term gives
	\begin{equation}\label{eq:peb1}
		F_{\Lambda}(p', p) = F_{\Lambda}(p', p') + (p - p') \frac{\partial}{\partial r} F_{\Lambda}(p', r)\mid_{r = p'} + \frac{(p - p')^2}{2}\frac{\partial^2}{\partial r^2} F_{\Lambda}(p', r)\mid_{r = \tilde{p}}, 
	\end{equation}
	for some $\tilde{p} \in [p', p]$. Clearly, one has $F_{\Lambda}(p', p') = 0$. The first derivative of the relative entropy also vanishes: 
		\begin{equation}\label{eq:peb2}
		\frac{\partial}{\partial r}F_{\Lambda}(p', r)\mid_{r = p'} = - \frac{\partial}{\partial r} \Big(\sum_{C \owns 0} \phi_{\Lambda, r}(C)\Big) \mid_{r = p'} = 0,
	\end{equation}
	where we used that $\sum_{C \owns 0} \phi_{\Lambda, r}(C)=1$ for every $r\in [0,1]$. Then, Proposition \ref{Prop: entropic bound} shows that
	\begin{equation}\label{eq:peb3}
		 \frac{\partial^2}{\partial r^2} F_{\Lambda}(p', r)\leq K \Triangle_{\Lambda}(\beta_c) \phi_{\Lambda, p'}\big[|\mathcal C(0)|\big]
	\end{equation}
	uniformly in $1-e^{-\beta_c}\leq p' < r \leq p_c$. Therefore, combining \eqref{eq:peb1}--\eqref{eq:peb3} yields,  
	\begin{equation}
		0 \leq F_{\Lambda}(p', p)=\mathrm{H}(\phi_{\Lambda, p'} \| \phi_{\Lambda, p}) \leq \frac{(p - p')^2}{2} K\Triangle_{\Lambda}(\beta_c)\phi_{\Lambda, p'}\big[ |\mathcal C(0)| \big]. 
	\end{equation}
	Finally, let $\mathcal{A}$ be an event measurable in terms of $\mathcal C_e(0)$. Then, by \eqref{eq:pinsker}, one has
	\begin{equation}
		|\phi_{\Lambda, p}[\mathcal{A}] - \phi_{\Lambda, p'}[\mathcal{A}]|\leq \sqrt{K}|p-p'|\sqrt{\Triangle_{\Lambda}(\beta_c)\max \{\phi_{\Lambda, p}[\mathcal{A}], \phi_{\Lambda, p'}[\mathcal{A}]\} \phi_{\Lambda, p'}\big[|\mathcal C(0)|\big]}.
	\end{equation}
The result as stated in Theorem \ref{Thm: entropic bound pinsker} follows by letting $p'=1-e^{-2\beta'}$ and $p = 1 - e^{-2\beta}$, and choosing $C$ appropriately in terms of $K$. 
\end{proof}

We turn to the proof of Proposition \ref{Prop: entropic bound}. For a set $C \subset E(\Lambda)$, we write $\overline{C}$ for the collection of edges that share an endpoint with an edge of $C$.

\begin{proof}[Proof of Proposition \textup{\ref{Prop: entropic bound}}]
	Let $\Lambda$ be a finite box and fix $1-e^{-\beta_c} \leq p' < p \leq p_c$. Recall from \eqref{eq:connection pbeta partition function to Ising} that, writing $p=1-e^{-2\beta}$, one has
	\begin{equation}\label{eq:ppeb1}
		\mathbf{Z}^{\textup{FK}}_{\Lambda,p}=(1-p)^{|E(\Lambda)|/2}\mathbf Z^{\textup{Ising}}_{\Lambda,\beta}.
	\end{equation}
	Let $0 \in C \subset E(\Lambda)$ be a set of edges, and denote by $\partial C$ the edges neighboring $C$. 
	Write $\Lambda \setminus \overline{C}$ for the graph induced by deleting all of $\overline{C}$. 
	Using \eqref{eq:ppeb1},
	\begin{equation}\label{eq:ppeb2}
		\phi_{\Lambda, p}(C) = 2 p^{|C|}(1 - p)^{|\partial C|} \frac{\mathbf{Z}^{\mathrm{FK}}_{\Lambda \setminus \overline{C}, p}}{\mathbf{Z}^{\mathrm{FK}}_{\Lambda, p}}=2p^{|C|}(1-p)^{|\partial C|}(1-p)^{-|\overline{C}|/2}\frac{\mathbf{Z}^{\textup{Ising}}_{\Lambda\setminus \overline{C},\beta}}{\mathbf{Z}^{\textup{Ising}}_{\Lambda,\beta}}, 
	\end{equation}
	where the factor $2$ comes from the weight $2^{k^0(\omega)}$ in \eqref{eq:def FK measure with p}. Plugging \eqref{eq:ppeb2} in \eqref{eq:expression derivates H} (with $k=2$) yields
	\begin{align}\notag
		\frac{\partial^2}{\partial p^2} F_{\Lambda}(p', p) =& - \sum_{C \owns 0} \phi_{\Lambda, p'}(C) \Big[\frac{\partial^2}{\partial p^2} \log\left(2 p^{|C|}(1 - p)^{|\partial C|}\right) + \frac{\partial^2}{\partial p^2} \log\left( (1 - p)^{|\overline{C}|/2}\right)\Big] \\
		&-\sum_{C \owns 0 } \phi_{\Lambda, p'}(C) \frac{\partial^2}{\partial p^2} \log\bigg(\frac{\mathbf{Z}^{\textup{Ising}}_{\Lambda\setminus \overline{C},\beta}}{\mathbf{Z}^{\textup{Ising}}_{\Lambda,\beta}} \bigg). \label{eq:ppeb3}
	\end{align}
	Observe that, there exists $K_1=K_1(d)>0$ such that, for $1-e^{-\beta_c}\leq p\leq p_c$, one has
	\begin{equation}\label{eq:ppeb4}
		\left|\frac{\partial^2}{\partial p^2} \log\left(2 p^{|C|}(1 - p)^{|\partial C|}\right) + \frac{\partial^2}{\partial p^2} \log\left( (1 - p)^{|\overline{C}|/2}\right)\right|\leq K_1|\overline{C}|.
	\end{equation}
	For the second term in \eqref{eq:ppeb3}, we would like to take the derivative with respect to $\beta$ instead of $p$. Applying the chain rule twice, we obtain
	\begin{equation}\label{eq:ppeb5}
		\frac{\partial^2}{\partial p^2} \log\bigg( \frac{\mathbf{Z}^{\textup{Ising}}_{\Lambda\setminus \overline{C},\beta}}{\mathbf{Z}^{\textup{Ising}}_{\Lambda,\beta}}\bigg) = \frac{1}{2(1 - p)^2}\frac{\partial}{\partial \beta}\log\bigg( \frac{\mathbf{Z}^{\textup{Ising}}_{\Lambda\setminus \overline{C},\beta}}{\mathbf{Z}^{\textup{Ising}}_{\Lambda,\beta}} \bigg) + \frac{1}{4(1 - p)^2} \frac{\partial^2}{\partial^2 \beta}\log\bigg( \frac{\mathbf{Z}^{\textup{Ising}}_{\Lambda\setminus \overline{C},\beta}}{\mathbf{Z}^{\textup{Ising}}_{\Lambda,\beta}}\bigg), 
	\end{equation}
	where we used the relation $p=1-e^{-2\beta}$ to get
	\begin{equation}
		\Big(\frac{\partial \beta}{\partial p}\Big)^2=\frac{1}{4(1-p)^2}, \qquad \frac{\partial^2 \beta}{\partial p^2}=\frac{1}{2(1-p)^2}.
	\end{equation}
	Putting \eqref{eq:ppeb4} and \eqref{eq:ppeb5} in \eqref{eq:ppeb3}, and using again that $1-e^{-\beta_c}\leq p \leq p_c$, yields the existence of $K_2=K_2(d)>0$ such that
	\begin{equation}\label{eq:ppeb6}
		\frac{\partial^2}{\partial p^2} F_{\Lambda}(p',p) \leq K_2 \sum_{C \owns 0} \phi_{\Lambda, p'}(C)\bigg( |\overline{C}| - \frac{\partial}{\partial \beta}\log\bigg( \frac{\mathbf{Z}^{\textup{Ising}}_{\Lambda\setminus \overline{C},\beta}}{\mathbf{Z}^{\textup{Ising}}_{\Lambda,\beta}} \bigg) - \frac{\partial^2}{\partial \beta^2}\log\bigg( \frac{\mathbf{Z}^{\textup{Ising}}_{\Lambda\setminus \overline{C},\beta}}{\mathbf{Z}^{\textup{Ising}}_{\Lambda,\beta}} \bigg)\bigg)
	\end{equation}
	We are left to bound the terms involving the derivatives in $\beta$ in \eqref{eq:ppeb6}. We will show that there exist constants $K_3,K_4$ depending only on $d$ such that:
	\begin{equation} \label{eq: bound first derivative entropic}
		-\phi_{\Lambda, p'} \bigg[ \frac{\partial}{\partial \beta}\log\bigg( \frac{\mathbf{Z}^{\textup{Ising}}_{\Lambda\setminus \overline{C},\beta}}{\mathbf{Z}^{\textup{Ising}}_{\Lambda,\beta}} \bigg)  \bigg]
		\leq K_3 B_{\Lambda}(\beta) \phi_{\Lambda, p'} \big[|\overline{\mathcal C_e(0)}| \big]
	\end{equation}
	and
	\begin{equation} \label{eq: bound second deriviatve entropic}
		-\phi_{\Lambda, p'} \bigg[ \frac{\partial^2}{\partial \beta^2} \log\bigg(\frac{\mathbf{Z}^{\textup{Ising}}_{\Lambda\setminus \overline{C},\beta}}{\mathbf{Z}^{\textup{Ising}}_{\Lambda,\beta}}  \bigg) \bigg]
		\leq K_4\Triangle_{\Lambda}(\beta) \phi_{\Lambda, p'} \big[|\overline{\mathcal C_e(0)}| \big].
	\end{equation}
	Before moving to the proof of \eqref{eq: bound first derivative entropic} and \eqref{eq: bound second deriviatve entropic}, let us quickly show how they allow to conclude the proof. Plugging \eqref{eq: bound first derivative entropic} and \eqref{eq: bound second deriviatve entropic} and the inequality $B_{\Lambda}(\beta)\leq \Triangle_{\Lambda}(\beta)\leq \Triangle_{\Lambda}(\beta_c)$ in \eqref{eq:ppeb6} yields
	\begin{equation}
		\frac{\partial^2}{\partial p^2} F_{\Lambda}(p',p) \leq 4d\Big(K_2+(K_2K_3+K_2K_4)\Triangle_{\Lambda}(\beta_c)\Big)\phi_{\Lambda,p'}\big[|\mathcal C(0)|\big]\leq K_5 \Triangle_{\Lambda}(\beta_c)\phi_{\Lambda,p'}\big[|\mathcal C(0)|\big],
	\end{equation}
	for some $K_5=K_5(d)>0$, where we used that $|\overline{\mathcal C_e(0)}|\leq 2\cdot 2d\cdot |\mathcal C(0)|$. This is exactly the statement of Proposition \ref{Prop: entropic bound}.

	\paragraph{Proof of \eqref{eq: bound first derivative entropic}.} Fix a set $0 \in C \subset E(\Lambda)$. We start with the first derivative. Below, we abuse notations and also view $\overline{C}$ as a vertex set made of the endpoints of all the elements of $\overline{C}$. A standard computation gives
	\begin{align}\notag
		-\frac{\partial}{\partial \beta}\log\bigg( \frac{\mathbf{Z}^{\textup{Ising}}_{\Lambda\setminus \overline{C},\beta}}{\mathbf{Z}^{\textup{Ising}}_{\Lambda,\beta}} \bigg)  = \sum_{xy \in E(\Lambda)} \langle \sigma_x \sigma_y \rangle_{\Lambda, \beta} - \langle \sigma_x \sigma_y \rangle_{\Lambda \setminus \overline{C}, \beta}&\leq \sum_{xy \in E(\Lambda)} \sum_{z \in \partial\overline{C}} \langle \sigma_x \sigma_z \rangle_{\Lambda \setminus \overline{C}, \beta}\langle \sigma_z\sigma_y \rangle_{\Lambda, \beta} \notag\\
		&\leq \sum_{z \in \overline{C}} \sum_{xy \in E(\Lambda)} \langle \sigma_x \sigma_z \rangle_{\Lambda, \beta} \langle \sigma_z \sigma_y \rangle_{\Lambda, \beta} \notag\\
		&\leq K_6 \sum_{z\in \overline{C}}\sum_{x\in \Lambda}\langle \sigma_x\sigma_z\rangle_{\Lambda,\beta}^2\notag\\
		&\leq 2K_6 |\overline{C}| B_{\Lambda}(\beta), 
	\end{align}
	where in the first line we used Lemma \ref{Lem: bubble bound with C}, in the second line we used \eqref{eq:monot Griffiths}, in the third line we used Lemma \ref{lem:trivial comparison neighbors} and $K_6=K_6(d)>0$, and in the last line we used the definition of $B_{\Lambda}(\beta)$. Taking the expectation with respect to $C$ shows \eqref{eq: bound first derivative entropic}.
	
	\paragraph{Proof of \eqref{eq: bound second deriviatve entropic}.} The second derivative is more delicate. 
	A similar computation as above shows that
	\begin{align*}
		-\frac{\partial^2}{\partial \beta^2} \log\bigg(\frac{\mathbf{Z}^{\textup{Ising}}_{\Lambda\setminus \overline{C},\beta}}{\mathbf{Z}^{\textup{Ising}}_{\Lambda,\beta}}  \bigg)  = \sum_{xy, zw \in E(\Lambda)}  \langle \sigma_x\sigma_y ; \sigma_z \sigma_w \rangle_{\Lambda, \beta} - \langle \sigma_x \sigma_y; \sigma_z \sigma_w \rangle_{\Lambda \setminus \overline{C}, \beta}. 
	\end{align*}
Introduce the interaction $J(t)$ (on $\Lambda$) defined for $t\in [0,1]$ by:
	\begin{equation}
		J_{e}(t) = \begin{cases}
			\beta & e \in E(\Lambda) \setminus \overline{C} \\
			t\beta & e \in \overline{C} 
		\end{cases},
	\end{equation}
	and let $\langle \cdot \rangle_{t} := \langle \cdot \rangle_{\Lambda, J(t)}$ be the induced Ising measure. 
	Writing
	\begin{equation}
		\langle \sigma_x\sigma_y ; \sigma_z \sigma_w ; \sigma_{u}\sigma_v \rangle_{t}
		= \langle \sigma_x \sigma_y \sigma_z \sigma_w ; \sigma_u\sigma_v \rangle_{t} - \langle \sigma_x\sigma_y\rangle_t \langle \sigma_z \sigma_w; \sigma_u\sigma_v \rangle_t - \langle \sigma_z \sigma_w \rangle_t\langle \sigma_x \sigma_y ; \sigma_u \sigma_v \rangle_t, 
	\end{equation}
	we find, by the fundamental theorem of calculus (and Remark \ref{Rem:second derivative}) that
	\begin{equation}
		 \langle \sigma_x\sigma_y ; \sigma_z \sigma_w \rangle_{\Lambda, \beta} - \langle \sigma_x \sigma_y; \sigma_z \sigma_w \rangle_{\Lambda \setminus \overline{C}, \beta} = \beta \int_0^1 \sum_{uv \in \overline{C}} \langle \sigma_x\sigma_y ; \sigma_z \sigma_w ; \sigma_{u}\sigma_v \rangle_{t} \mathrm{d}t. 
	\end{equation}
	Combining Lemmas \ref{lem:double truncated} and \ref{lem:trivial comparison neighbors} shows that there exists some constant $K_7=K_7(d)>0$ such that
	\begin{equation}
		\langle \sigma_x\sigma_y ; \sigma_z \sigma_w ; \sigma_{u}\sigma_v \rangle_{t} \leq K_7 \langle \sigma_x \sigma_z\rangle_{1} \langle \sigma_w \sigma_u \rangle_1 \langle \sigma_v \sigma_y \rangle_1 = K_7  \langle \sigma_x \sigma_z\rangle_{\beta} \langle \sigma_w \sigma_u \rangle_{\beta} \langle \sigma_v \sigma_y \rangle_{\beta}
	\end{equation}
	Combining the last two displayed equations, 
	\begin{equation}
		-\frac{\partial^2}{\partial \beta^2} \log\bigg(\frac{\mathbf{Z}^{\textup{Ising}}_{\Lambda\setminus \overline{C},\beta}}{\mathbf{Z}^{\textup{Ising}}_{\Lambda,\beta}}  \bigg) \leq K_7\beta \sum_{\substack{xy, zw \in E(\Lambda) \\ uv \in \overline{C}}} \langle \sigma_x \sigma_z\rangle_{\Lambda, \beta} \langle \sigma_w \sigma_u \rangle_{\Lambda, \beta}\langle \sigma_v \sigma_y \rangle_{\Lambda, \beta} \leq K_8 |\overline{C}| \Triangle_{\Lambda}(\beta),
	\end{equation}
	where the second inequality follows again from Lemma \ref{lem:trivial comparison neighbors}, and where $K_8=K_8(d)>0$. Averaging over $C$ concludes the proof of \eqref{eq: bound second deriviatve entropic}.
	\end{proof}

\subsection{Upper bound in dimensions $d>6$}\label{sec:FKd>6UB}
In this subsection, we prove the upper bound in Theorem \ref{thm:FKfree d>6}. The proof will follow from the following simple renormalisation inequality.
\begin{Lem}\label{lem:renormalisation inequality d>6} Let $d>6$. There exist $K>0$ and $k_0\geq 1$ such that the following holds. For every $k\geq k_0$, letting $n=2^k$ and $u_n:=\phi_{\beta_c}[0\connect{}\partial \Lambda_n]$, one has
\begin{equation}
	u_n\leq \frac{1}{10}u_{n/2}+\frac{K}{n}\sqrt{u_n}.
\end{equation}
\end{Lem}
Before proving this result, we show how it implies the upper bound in Theorem \ref{thm:FKfree d>6}.
\begin{proof}[Proof of the upper bound in Theorem \textup{\ref{thm:FKfree d>6}}] We let $K,k_0>0$ be given by Lemma \ref{lem:renormalisation inequality d>6}, and let $u_n=\phi_{\beta_c}[0\connect{}\partial \Lambda_n]$. Observe that it suffices to show that there exists a constant $C_0>0$ such that for every $k\geq 0$, 
\begin{equation}\label{eq:pfd>6 1}
	u_{2^k}\leq \frac{C_0}{(2^k)^2}.
\end{equation}
Indeed, if \eqref{eq:pfd>6 1} holds, then for every $n\geq 1$, if $k\geq 0$ is such that $2^k\leq n<2^{k+1}$, then
\begin{equation}
	u_n\leq u_{2^k}\leq \frac{C_0}{(2^k)^2}\leq \frac{4C_0}{n^2}.
\end{equation}
Set $C_0:=4K^2\vee (2^{k_0})^2$. We now prove \eqref{eq:pfd>6 1} by induction on $k\geq 0$. Since $u_{2^k}\leq 1$ and $C_0\geq (2^{k_0})^2$, \eqref{eq:pfd>6 1} holds with $k\leq k_0$. Let $k\geq k_0$ and assume that it holds with $k-1$. By Lemma \ref{lem:renormalisation inequality d>6}, one has
\begin{equation}\label{eq:pfd>6 2}
	u_{2^{k}}\leq \frac{1}{10}u_{2^{k-1}}+\frac{K}{2^k}\sqrt{u_{2^k}}.
\end{equation}
Applying the induction hypothesis gives
\begin{equation}\label{eq:pfd>6 3}
	u_{2^{k}}\leq \frac{(2C_0/5)}{(2^k)^2}+\frac{K}{2^k}\sqrt{u_{2^k}}.
\end{equation}
We now take cases: if $\frac{K}{2^k}\sqrt{u_{2^k}}>\frac{(2C_0/5)}{(2^k)^2}$, then \eqref{eq:pfd>6 3} rewrites
\begin{equation}
	u_{2^k}\leq \frac{2K}{2^k}\sqrt{u_{2^k}},
\end{equation}
which implies that
\begin{equation}
	u_{2^k}\leq \frac{4K^2}{(2^k)^2}\leq \frac{C_0}{(2^k)^2},
\end{equation}
by our choice of $C_0$. If not, then \eqref{eq:pfd>6 3} rewrites
\begin{equation}
	u_{2^k}\leq \frac{(4C_0/5)}{(2^k)^2}\leq \frac{C_0}{(2^k)^2}.
\end{equation}
In either case, we obtain the desired result. This concludes the proof.
\end{proof}
We now turn to the proof of Lemma \ref{lem:renormalisation inequality d>6}. We rely on the entropic bound obtained in Theorem \ref{Thm: entropic bound pinsker}, 
together with Lemma \ref{lem:existence of S which drops d>4} as in the wired case: if $d>4$, there exists $C_1=C_1(d)>0$ such that the following holds: for every $\beta <\beta_c$, and every $n \geq C_1(\beta_c - \beta)^{-1/2}$, there exists $S \subset \Lambda_{n/2}$ which satisfies
\begin{equation}\label{eq:input FK d>6}
	\varphi_{\beta}(S)\leq \frac{1}{10}.
\end{equation}
\begin{proof}[Proof of Lemma \textup{\ref{lem:renormalisation inequality d>6}}] Let $k\geq 1$ to be taken large enough. Set $n=2^k$, and let $\beta'=\beta'(n)=\beta_c-\frac{A}{n^2}$ for some $A>0$ to be chosen large enough. Let $C_1$ be given by \eqref{eq:input FK d>6} above. Fix $A$ large enough so that
\begin{equation}
	C_1(\beta_c-\beta')^{-1/2}=\frac{C_1}{\sqrt{A}}n<n-1.
\end{equation}
and then take $k\geq k_0$, where $k_0$ is such that
\begin{equation}\label{eq:prid>6 0}
	\beta_c-\frac{A}{(2^{k_0})^2}\geq \frac{\beta_c}{2}.
\end{equation}
By \eqref{eq:input FK d>6}, there exists $S\subset \Lambda_{(n-1)/2}$ such that $\varphi_{\beta'}(S)\leq \tfrac{1}{10}$. An application of the BK-type inequality of Proposition \ref{prop:BK type for FK} with $S$ gives
\begin{equation}\label{eq:prid>6 1}
	\phi_{\beta'}[0\connect{}\partial\Lambda_n]\leq \sum_{\substack{u\in S\\v\notin S\\u\sim v}}\phi_{S,\beta'}^0[0\connect{}u]\beta \phi_{\beta'}[v\connect{}\partial \Lambda_n]\leq \varphi_{\beta'}(S)\phi_{\beta'}[0\connect{}\Lambda_{n/2}]\leq \frac{1}{10}u_{n/2},
\end{equation}
where in the last inequality we used the defining property of $S$ and the fact that $\phi_{\beta_c}$ stochastically dominates $\phi_{\beta'}$ (see Proposition \ref{prop:stoch dom FK}). Since $d>6$, the infrared bound \eqref{eq:IRB} implies that the triangle diagram $\nabla(\beta_c)$ (defined in \eqref{eq:def triangle and bubble}) is finite. We can therefore apply the infinite-volume version of Theorem \ref{Thm: entropic bound pinsker} with $\beta=\beta_c$, $\beta'$ (which is licit by \eqref{eq:prid>6 0}), and $\mathcal A:=\{0\connect{}\partial \Lambda_n\}$ to obtain $C_2=C_2(d)>0$ such that
\begin{equation}\label{eq:prid>6 2}
	u_n\leq \phi_{\beta'}[0\connect{}\partial\Lambda_n]+\frac{C_2A}{n^2}\sqrt{\nabla(\beta_c)u_n \chi(\beta')},
\end{equation}
where we used the aforementioned stochastic domination to argue that $\phi_{\beta'}[0\connect{}\partial\Lambda_n]\leq u_n$, and the Edwards--Sokal coupling \eqref{eq: consequences of ESC} to argue that $\phi_{\beta'}[|\mathcal C(0)|]=\sum_{x\in \mathbb Z^d}\langle \sigma_0\sigma_x\rangle_{\beta'}=\chi(\beta')$. Plugging \eqref{eq:prid>6 1} and the bound $\chi(\beta')\lesssim n^2$ (see Proposition \ref{prop:susceptibility}) in \eqref{eq:prid>6 2} yields the existence of $C_3=C_3(d)>0$ such that
\begin{equation}
	u_n\leq \frac{1}{10}u_{n/2}+\frac{C_3}{n}\sqrt{u_n}.
\end{equation} 
This concludes the proof.
\end{proof}

\subsection{Upper bound in dimension $d=6$}\label{sec:FKd=6UB}
The proof of Theorem \ref{thm:FKfree d=6} ($d = 6$) is similar to that of Theorem \ref{thm:FKfree d>6} ($d > 6$). 
The main difference is that the triangle condition no longer holds: by \eqref{eq: lower bound full-space sec 2}, one has $\nabla(\beta_c)=\infty$. However, the sequence $(\nabla_{\Lambda_n}(\beta_c))_{n\geq 1}$ blows up sufficiently slowly (see \eqref{eq: triangle bound d=6}) to extend the previous strategy to obtain an appropriate renormalisation inequality, as expressed in the following lemma. 

\begin{Lem}\label{lem:renormalisation inequality d=6}
	Let $d = 6$. There exist $K>0$ and $k_0 \geq 1$ such that the following holds. 
	For every $k \geq k_0$, letting $n = 2^k$ and $u_n := \phi_{\beta_c}[0 \leftrightarrow \partial \Lambda_n]$, one has
	\begin{equation}
		u_{n} \leq \frac{1}{10} u_{n / 2}  + \frac{K \sqrt{\log(n)}}{n} \sqrt{u_n} + \frac{K}{n^3}. 
	\end{equation}
\end{Lem}

Before proving this result, let us argue how Theorem \ref{thm:FKfree d=6} follows from it. 

\begin{proof}[Proof of Theorem \textup{\ref{thm:FKfree d=6}}]
	Let $K, k_0$ be given by Lemma \ref{lem:renormalisation inequality d=6} and assume without loss of generality that $k_0 \geq 3$. 
	Write $u_n = \phi_{\beta_c}[0 \leftrightarrow \partial \Lambda_n]$. 
	It suffices to prove that there is some constant $C_0 < \infty$ such that for all $k \geq 1$, 
	\begin{equation} \label{eq:pfd=6 1}
		u_{2^k} \leq \frac{C_0 k}{(2^k)^2}. 
	\end{equation}
	Set $C_0 := 4K^2 \vee (2^{k_0})^2$. 
	We prove \eqref{eq:pfd=6 1} by induction on $k \geq 1$. 
	Since $u_k \leq 1$ for all $k \geq 1$, \eqref{eq:pfd=6 1} holds for $k \leq k_0$ by definition of $C_0$. 
	Now, fix $k \geq k_0$ and assume that \eqref{eq:pfd=6 1} holds for $k - 1$. 
	Applying Lemma \ref{lem:renormalisation inequality d=6} gives
	\begin{equation*} 
			u_{2^{k}} \leq \frac{1}{10} u_{2^{k - 1}}  + \frac{K \sqrt{k}}{2^{k}} \sqrt{u_{2^k}} + \frac{K}{(2^k)^3}
	\end{equation*}
	and hence by the induction hypothesis we have
	\begin{equation}\label{eq:pfd=6 2}
			u_{2^{k}} \leq \frac{2C_0}{5} \frac{k}{(2^k)^2}  + \frac{K\sqrt{k}}{2^{k}} \sqrt{u_{2^k}} + \frac{K}{(2^k)^3}. 
	\end{equation}
	We now take cases: suppose that
	\begin{equation} \label{eq:pfd=6 3}
		\frac{2C_0}{5} \frac{k}{(2^{k})^2} + \frac{K}{(2^k)^3} < \frac{K \sqrt{k}}{2^k}\sqrt{u_{2^k}}. 
	\end{equation}
	Then, \eqref{eq:pfd=6 2} rewrites as
	\begin{equation}
		u_{2^{k}} \leq \frac{2K\sqrt{k}}{2^k}\sqrt{u_{2^k}},
	\end{equation}
	which implies that
	\begin{equation} \label{eq:pfd=6 4}
		u_{2^k} \leq \frac{4K^2 k}{(2^{k})^2} \leq \frac{C_0 k }{(2^k)^2}
	\end{equation}
	by definition of $C_0$. If \eqref{eq:pfd=6 3} does not hold, then \eqref{eq:pfd=6 2} instead becomes
	\begin{equation} \label{eq:pfd=6 5}
		u_{2^{k}} \leq \frac{4/5C_0 k}{(2^k)^2} + \frac{2K}{(2^k)^3} \leq \frac{4/5C_0 k}{(2^k)^2} + \frac{C_0 / 5}{(2^k)^2} \leq \frac{C_0 k}{(2^k)^2},
	\end{equation}
	where we used that $2K \leq 4K^2 \leq C_0$, and that $2^k\geq 8\geq 5$ since $k_0 \geq 3$. In conclusion, in either of the two cases \eqref{eq:pfd=6 1} is true for $k$, which concludes the induction. 
\end{proof}

We next turn to the proof of the renormalisation Lemma \ref{lem:renormalisation inequality d=6}. 
We will again rely on the entropic bound obtained in Theorem \ref{Thm: entropic bound pinsker}, 
together with Lemma \ref{lem:existence of S which drops d>4}: if $d>4$, there exists $C_1 = C_1(d)>0$ such that the following holds: for every $\beta <\beta_c$, and every $n \geq C_1(\beta_c - \beta)^{-1/2}$, there exists $S \subset \Lambda_{n/2 }$ which satisfies
\begin{equation}\label{eq:input FK d=6}
	\varphi_{\beta}(S)\leq \frac{1}{10}.
\end{equation}
However, this time we will need two additional inputs. 
First, there exists a constant $C_2 < \infty$ such that for every $\alpha \geq 1$, and every $n \geq 1$, we have
\begin{equation} \label{eq: triangle bound d=6}
	\nabla_{\Lambda_{n^{\alpha}}}(\beta_c) \leq C \sum_{x, y \in \Lambda_{2n^{\alpha}}} \frac{1}{(1\vee|x|)^{d - 2}} \frac{1}{(1\vee|x - y|)^{d - 2}} \frac{1}{(1\vee|y|)^{d - 2}} \leq C_2 \alpha \log(n).
\end{equation} 
This follows from \eqref{eq:monot Griffiths} together with the infrared bound \eqref{eq:IRB}. Second, we will need the following comparison lemma, 
which is a consequence of the existence of an algebraic bound on the wired one-arm probability (as obtained in Theorem \ref{thm:FKwired d>4}). 

\begin{Lem} \label{lem:FKfree-free-boundary-pushing}
	Let $d > 4$. There exist $\alpha > 1$ and $C_3>0$ such that for all $n \geq 1$, 
	\begin{equation}
		\phi_{\beta_c}[0 \connect{} \partial \Lambda_n] \leq \phi^0_{\Lambda_{n^\alpha}, \beta_c}[0 \connect{} \partial \Lambda_n] + \frac{C_3}{n^3}. 
	\end{equation}
\end{Lem}

\begin{proof}
	Let $\alpha \geq 1$. By Markov's property and stochastic domination (see Propositions \ref{prop:DMP} and \ref{prop:stoch dom FK}), for every $n\geq 1$,
	\begin{equation}
		\phi_{\beta_c}[0 \connect{} \partial \Lambda_{n^{\alpha}}]=\phi_{\beta_c}\big[\phi^\xi_{\Lambda_{n^\alpha},\beta_c}[0 \connect{} \partial \Lambda_{n^{\alpha}}]\big] \leq \phi_{\beta_c, \Lambda_{n^{\alpha}}}^1[0 \connect{}\partial \Lambda_{n^{\alpha}}] \lesssim \frac{1}{n^{\alpha}},
	\end{equation}
	where we used Theorem \ref{thm:FKwired d>4}.
	Write $\mathcal{C}_n = \{\Lambda_n \connect{}\partial \Lambda_{2 n^{\alpha}}\}$. 
	The last displayed equation, together with a union bound and inclusion of events shows the existence of $C_3=C_3(d)>0$ such that, 
	\begin{equation}
		\phi_{\beta_c}[\mathcal{C}_n] \leq \sum_{x \in \partial \Lambda_{n}} \phi_{\beta_c}[x \connect{}\partial \Lambda_{2n^{\alpha}}] \lesssim  n^{d-1}\phi_{\beta_c}[0 \connect{} \partial \Lambda_{n^\alpha}] \leq \frac{C_3 n^{d - 1}}{n^{\alpha}}. 
	\end{equation}
	We now let $\alpha = d - 1 + 3$. Observe that on the complement of $\mathcal{C}_n$, there exists a cut set (i.e. a set of edges that separates $\Lambda_n$ and $\partial \Lambda_{n^\alpha}$) of closed edges between $\Lambda_n$ and $\partial \Lambda_{n^{\alpha}}$. 
	Since $\{0 \leftrightarrow \partial \Lambda_n\}$ is an increasing event, the domain Markov property and monotonicity in boundary conditions yield
	\begin{equation}
		\phi_{\beta_c}[0 \connect{} \partial \Lambda_n] = \phi_{\beta_c}[0 \connect{} \partial \Lambda_n, \mathcal{C}_n^c] + \phi_{\beta_c}[\mathcal{C}_n] \leq \phi^0_{\Lambda_{n^\alpha}, \beta_c}[0 \connect{} \partial \Lambda_n] + \frac{C_3}{n^3}, 
	\end{equation}
	which concludes the proof. 
\end{proof}

We can now finally prove Lemma \ref{lem:renormalisation inequality d=6}. 

\begin{proof}[Proof of Lemma \textup{\ref{lem:renormalisation inequality d=6}}]
	Fix $C_1$ as in \eqref{eq:input FK d=6} and $\alpha=d-1+3$ as in Lemma \ref{lem:FKfree-free-boundary-pushing}. 
	Take $A$ be so large that $C_1 / \sqrt{A} < 1/2$ and $k_0 \geq 2$ so that 
	\begin{equation} \label{eq: check_beta_d=6}
		\beta_c - \frac{A}{(2^{k_0})^2} \geq \frac{\beta_c}{2}. 
	\end{equation}
	Fix $k \geq k_0$, write $n = 2^{k}$ and define $\beta_n = \beta_c - \frac{A}{n^2}$. 
	We thus have
	\begin{equation}
		C_1(\beta_c - \beta_n)^{-1/2} = \frac{C_1}{\sqrt{A}}n < n - 2, 
	\end{equation}
	since $n \geq 2^{k_0} \geq 4$. 
	Therefore \eqref{eq:input FK d=6} implies that for each such $n$, there exists a set $S_n \subset n/2 - 1$ for which $\varphi_{\beta_n}(S_n) < \frac{1}{10}$. 
	
	The BK-type inequality of Proposition \ref{prop:BK type for FK} applied to $S_n$ gives
	\begin{equation} \label{eq:FKfree-renorm 1}
		\begin{split}
		\phi^0_{\Lambda_{n^\alpha}, \beta_n}[0 \connect{} \partial \Lambda_n] &\leq \sum_{\substack{u \in S_n \\ v \notin S_n \\ u \sim v}} \phi_{S_n, \beta_n}^0[0\connect{} u]\beta_n \phi^0_{\Lambda_{n^\alpha}, \beta_n}[v \connect{} \partial \Lambda_n] \\
		&\leq \varphi_{\beta_n}(S_n)\phi_{\beta_n}[0 \connect{} \partial \Lambda_{n / 2}] \leq \frac{1}{10}u_{n/2}, 
		\end{split}
	\end{equation}
	where in the second inequality we used inclusion of events, translation invariance and Proposition \ref{prop:stoch dom FK}, and in the last inequality we used the defining property of $S_n$ and Proposition \ref{prop:stoch dom FK} again. 
	
	We next apply Theorem \ref{Thm: entropic bound pinsker} to $\Lambda = \Lambda_{n^\alpha}$, $\beta = \beta_c$, $\beta' = \beta_n$ (which is within the range by \eqref{eq: check_beta_d=6}) and $\mathcal{A} = \{0 \leftrightarrow \partial \Lambda_n\}$, 
	to obtain the existence of a constant $C_4 < \infty$ such that
	\begin{equation} \label{eq:FKfree-renorm 2}
		\phi^0_{\Lambda_{n^{\alpha}}, \beta_c}[0 \connect{}\partial \Lambda_n] \leq \phi_{\Lambda_{n^{\alpha}}, \beta_n}[0 \connect{}\partial \Lambda_n] + \frac{C_4 }{n^2} \sqrt{\nabla_{\Lambda_{n^\alpha}}(\beta_c) \phi^0_{\Lambda_{n^{\alpha}}, \beta_c}[0 \connect{} \partial \Lambda_n] \chi(\beta_n)}, 
	\end{equation}
	which follows from the exact same arguments as \eqref{eq:prid>6 2}. 
	Plugging equations \eqref{eq:FKfree-renorm 1} and \eqref{eq: triangle bound d=6} together with Proposition \ref{prop:susceptibility} into \eqref{eq:FKfree-renorm 2}, and using that $\phi^0_{\Lambda_{n^{\alpha}}, \beta_c}[0 \connect{} \partial \Lambda_n]\leq u_n$ (by Proposition \ref{prop:stoch dom FK})  gives the existence of a constant $C_5=C_5(d)>0$ such that
	\begin{equation} \label{eq:FKfree-renorm 3}
		\phi^0_{\Lambda_{n^\alpha}, \beta_n}[0 \connect{} \partial \Lambda_n] \leq \frac{1}{10}u_{n/2} + \frac{C_5 \sqrt{\log(n)}}{n} \sqrt{u_n}. 
	\end{equation}
	Now, by choice of $\alpha$, we can use Lemma \ref{lem:FKfree-free-boundary-pushing} to $u_n$ and consecutively \eqref{eq:FKfree-renorm 3} to obtain
	\begin{equation}
		u_n \leq \frac{1}{10}u_{n/2} + \frac{C_5 \sqrt{\log(n)}}{n}\sqrt{u_n} + \frac{C_3}{n^3}
	\end{equation}
	where $C_3$ is as in Lemma \ref{lem:FKfree-free-boundary-pushing}, showing the result. 
\end{proof}

\subsection{Lower bounds}\label{sec:FKsecond moment}
We now prove the lower bounds in Theorems \ref{thm:FKfree d>6} and \ref{thm:FKfree d=6}. We will in fact prove the following (slightly) stronger result which also includes the dimension $d=5$.

\begin{Prop} \label{prop: FK-free lower bound} 
	Let $d>4$. There exist $M>1$ and $c>0$ such that for every $n\geq 1$,
	\begin{equation}
		\phi^0_{\Lambda_{Mn}, \beta_c}[0\connect{\:}\partial\Lambda_n]\geq \frac{c}{n^2}.
	\end{equation}
\end{Prop}
The proof goes through a second moment method, similarly to Bernoulli percolation. An important input in our argument is the following result from\footnote{The result of \cite{CamiaJiangNewman21} is conditioned on an up-to-constant estimate on the two-point function at criticality. Such an input is now available thanks to \cite{DuminilPanis2024newLB}.} \cite{CamiaJiangNewman21}.
\begin{Thm}[{\hspace{1pt}\cite[Theorem~1]{CamiaJiangNewman21}}]\label{thm:CAMIAJIANGNEWMAN} Let $d>4$. There exist $M >2$ and $c_0>0$ such that, for every $n\geq 1$, and every $x,y\in \Lambda_{2n}$, one has
\begin{equation}
	\langle\sigma_x\sigma_y\rangle_{\Lambda_{Mn},\beta_c}\geq \frac{c_0}{(1\vee|x-y|)^{d-2}}.
\end{equation}
\end{Thm}

\begin{proof}[Proof of Proposition \textup{\ref{prop: FK-free lower bound}}] Let $M$ and $c_0$ be given by Theorem \ref{thm:CAMIAJIANGNEWMAN}. For $n\geq 1$, define
\begin{equation}\label{eq:def Z}
	\mathcal Z:=|\mathcal C(0)\cap (\Lambda_{2n}\setminus \Lambda_{n-1})|.
\end{equation}
By the Cauchy--Schwarz inequality,
\begin{equation}\label{eq:plbFK1}
	\phi_{\Lambda_{Mn},\beta_c}^0[0\connect{}\partial \Lambda_n]\geq \phi^0_{\Lambda_{Mn},\beta_c}[\mathcal Z>0]\geq \frac{\phi_{\Lambda_{Mn},\beta_c}^0[\mathcal Z]^2}{\phi_{\Lambda_{Mn},\beta_c}^0[\mathcal Z^2]}.
\end{equation}
On the one hand, Theorem \ref{thm:CAMIAJIANGNEWMAN} yields
\begin{equation}\label{eq:plbFK2}
	\phi_{\Lambda_{Mn},\beta_c}^0[\mathcal Z]=\sum_{x\in \Lambda_{2n}\setminus\Lambda_{n-1}}\langle \sigma_0\sigma_x\rangle_{\Lambda_{Mn},\beta_c}\geq \sum_{x\in \Lambda_{2n}\setminus\Lambda_{n-1}}\frac{c_0}{n^{d-2}} \gtrsim n^{d} n^{2-d}=n^2.
\end{equation}
On the other hand, an application of Proposition \ref{prop:tree graph} and \eqref{eq:monot Griffiths} gives 
\begin{align}\notag
	\phi_{\Lambda_{Mn},\beta_c}^0[\mathcal Z^2]&=\sum_{x,y\in \Lambda_{2n}\setminus\Lambda_{n-1}}\phi^0_{\Lambda_{Mn},\beta_c}[0\connect{}x,y]
	\\&\leq \sum_{\substack{z,z'\in \Lambda_{Mn},\\ z\sim z'}}\sum_{x,y\in \Lambda_{2n}\setminus\Lambda_{n-1}} \langle \sigma_0\sigma_{z'}\rangle_{\beta_c}\langle \sigma_x\sigma_{z'}\rangle_{\beta_c}\langle \sigma_y\sigma_z\rangle_{\beta_c}\notag
	\\&\lesssim \sum_{z\in \Lambda_{Mn}}\sum_{x,y\in \Lambda_{2n}\setminus\Lambda_{n-1}}\frac{1}{(1\vee|z|)^{d-2}}\frac{1}{(1\vee|x-z|)^{d-2}}\frac{1}{(1\vee|y-z|)^{d-2}}\notag
	\\&\lesssim n^d n^{6-d}=n^6, \label{eq:plbFK3}
	\end{align}
	where we have used the infrared bound \eqref{eq:IRB} in the third line. Plugging \eqref{eq:plbFK2} and \eqref{eq:plbFK3} in \eqref{eq:plbFK1} yields
	\begin{equation}
		\phi^0_{\Lambda_{Mn},\beta_c}[0\connect{}\partial \Lambda_n]\gtrsim \frac{(n^2)^2}{n^6}=\frac{1}{n^2},
	\end{equation}
	and concludes the proof.
\end{proof}

\subsection{Computation of the volume exponent}\label{sec:volume_bounds}

We now prove Theorem~\ref{thm:volume_exp}, which computes the volume critical exponent $\delta$ in all dimensions $d\geq6$. The proof follows rather easily from the one-arm upper bounds provided by Theorems~\ref{thm:FKfree d>6} and \ref{thm:FKfree d=6}, and simple moment estimates relying on bounds on the critical two-point function \eqref{eq:IRB}--\eqref{eq: lower bound full-space sec 2} and the tree graph inequality of Proposition~\ref{prop:tree graph}.
\begin{proof}[Proof of Theorem~\textup{\ref{thm:volume_exp}}] We first prove the upper bounds. Let $n\geq 1$. For every $N\geq2$, one has
	\begin{align}\notag
		\phi_{\beta_c}[|\mathcal C(0)|\geq n] &\leq \phi_{\beta_c}[0\leftrightarrow \partial \Lambda_N] + \phi_{\beta_c}[|\mathcal C(0)\cap \Lambda_N|\geq n] \\
		&\lesssim \frac{(\log N)^{\mathds{1}_{d=6}}}{N^2} + \frac{\phi_{\beta_c}[|\mathcal C(0)\cap \Lambda_N|]}{n}\notag \\
		&\lesssim \frac{(\log N)^{\mathds{1}_{d=6}}}{N^2} + \frac{N^2}{n},\label{eq:pvol1}
	\end{align}
	where in the second line we used Theorems~\ref{thm:FKfree d>6} and \ref{thm:FKfree d=6} together with Markov's inequality, and in the third line used the fact that $\phi_{\beta_c}[|\mathcal C(0)\cap \Lambda_N|] = \sum_{x\in \Lambda_N} \phi_{\beta_c}[0\leftrightarrow x]\lesssim N^2$ by the infrared bound \eqref{eq:IRB}. The desired bound follows readily by optimising in $N$ \eqref{eq:pvol1}, which can be achieved choosing $N:=( n(\log n)^{\mathds{1}_{d=6}} )^{1/4}$.
	
	We now proceed with the proof of the lower bounds. As before, define the random variable
	\begin{equation} 
	\mathcal Z:= | \mathcal C(0)\cap (\Lambda_{2N}\setminus \Lambda_{N-1})|,
	\end{equation} 
	for some integer $N$ to be chosen later (in terms of $n$).
By the Paley--Zygmund inequality,
	\begin{equation}\label{eq:Paley-Zygmund}
		\phi_{\beta_c}\big[\mathcal Z\geq \tfrac{1}{2} \phi_{\beta_c}[\mathcal Z \,|\, \mathcal Z>0]\big]\geq \frac{\phi_{\beta_c}[\mathcal Z]^2}{4\phi_{\beta_c}[\mathcal Z^2]}.
	\end{equation}
	Since $\phi_{\beta_c}[0\leftrightarrow x] \gtrsim (1\vee |x|)^{-d+2}$ by \eqref{eq: lower bound full-space sec 2}, we obtain
	\begin{equation}\label{eq:Z_first-moment}
		\phi_{\beta_c}[\mathcal Z] = \sum_{x\in \Lambda_{2N}\setminus \Lambda_{N-1}} \phi_{\beta_c}[0\leftrightarrow x] \gtrsim N^2.
	\end{equation}
	Since $\phi_{\beta_c}[\mathcal Z>0]=\phi_{\beta_c}[0\leftrightarrow \partial \Lambda_{N+1}]\lesssim (\log N)^{\mathds{1}_{d=6}}/ N^2$ by Theorems~\ref{thm:FKfree d>6} and \ref{thm:FKfree d=6}, we conclude that
	\begin{equation}\label{eq:Z_first-moment_cond}
		\phi_{\beta_c}[\mathcal Z \,|\, \mathcal Z>0] \gtrsim cN^4/ (\log N)^{\mathds{1}_{d=6}} \geq 2n,
	\end{equation}
	where the last inequality is guaranteed by setting $N:=A( n(\log n)^{\mathds{1}_{d=6}} )^{1/4}$, for some large enough constant $A=A(d)>0$.
	Finally, we upper bound the second moment of $\mathcal Z$ by applying Proposition~\ref{prop:tree graph}. The computation is exactly the same as \eqref{eq:plbFK3} and we obtain
	\begin{equation}\label{eq:Z_second-moment}
		\phi_{\beta_c}[\mathcal Z^2]\lesssim N^6.
	\end{equation}
	The desired bound follows readily from \eqref{eq:Paley-Zygmund}--\eqref{eq:Z_second-moment} and the choice of $N$.
\end{proof}

\begin{Rem}\label{rem:vol_lower-bound_finite-vol}
	By using Theorem \ref{thm:CAMIAJIANGNEWMAN}, the proof of the lower bound in \eqref{eq:Z_first-moment} readily extends to the measure $\phi^0_{\Lambda_{MN},\beta_c}$, for the constant $M>1$ given by the theorem.
\end{Rem}

\begin{Rem} An alternative approach to obtain a lower bound on the volume exponent is the following, 
	based on a similar strategy used for Bernoulli percolation in \cite{hutchcroft2022derivation} (see also the appendix of \cite{vEGPSperco}). For sake of simplicity, we only explain the argument for $d>6$.
	The entropic bound of Theorem \ref{Thm: entropic bound pinsker} applied to the event $\mathcal{A} = \{|\mathcal{C}(0)| \geq n\}$ at $\beta' = \beta_n = \beta_c - 1/n^{1/2}$ and $\beta = \beta_c$ gives 
	\begin{equation}
		\phi_{\beta_c}[|\mathcal{C}(0)| \geq n] \leq \phi_{\beta_n}[|\mathcal{C}(0)| \geq n] + C(\beta_c - \beta_n) \sqrt{	\phi_{\beta_c}[|\mathcal{C}(0)| \geq n] \nabla(\beta_c) \chi(\beta_n)}. 
	\end{equation}
	Now using Markov's inequality to first bound $\phi_{\beta_n}[|\mathcal{C}(0)| \geq n] \leq \chi(\beta_n) / n$ and then Proposition \ref{prop:susceptibility} to bound the susceptibility shows
	\begin{equation}
		\phi_{\beta_c}[|\mathcal{C}(0)| \geq n]  \leq \frac{C}{n^{1/2}} + \frac{C}{n^{1/4}} \sqrt{\phi_{\beta_c}[|\mathcal{C}(0)| \geq n] },
	\end{equation}
	which readily implies the result. Note that this strategy only relies on the entropic bound and the near critical blowup of the susceptibility of \cite{AizenmanGeometricAnalysis1982}. 
	(The same strategy can be made to work for $d = 6$, but this needs the comparison of Lemma \ref{lem:FKfree-free-boundary-pushing}, 
	which can be proved with the input from equation \eqref{eq:warm-up-bound-wired}). 
\end{Rem}

\subsection{One-arm exponent below the upper-critical dimension}\label{sec:FKnonMF}
This last subsection is devoted to the proofs of Theorems \ref{thm: FKfree d=5}--\ref{thm: FKfree d=3}. They all rely on a (new) bound on the four-point function of the FK-Ising model presented in Section \ref{sec:boundfourpoint}.
\subsubsection{Bound on the four-point function of the FK-Ising model}\label{sec:boundfourpoint}

We would like to obtain the lower bounds in Theorems \ref{thm: FKfree d=5}--\ref{thm: FKfree d=3} using a second moment method involving the random variable $\mathcal{Z}$ defined in \eqref{eq:def Z}. The main problem is that our only tool to bound the second moment of $\mathcal Z$ is the tree-graph inequality of Proposition \ref{prop:tree graph}. This bound is expected to be sharp (up to a multiplicative constant) in the mean-field regime $d>6$. If we believe that $d_c^{\textup{FK-Ising}}=6$, it will not give a sharp bound in dimensions $d\in \{3,4,5\}$. Indeed, as observed in Proposition \ref{prop: FK-free lower bound}, the use of the tree-graph inequality only gives a lower bound on the one-arm probability of order $n^{-2}$ when $d=5$. A similar phenomenon holds for Bernoulli percolation. There, a sharp control on the three-point function can be achieved using Gladkov's inequality \cite{gladkov2024percolation}. More precisely, Gladkov proved that for every $x,y,z\in \mathbb Z^d$, and any $p\in [0,1]$,
\begin{equation}\label{eq:Gladkov}
	\mathbb P_p[x\connect{}y,z]^2\leq 8\mathbb P_p[x\connect{}y]\mathbb P_p[y\connect{}z]\mathbb P_p[z\connect{}x].
\end{equation}
His proof relies on subtle exploration algorithms as well as on the BK inequality. This makes its extension to the FK-Ising model not straightforward. However, it is reasonable to expect that a version of \eqref{eq:Gladkov} also holds in this context. Let us mention that \eqref{eq:Gladkov} has been efficiently used (and generalised) in the context of long-range percolation below the upper-critical dimension \cite{hutchcroft2025criticalII}.

We will not prove a version of \eqref{eq:Gladkov} for the FK-Ising model, but rather observe that well-known correlation inequalities available at the level of the spin model provide a control on the four-point function which we conjecture to be sharp (up to a multiplicative constant) in dimensions $d<6$. If $x,y,z,t \in \mathbb Z^d$, we let $\{x\connect{}y,z,t\}$ denote the event that $x,y,z,t$ are in the same cluster. 
\begin{Prop}\label{prop:four point} Let $d\geq 2$ and $\beta\geq 0$. Then, for every $x,y,z,t\in \mathbb Z^d$, one has
\begin{equation}
	\phi_{\beta}[x\connect{}y,z,t]\leq \langle \sigma_x\sigma_y\rangle_{\beta}\langle \sigma_z\sigma_t\rangle_{\beta}+\langle \sigma_x\sigma_z\rangle_{\beta}\langle \sigma_y\sigma_t\rangle_{\beta}+\langle \sigma_x\sigma_t\rangle_{\beta}\langle \sigma_y\sigma_z\rangle_{\beta}.
\end{equation}
\end{Prop}
\begin{proof} Set $A=\{x,y,z,t\}$. Observe that
\begin{equation}\label{eq:pd=51}
	\phi_\beta[x\connect{}y,z,t]\leq \phi_{\beta}[\mathcal F_A],
\end{equation}
where $\mathcal F_A$ is the percolation event defined as follows: 
\begin{equation}
\mathcal F_A:=\{ \omega\in \{0,1\}^{E(\mathbb Z^d)}:\textup{ Every cluster of $\omega$ intersects $A$ an even number of times}\}.
\end{equation}
By the Edwards--Sokal coupling (see \cite[(1.7)]{DuminilLecturesOnIsingandPottsModels2019}), one has
\begin{equation}\label{eq:pd=52}
	\phi_{\beta}[\mathcal F_A]=\langle \sigma_x\sigma_y\sigma_z\sigma_t\rangle_{\beta}.
\end{equation}
Finally, using Lebowitz' inequality \cite{Lebowitz1974Inequ}, we get that 
\begin{equation}
\langle \sigma_x\sigma_y\sigma_z\sigma_t\rangle_{\beta}\leq\langle \sigma_x\sigma_y\rangle_{\beta}\langle \sigma_z\sigma_t\rangle_{\beta}+\langle \sigma_x\sigma_z\rangle_{\beta}\langle \sigma_y\sigma_t\rangle_{\beta}+\langle \sigma_x\sigma_t\rangle_{\beta}\langle \sigma_y\sigma_z\rangle_{\beta}.
\end{equation} Plugging this observation in \eqref{eq:pd=52} and then using \eqref{eq:pd=51} concludes the proof.
\end{proof}
\begin{Rem} Let us briefly explain why it is reasonable to expect this inequality to be sharp (up to a multiplicative constant) at criticality in dimensions $d<6$. We used two inequalities in the proof of Proposition \ref{prop:four point}. The first one is \eqref{eq:pd=51}. Building on the discussion below Theorem \ref{thm: FKfree d=3}, it is reasonable to expect that, in dimensions $d<6$, if four points $x,y,z,t$ are paired and at mutual distance $\asymp n$ of each other, then they should lie in the same cluster with positive probability. The second inequality we used is Lebowitz' inequality \cite{Lebowitz1974Inequ}. Once again, for a configuration of points $x,y,z,t$ as before, this gives 
\begin{equation}
\langle \sigma_x\sigma_y\sigma_z\sigma_t\rangle_{\beta_c}\lesssim \max_{u,v\in \{x,y,z,t\}: u\neq v}\langle \sigma_u\sigma_v\rangle_{\beta_c}^2\lesssim\langle \sigma_0\sigma_{n\mathbf{e}_1}\rangle_{\beta_c}^2.
\end{equation}
Thanks to Proposition \ref{prop:secondineqGriffiths}, one has $\langle \sigma_{x}\sigma_y\sigma_z\sigma_t\rangle_{\beta_c}\geq \langle \sigma_{x}\sigma_y\rangle_{\beta_c}\langle \sigma_z\sigma_t\rangle_{\beta_c}\gtrsim \langle\sigma_0\sigma_{n\mathbf{e}_1}\rangle_{\beta_c}^2$. As a result, we did not lose anything by using Lebowitz' inequality. The above reasoning gives a weak justification of why we may hope the use of Proposition \ref{prop:four point} to yield sharp results in dimensions $d\in \{3,4,5\}$.
\end{Rem}

\subsubsection{The case $d=5$}

We are now in a position to prove Theorem \ref{thm: FKfree d=5}.

\begin{proof}[Proof of Theorem \textup{\ref{thm: FKfree d=5}}] Fix $d=5$ and let $n\geq 1$. As above, we let
\begin{equation}
	\mathcal Z:=|\mathcal C(0)\cap \Lambda_{2n}\setminus\Lambda_{n-1}|.
\end{equation}
By H\"older's inequality, one has
\begin{equation}
	\phi_{\beta_c}[\mathcal Z]=\phi_{\beta_c}[\mathcal Z \mathds{1}_{\mathcal Z>0}]\leq \phi_{\beta_c}[\mathcal Z^3]^{1/3}\phi_{\beta_c}[\mathcal Z>0]^{2/3},
\end{equation}
so that
\begin{equation}\label{eq:pFKd=5 1}
	\phi_{\beta_c}[0\connect{}\partial \Lambda_n]\geq \phi_{\beta_c}[\mathcal Z>0]\geq \frac{\phi_{\beta_c}[\mathcal Z]^{3/2}}{\phi_{\beta_c}[\mathcal Z^3]^{1/2}}.
\end{equation}
The exact same computation as in \eqref{eq:plbFK2} gives
\begin{equation}\label{eq:pFKd=5 2}
	\phi_{\beta_c}[\mathcal Z]\gtrsim n^2.
\end{equation}
Moroever, by Proposition \ref{prop:four point},
\begin{align}\notag
	\phi_{\beta_c}[\mathcal Z^3]&=\sum_{y,z,t\in \Lambda_{2n}\setminus \Lambda_{n-1}}\phi_{\beta_c}[0\connect{}y,z,t]
	\\&\leq \sum_{y,z,t\in \Lambda_{2n}\setminus \Lambda_{n-1}}\Big(\langle \sigma_0\sigma_y\rangle_{\beta_c}\langle \sigma_z\sigma_t\rangle_{\beta_c}+(y\Leftrightarrow z)+(y\Leftrightarrow t)\Big).\label{eq:pFKd=5 3}
\end{align}
Using the infrared bound \eqref{eq:IRB}, we obtain
\begin{equation}\label{eq:pFKd=5 4}
	\sum_{y,z,t\in \Lambda_{2n}\setminus\Lambda_{n-1}}\langle \sigma_0\sigma_y\rangle_{\beta_c}\langle \sigma_z\sigma_t\rangle_{\beta_c}\lesssim n^{2-d} \cdot (n^d)^2\cdot n^2=n^{d+4}=n^9.
\end{equation}
This gives $\phi_{\beta_c}[\mathcal Z^3]\lesssim n^9$. Plugging this bound and \eqref{eq:pFKd=5 2} in \eqref{eq:pFKd=5 1} yields
\begin{equation}
	\phi_{\beta_c}[0\connect{}\partial \Lambda_n]\gtrsim \frac{(n^2)^{3/2}}{(n^9)^{1/2}}=\frac{1}{n^{3/2}},
\end{equation}
concluding the proof.
\end{proof}
\subsubsection{The case $d=4$}
We now turn to the case $d=4$. 
\begin{proof}[Proof of Theorem \textup{\ref{thm: FKfree d=4}}] Let $d=4$ and $n\geq 1$. First, observe that by Markov's property and stochastic domination (see Propositions \ref{prop:DMP} and \ref{prop:stoch dom FK})
\begin{equation}
	\phi_{\beta_c}[0\connect{}\partial\Lambda_n]=\phi_{\beta_c}\big[\phi^\xi_{\Lambda_n,\beta_c}[0\connect{}\partial \Lambda_n]\big]\leq \phi^1_{\Lambda_n,\beta_c}[0\connect{}\partial \Lambda_n]\lesssim \frac{(\log n)^{5/2}}{n},
\end{equation}
where we used the upper bound of Theorem \ref{thm:FKwired d=4} in the last inequality. 

For the reversed inequality, we proceed as in the proof of Theorem \ref{thm: FKfree d=5}. Again, let 
\begin{equation}
	\mathcal Z:=|\mathcal C(0)\cap \Lambda_{2n}\setminus\Lambda_{n-1}|.
\end{equation}
As a consequence of H\"older's inequality (like in \eqref{eq:pFKd=5 1}), one has 
\begin{equation}\label{eq:pFKd=4 1}
	\phi_{\beta_c}[0\connect{}\partial \Lambda_n]\geq \frac{\phi_{\beta_c}[\mathcal Z]^{3/2}}{\phi_{\beta_c}[\mathcal Z^3]^{1/2}}.
\end{equation}
The same computation as in \eqref{eq:pFKd=5 4} gives
\begin{equation}\label{eq:pFKd=4 2}
	\phi_{\beta_c}[\mathcal Z^3]\lesssim n^{d+4}=n^8.
\end{equation}
Using \eqref{eq: lb d=4 sec 2} we find that
\begin{equation}\label{eq:pFKd=4 3}
	\phi_{\beta_c}[\mathcal Z]=\sum_{x\in \Lambda_{2n}\setminus\Lambda_{n-1}}\langle \sigma_0\sigma_x\rangle_{\beta_c}\gtrsim \frac{n^4}{n^{2}\log n}=\frac{n^2}{\log n}.
\end{equation}
Plugging \eqref{eq:pFKd=4 2} and \eqref{eq:pFKd=4 3} in \eqref{eq:pFKd=4 1} yields
\begin{equation}
	\phi_{\beta_c}[0\connect{}\partial \Lambda_n]\gtrsim \frac{(n^{2}/\log n)^{3/2}}{(n^8)^{1/2}}=\frac{1}{n(\log n)^{3/2}}.
\end{equation}
This concludes the proof.

\end{proof}

\begin{Rem} It is possible to obtain a slightly worse lower bound in $d=4$ using the following method. If $k\geq 1$ and $x\in \mathbb Z^d$, we let $\Lambda_k(x)=\Lambda_k+x$. Write
\begin{align}
	\phi_{\beta_c}[0\connect{}2n\mathbf{e}_1]&\leq \phi_{\beta_c}[\mathds{1}_{0\connect{}\partial \Lambda_n}\mathds{1}_{2n\mathbf{e}_1\connect{}\partial \Lambda_{n}(2n\mathbf{e}_1)}]\notag
	\\&= \phi_{\beta_c}\big[\mathds{1}_{0\connect{}\partial \Lambda_n}\phi_{\Lambda_n(2n\mathbf{e}_1),\beta_c}^\xi[2n\mathbf{e}_1\connect{}\partial \Lambda_n(2n\mathbf{e}_1)]\big]\notag
	\\&\lesssim \frac{(\log n)^{5/2}}{n}\phi_{\beta_c}[0\connect{}\partial \Lambda_n], \label{eq:pFKd=4 1 bis} 
\end{align}
where in the second line we used Markov's property (Proposition \ref{prop:DMP}), and in the third line we used Proposition \ref{prop:stoch dom FK} and Theorem \ref{thm:FKwired d=4}.
Now, by \eqref{eq: lb d=4 sec 2} and \eqref{eq: consequences of ESC}, one has $\phi_{\beta_c}[0\connect{}2n\mathbf{e}_1]\gtrsim \frac{1}{n^2\log n}$. Plugging this bound in \eqref{eq:pFKd=4 1 bis} concludes the proof. 
\end{Rem}
\subsubsection{The case $d=3$}
We conclude this section with the case $d=3$. As explained in the introduction, this result is conditioned on the existence of $\eta$ defined in \eqref{eq:def eta}. Under this assumption, \cite[Theorem~1.5]{DuminilPanis2024newLB} gives that: for every $n\geq 1$,
\begin{equation}\label{eq:input d=3}
	\langle \sigma_0\sigma_x\rangle_{\beta_c}\gtrsim \frac{1}{n^{3/2+o(1)}},
\end{equation}
where $o(1)$ tends to $0$ as $n$ tends to infinity. 
\begin{proof}[Proof of Theorem \textup{\ref{thm: FKfree d=3}}] Again, we follow the exact same strategy of proof as Theorem \ref{thm: FKfree d=5}. Let $d=3$ and $n\geq 1$. Letting $\mathcal Z:=|\mathcal C(0)\cap \Lambda_{2n}\setminus\Lambda_{n-1}|$, one has (by \eqref{eq:pFKd=5 1}) 
\begin{equation}\label{eq:pFKd=3 1}
	\phi_{\beta_c}[0\connect{}\partial \Lambda_n]\geq \frac{\phi_{\beta_c}[\mathcal Z]^{3/2}}{\phi_{\beta_c}[\mathcal Z^3]^{1/2}}.
\end{equation}
As in \eqref{eq:pFKd=5 4}, we find that
\begin{equation}\label{eq:pFKd=3 2}
	\phi_{\beta_c}[\mathcal Z^3]\lesssim n^{d+4}=n^7.
\end{equation}
Moreover, using \eqref{eq:input d=3} we obtain
\begin{equation}\label{eq:pFKd=3 3}
	\phi_{\beta_c}[\mathcal Z]=\sum_{x\in \Lambda_{2n}\setminus\Lambda_{n-1}}\langle \sigma_0\sigma_x\rangle_{\beta_c}\gtrsim \frac{n^3}{n^{3/2+o(1)}}=n^{3/2+o(1)}.
\end{equation}
Combining \eqref{eq:pFKd=3 2} and \eqref{eq:pFKd=3 3} in \eqref{eq:pFKd=3 1} yields
\begin{equation}
	\phi_{\beta_c}[0\connect{}\partial \Lambda_n]\gtrsim \frac{(n^{3/2+o(1)})^{3/2}}{(n^7)^{1/2}}=\frac{1}{n^{5/4+o(1)}},
\end{equation}
which concludes the proof.
\end{proof}

\section{The random current representation of the Ising model}\label{sec:RCR}

This section presents the main geometric 
tool of the paper: the random current representation together with the closely 
related \emph{backbone} expansion. This representation was popularised by Aizenman \cite{AizenmanGeometricAnalysis1982} and has been the subject of many subsequent works, see for instance \cite{AizenmanGrahamRenormalizedCouplingSusceptibility4d1983,AizenmanFernandezCriticalBehaviorMagnetization1986,AizenmanBarskyFernandezSharpnessIsing1987,AizenmanDuminilSidoraviciusContinuityIsing2015,DuminilTassionNewProofSharpness2016,AizenmanDuminilTassionWarzelEmergentPlanarity2019,AizenmanDuminilTriviality2021,PanisTriviality2023,KrachunPanagiotisPanisScalinglimit2023,chen2023conformal,PanisIIC2024,DuminilPanis2024newLB,Liu2023Plateau,duminil2025conformal}. 

The first two subsections review standard material, and we refer to \cite{DuminilLecturesOnIsingandPottsModels2019,Panis2024applications} for a more detailed introduction.
In Section \ref{sec:RCRcorineq}, we prove Lemmas \ref{Lem: bubble bound with C} and \ref{lem: random current proposition}, which were used earlier in Sections 
\ref{sec:FKwired} and \ref{sec:FKfree}.
 Finally, in Section \ref{sec:RCRFK}, we describe how random currents can be used to analyse the FK-Ising model. In particular, we prove two new fundamental inequalities: an analogue of the van den Berg--Kesten (BK) inequality (Proposition \ref{prop:BK type for FK}), together with a \emph{tree graph inequality} (Proposition \ref{prop:tree graph}).

\subsection{Random currents}

Let $G=(V,E)$ be a finite graph. We let $\mathfrak{g}$ denote the \emph{ghost} vertex and define $G_{\mathfrak{g}}$ to be the graph of vertex set $V_\mathfrak{g}:=V\cup \{\mathfrak{g}\}$ and edge set $E_{\mathfrak{g}}:=E\cup \{x\mathfrak{g}: x\in V\}$.
\begin{Def} A \textit{current} $\n$ on $G_\fg$ is a function defined on the edge set $E_\fg$ and taking its values in $\mathbb N=\lbrace 0,1,\ldots\rbrace$. We let $\Omega_{G_\fg}$ be the set of currents on $G_\fg$. The set of \emph{sources} of $\n$, denoted by $\sn$, is defined as
\begin{equation}
\sn:=\Big\{ x \in V_\fg \: : \: \sum_{y\in V_\fg:\: xy\in E_\fg}\n_{xy}\textup{ is odd}\Big\}.
\end{equation}
We also define, for every $\beta\geq 0$ and $\mathsf{h}\in (\mathbb R^+)^V$, 
\begin{equation}
w_{\beta,\mathsf{h}}^G(\n):=\prod_{\substack{xy\in E}}\dfrac{\beta^{\n_{xy}}}{\n_{xy}!}\prod_{x\in V}\frac{\mathsf{h}_x^{\n_{x\fg}}}{\n_{x\fg}!}.
\end{equation}
If $\mathsf{h}=0$, we simply write $w_{\beta}^G(\n)=w_{\beta,0}^G(\n)$.  Similarly, we define $\Omega_G$, the set of currents on $G$. 
\end{Def}

It is classical (see \cite{AizenmanGeometricAnalysis1982,DuminilLecturesOnIsingandPottsModels2019}) that the correlation functions of the Ising model are related to currents: if we set $\sigma_{\fg}=1$, then, for every $A\subset V_\fg$, one has,
\begin{equation}\label{eq: relation partition functions}
	\sum_{\sigma\in \{-1,1\}^V}\sigma_A \exp\Big(\beta\sum_{xy\in E}\sigma_x\sigma_y+\sum_{x\in V}\mathsf{h}_x\sigma_x\Big)=2^{|V|}\sum_{\n\in \Omega_\fg:\: \sn=A}w^G_{\beta,\mathsf{h}}(\n),
\end{equation}
so that
\begin{equation}\label{equation correlation rcr}
\left\langle \sigma_A\right\rangle_{G,\beta,\mathsf{h}}=\dfrac{\sum_{\n\in \Omega_{G_\fg}:\:\sn=A}w_{\beta,\mathsf{h}}^G(\n)}{\sum_{\n\in \Omega_{G_g}:\:\sn=\emptyset}w_{\beta,\mathsf{h}}^G(\n)}. 
\end{equation}

The trace of a current $\n$ naturally induces a percolation configuration $(\mathds{1}_{\n_{xy}>0})_{xy \in E_\fg}$ on $G_\fg$. A current $\n$ with source set $\{x,y\}$ induces a percolation configuration which is described by a path between $x$ and $y$ together with a collection of loops, see Figure \ref{fig:acurrent} for an illustration.
The connectivity properties of this percolation model play a determinant role in the analysis of the Ising model. This justifies the following definitions.

\begin{Def} Let $\n\in \Omega_{G_\fg}$ and $x,y\in V_\fg$.
\begin{enumerate}
    \item[$(i)$] We say that $x$ is \emph{connected} to $y$ in $\n$ and write $x\overset{\n}{\longleftrightarrow} y$, if there is a sequence of points $x_0=x,x_1,\ldots, x_\ell=y$ such that $\n_{x_ix_{i+1}}>0$ for $0\leq i \leq \ell-1$.
    \item[$(ii)$] The \emph{cluster} of $x$ in $\n$, denoted by $\mathbf{C}_\n(x)$, is the set of points connected to $x$ in $\n$.
\end{enumerate}
\end{Def}

At this stage, the above definitions lack a probabilistic interpretation. In particular, equation \eqref{equation correlation rcr} does not admit one, since the sums in the numerator and denominator range over distinct sets. The true advantage of the current expansion emerges through a simple yet powerful combinatorial tool: the \emph{switching lemma}. This identity was first derived in \cite{GriffithsHurstShermanConcavity1970}, but its deep connections to percolation theory flourished in the seminal work of Aizenman \cite{AizenmanGeometricAnalysis1982}. 

Below, we use the fact that if $G=(V,E)$ and $H=(V',E')$ are two graphs satisfying $G\subset H$, and if $\n\in \Omega_G$, then $\n$ naturally extends to an element of $\Omega_H$ (that we still denote $\n$) by setting, for each $e\in E'\setminus E$, $\n_e=0$. Additionally, if $\n\in \Omega_H$, we let $\n_{|G}$ be its restriction to $E$.
\begin{Lem}[Switching lemma]\label{lem:switching} Let $G=(V,E)$ and $H=(V',E')$ be two finite graphs such that $G\subset H$. Let $\beta\geq 0$, and $\mathsf{h}\in (\mathbb R^+)^{V'}$. For every $A\subset V$, $B\subset V'$, and every $F:\Omega_{H}\rightarrow \mathbb R$, one has
\begin{multline}\label{eq: switching lemma}
\sum_{\substack{\n_1\in \Omega_G : \:\sn_1=A\\ \n_2\in \Omega_H :\: \sn_2=B}}F(\n_1+\n_2)w_{\beta,\mathsf{h}}^G(\n_1)w_{\beta,\mathsf{h}}^H(\n_2)\\=\sum_{\substack{\n_1\in \Omega_G: \:\sn_1=\emptyset\\ \n_2\in \Omega_H :\: \sn_2=B\Delta A}}F(\n_1+\n_2)w_{\beta,\mathsf{h}}^G(\n_1)w_{\beta,\mathsf{h}}^H(\n_2)\mathds{1}_{(\n_1+\n_2)_{|G}\in \mathcal{F}_{A}^G},
\end{multline}
where $B\Delta A=(B\cup A)\setminus (B\cap A)$ is the symmetric difference of sets, and $\mathcal{F}_{A}^G$ is given by
\begin{equation}\label{eq: event F_A}
\mathcal{F}_{A}^G=\lbrace \n \in \Omega_{G_\fg}: \: \exists \mathbf{m}\leq \n \:, \: \partial \mathbf{m}=A\rbrace.
\end{equation}
\end{Lem}

\begin{Rem}\label{rem: fs is a perco event} Observe that the event $\mathcal F_{A}^G$ introduced in \eqref{eq: event F_A} is a percolation event. Indeed, it is exactly the event that each connected component of the percolation configuration induced by $\n\in \Omega_{G_\fg}$ intersects $A$ an even number of times. In particular, if $A=\{x,y\}$ then $\mathcal F_A^G:=\{x\connect{\n\:}y\}$.
\end{Rem}

Before presenting applications of Lemma \ref{lem:switching}, we introduce the random current measures of interest. Let $G=(V,E)$, $H=(V',E')$, and $\beta,\mathsf{h}$ be as in the statement of Lemma \ref{lem:switching}. If $A\subset V$, we define a probability measure $\mathbf{P}_{G,\beta,\mathsf{h}}^A$ on $\Omega_G$ as follows: for every $\n\in \Omega_G$, 
\begin{equation}
    \mathbf{P}_{G,\beta,\mathsf{h}}^A[\n]:=\mathds{1}_{\sn=A}\frac{w_{\beta,\mathsf{h}}^G(\n)}{Z_{G,\beta,\mathsf{h}}^A},
\end{equation}
where $Z_{G,\beta,\mathsf{h}}^A:=\sum_{\n\in \Omega_{G_\fg}:\:\sn=A}w_{\beta,\mathsf{h}}^G(\n)$ is a normalisation constant. When $A=\lbrace x,y\rbrace$, we write $\mathbf{P}^{xy}_{G,\beta,\mathsf{h}}$ instead of $\mathbf{P}^{\{x,y\}}_{G,\beta,\mathsf{h}}$. Moreover, for $B\subset V'$, define
\begin{equation}
    \mathbf{P}_{G,H,\beta,\mathsf{h}}^{A,B}:=\mathbf{P}_{G,\beta,\mathsf{h}}^{A}\otimes\mathbf{P}_{H,\beta,\mathsf{h}}^{B}.
\end{equation}
If $G=H$ we write $\mathbf P^{A,B}_{G,\beta,\mathsf{h}}=\mathbf P^{A,B}_{G,G,\beta,\mathsf{h}}$. When $\mathsf{h}=0$, we drop the subscript $\mathsf{h}$ in the above notations. 

Focusing on the case where $\mathsf{h}=0$, it was proven in \cite{AizenmanDuminilSidoraviciusContinuityIsing2015} that if $G=(V,E)$ is an infinite (locally finite) graph and $A\subset V$ is a finite even set, it is possible to construct the infinite volume measure $\mathbf P^{A}_{G,\beta}$ as a weak limit of measures of the above type. In the case $G=\mathbb Z^d$ and $A\subset \mathbb Z^d$ (finite), the associated infinite volume measure is denoted by $\mathbf{P}_{\beta}^A$.

We now list two useful applications of Lemma \ref{lem:switching}. We let $G=(V,E)$ and $H=(V',E')$ be two graphs with $G\subset H$, and $\beta\geq 0$.
\begin{enumerate}
	\item[$\bullet$] If $x,y\in V$,
	\begin{equation}\label{eq:applicationswitch1}
		\langle \sigma_x\sigma_y\rangle_{G,\beta}^2=\mathbf P^{\emptyset,\emptyset}_{G,\beta}[x\connect{\n_1+\n_2\:}y].
	\end{equation}
	\item[$\bullet$] If $x,y,u\in V$,
	\begin{equation}\label{eq:applicationswitch2}
		\frac{\langle \sigma_x\sigma_u\rangle_{G,\beta}\langle \sigma_u\sigma_y\rangle_{H,\beta}}{\langle \sigma_x\sigma_y\rangle_{H,\beta}}=\mathbf P^{xy,\emptyset}_{H,G,\beta}[x\connect{(\n_1+\n_2)_{|G}\:}u].
	\end{equation}
	In particular, taking $u=y$ yields
	\begin{equation}\label{eq:applicationswitch3}
		\frac{\langle \sigma_x\sigma_y\rangle_{G,\beta}}{\langle \sigma_x\sigma_y\rangle_{H,\beta}}=\mathbf P^{xy,\emptyset}_{H,G,\beta}[x\connect{(\n_1+\n_2)_{|G}\:}y].
	\end{equation}
\end{enumerate}

\begin{Rem}\label{rem:extension currents to J} Let $G=(V,E)$ be a finite graph. Following \eqref{eq:defIsing}, all the definitions/properties above can be adapted to the case where $\beta$ is replaced by a general interaction $J = (J_{xy})_{xy \in E}$. 
In this case, for $\n\in \Omega_{G_\fg}$, we let
\begin{equation}
w_{J,\mathsf{h}}^G(\n):=\prod_{\substack{xy\in E}}\dfrac{J_{xy}^{\n_{xy}}}{\n_{xy}!}\prod_{x\in V}\frac{\mathsf{h}_x^{\n_{x\fg}}}{\n_{x\fg}!},
\end{equation}
and define similarly $\mathbf P^A_{G,J,\mathsf{h}}$.
\end{Rem}

\subsection{Backbone expansion}
As mentioned above, if $\n$ is a current with sources $\sn=\{x,y\}$, then $x$ and $y$ are connected in $\n$. The \emph{backbone} expansion of the Ising model---introduced in \cite{AizenmanGeometricAnalysis1982}---provides a way to explore one such path. We refer to \cite[Chapter~2]{Panis2024applications} for a more detailed introduction.

We fix $G=(V,E)$ a finite graph and $\prec$ an arbitrary ordering of $E$. We construct an ordering of $E_\fg$ from $\prec$---that we still denote $\prec$---by requiring that for each $x,y\in V$ with $y\sim x$, one has $xy\prec x\fg$. In words, the ordering prioritises the edges which are not going to $\mathfrak g$.

\begin{Def}[Backbone exploration] Let $x,y\in V_\fg$. Let $\n\in \Omega_{G_{\fg}}$ with $\sn=\{x,y\}$. The backbone of $\n$, denoted by $\Gamma(\n)$, is the unique oriented and edge self-avoiding path from $x$ to $y$, supported on edges $uv$ with $\n_{uv}$ odd, and which is minimal for $\prec$. The backbone $\Gamma(\n)=\{x_ix_{i+1}:  0\leq i < k\}$ is obtained via the following procedure:
\begin{enumerate}
    \item[$(i)$] Set $x_0=x$. The first edge $x_0x_1$ of $\Gamma(\n)$ is the earliest of all the edges $e$ emerging from $x_0$ for which $\n_e$ is odd.
    All edges $e$ containing $x_0$ and satisfying $e\prec x_0x_1$ and $e\neq x_0x_1$ are called \emph{explored} and have $\n_e$ even.
    \item[$(ii)$] Each edge $x_ix_{i+1}$ with $i\ge 1$ is the earliest of all edges $e$ emerging from $x_i$ that have not been explored previously, and for which $\n_e$ is odd.
    \item[$(iii)$] The exploration stops when it reaches $y=x_k$ for the first time.
\end{enumerate}
Let $\overline{\Gamma}(\n)$ denote the set of explored edges, i.e. $\Gamma(\n)$ together with all explored even edges.
\end{Def}

\begin{figure}[htb]
	\begin{center}
		\includegraphics[scale=1.1]{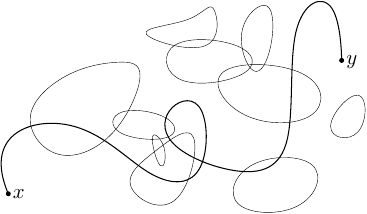}
		\caption{An illustration of the trace of a current $\n$ with $\sn=\{x,y\}$. The bold line represents the backbone $\Gamma(\n)$ of $\n$.}
		\label{fig:acurrent}
	\end{center}
\end{figure}

By definition, the set $\overline{\Gamma}(\n)$ is determined by $\Gamma(\n)$ and $\prec$. Similarly, $\Gamma(\n)$ can be deduced from $\overline{\Gamma}(\n)$ by choosing the odd edges according to $\prec$. The current $\n\setminus\overline{\Gamma}(\n)$, which we define to be the restriction of $\n$ to the complement of $\overline{\Gamma}(\n)$, is sourceless. Finally, if $\beta\geq 0$ and $\mathsf{h}\in (\mathbb R^+)^V$, we can write
\begin{equation}
	\langle \sigma_x\sigma_y\rangle_{G,\beta,\mathsf{h}}=\sum_{\gamma: x\rightarrow y}\rho_{G_\fg,\beta,\mathsf{h}}(\gamma),
\end{equation}
where for a path $\gamma$ with endpoints $\partial \gamma$, 
\begin{equation}
	\rho_{G_\fg,\beta,\mathsf{h}}(\gamma)=\langle \sigma_{\partial \gamma}\rangle_{G,\beta,\mathsf{h}}\mathbf P^{\partial \gamma}_{G,\beta,\mathsf{h}}[\Gamma(\n)=\gamma].
\end{equation}
When $\mathsf{h}=0$, we write $\rho_{G,\beta}(\gamma)=\rho_{G_\fg,\beta,0}(\gamma)$.

We now collect various useful properties of the backbone expansion. For proofs, we refer to \cite{AizenmanFernandezCriticalBehaviorMagnetization1986} or \cite[Chapter~2]{Panis2024applications}. Recall that we fix $G=(V,E)$ a finite graph. A path $\gamma$ (viewed as a sequence of edges) on $G_\fg$ is said to be \emph{consistent} if no edge of the sequence uses an edge cancelled by a previous step.

\begin{enumerate}
	\item[$\bullet$] If $\gamma$ is a consistent path and $\mathsf E$ is a subset of $E_\fg$ satisfying $\gamma\cap \mathsf E=\emptyset$, then
	\begin{equation}\label{eq:propbackb1}
		\rho_{G_\fg,\beta,\mathsf{h}}(\gamma)\leq \rho_{G_\fg\setminus\mathsf E,\beta,\mathsf{h}}(\gamma),
	\end{equation}
	where we abused notations and wrote $G_\fg\setminus\mathsf E$ for the graph of vertex set $V_\fg$ and edge set $E_\fg\setminus \mathsf E$.
	\item[$\bullet$] If $\gamma$ is a consistent path that is the concatenation of two consistent paths $\gamma_1$ and $\gamma_2$ (we write $\gamma=\gamma_1\circ \gamma_2)$, then
	\begin{equation}\label{eq:propbackb2}
		\rho_{G_\fg,\beta,\mathsf{h}}(\gamma)=\rho_{G_\fg,\beta,\mathsf{h}}(\gamma_1)\rho_{G_\fg\setminus \overline{\gamma_1},\beta,\mathsf{h}}(\gamma_2).
	\end{equation}
	\item[$\bullet$] If $\gamma\subset G$ is a consistent path,
	\begin{equation}\label{eq:propbackb3}
		\rho_{G_\fg,\beta,\mathsf{h}}(\gamma)\leq \beta^{|\gamma|},
	\end{equation}
	where $|\gamma|$ is the number of edges visited by $\gamma$.
	\item[$\bullet$] (Chain rule) Let $x,y,u\in V$. Then, 
	\begin{equation}\label{eq:chainrule}
		\sum_{\gamma:x\rightarrow y}\rho_{G_\fg,\beta,\mathsf{h}}(\gamma)\mathds{1}_{u\in \gamma}\leq \langle \sigma_x\sigma_u\rangle_{G,\beta,\mathsf{h}}\langle \sigma_u\sigma_y\rangle_{G,\beta,\mathsf{h}}.
	\end{equation}
\end{enumerate}
Again, \cite{AizenmanDuminilSidoraviciusContinuityIsing2015} allows to define infinite volume versions of the above. If $G=\mathbb Z^d$, we let $\rho_{\beta,\mathsf{h}}=\rho_{\mathbb Z^d_\fg,\beta,\mathsf{h}}$.

We first state a very simple lemma. We will need the following notation: if $G=(V,E)$ is a graph, $A\subset V$, and $\mathcal E\subset \Omega_G$, then
\begin{equation}
	Z^A_{G,\beta}[\mathcal E]=\sum_{\n\in \Omega_{G}: \: \sn=A}w^{G}_{\beta}(\n)\mathds{1}_{\n\in \mathcal E}.
\end{equation}
\begin{Lem}\label{lem:decompbackbone} Let $G=(V,E)$ be a finite graph and $\beta\geq 0$. Let $x,y\in V$ and $\gamma:x\rightarrow y$ such that $\rho_{G,\beta}(\gamma)>0$. Then, if $\mathcal E$ is measurable in terms of $\n_{|\overline{\gamma}^c}$,
\begin{equation}
	Z^{xy}_{G,\beta}[\Gamma(\n)=\gamma, \mathcal E]=Z^{xy}_{G,\beta}[\Gamma(\n)=\gamma] \mathbf P^\emptyset_{\overline{\gamma}^c,\beta}[\mathcal E].
\end{equation}
\end{Lem}
\begin{proof} Write $\n=\n_1+\n_2$ where $\n_1\in \Omega_{\overline{\gamma}}$ and $\n_2\in \Omega_{\overline{\gamma}^c}$. Observe that $w^G_{\beta}(\n)=w^{\overline{\gamma}}_\beta(\n_1)w^{\overline{\gamma}^c}_\beta(\n_2)$. Now, one has
\begin{align}\notag
	Z^{xy}_{G,\beta}[\Gamma(\n)=\gamma, \mathcal E]&=\sum_{\substack{\n_1\Omega_{\overline{\gamma}}: \: \sn_1=\{x,y\}\\\n_2\in \Omega_{\overline{\gamma}^c}:\: \sn_2=\emptyset}}w^{\overline{\gamma}}_\beta(\n_1)\mathds{1}_{\Gamma(\n_1)=\gamma}w^{\overline{\gamma}^c}_\beta(\n_2)\mathds{1}_{\n_2\in \mathcal E}
	\\&=Z^{xy}_{\overline{\gamma},\beta}[\Gamma(\n)=\gamma]Z^{\emptyset}_{\overline{\gamma}^c,\beta}[\mathcal E]\notag
	\\&=Z^{xy}_{G,\beta}[\Gamma(\n)=\gamma] \mathbf P^\emptyset_{\overline{\gamma}^c,\beta}[\mathcal E],
\end{align}
where in the first line we used that the event $\{\Gamma(\n)=\gamma\}$ is measurable in terms of $\n_1$, and in the last line we used that $Z^{xy}_{\overline{\gamma},\beta}[\Gamma(\n)=\gamma]Z^\emptyset_{\overline{\gamma}^c,\beta}=Z^{xy}_{G,\beta}[\Gamma(\n)=\gamma]$. This concludes the proof.
\end{proof}

We now turn to a consequence of \eqref{eq:propbackb1}. The main idea of the proof is to observe that increasing the external magnetic field can be viewed as enlarging the base graph. 
\begin{Lem}\label{lem:monobackbone} Let $G=(V,E)$ be a finite graph and $\beta\geq 0$. If $\mathsf{h},\mathsf{h}'\in (\mathbb R^+)^V$ satisfy $\mathsf{h}\leq \mathsf{h}'$ and $v\in V$, then
\begin{equation}
\langle \sigma_v\rangle_{G,\beta,\mathsf{h}'}-\langle \sigma_v\rangle_{G,\beta,\mathsf{h}}\leq \sum_{\gamma:0\rightarrow \fg}\rho_{G,\beta,\mathsf{h}'}(\gamma)\mathds{1}_{\gamma \cap \{x\fg: \mathsf{h}_x'>\mathsf{h}_x\}\neq \emptyset}.
\end{equation}
\end{Lem}
\begin{proof} 
 
Let us define a new graph $G_\fg'=(V_\fg',E_\fg')$ as follows: $G_\fg'$ is the graph of vertex set $V_\fg'=V_\fg$ and of edge set $E_\fg'=E_\fg \cup E_0$, where $E_0$ is a copy of $E_\fg\setminus E$. Note that $G_\fg\subset G_\fg'$. At this stage, it is useful to explain how the elements of $E_\fg'$ are ordered. We make the following choice: we extend the ordering $\prec$ on $E_\fg$ to $E_\fg'$ by requiring that for every $e\in E_\fg\setminus E$ and its copy $e'\in E_0$, one has $e\prec e'$. This means that when the edges coming from $x\in V$ are explored, the one lying in $E_0$ is explored last. Now, observe that
\begin{equation}
	\langle \sigma_v\rangle_{G,\beta,\mathsf{h}'}=\frac{\sum_{\n\in \Omega_{G_\fg'}:\sn=\{v,\fg\}}w^G_{\beta,\mathsf{h},\mathsf{h}'-\mathsf{h}}(\n)}{\sum_{\n\in \Omega_{G_\fg'}:\sn=\emptyset}w^G_{\beta,\mathsf{h},\mathsf{h}'-\mathsf{h}}(\n)},
\end{equation}
where 
\begin{equation}
	w^G_{\beta,\mathsf{h},\mathsf{h}'-\mathsf{h}}(\n)=\prod_{\substack{xy\in E}}\dfrac{\beta^{\n_{xy}}}{\n_{xy}!}\prod_{e=x\fg\in E_\fg\setminus E}\frac{\mathsf{h}_x^{\n_{x\fg}}}{\n_{x\fg}!}\prod_{e=x\fg\in E_0}\frac{(\mathsf{h}_x'-\mathsf{h}_x)^{\n_{x\fg}}}{\n_{x\fg}!}.
\end{equation}
Additionally, let us write for $A\subset V_\fg'$, $Z^A_{G,\beta,\mathsf{h},\mathsf{h}'-\mathsf{h}}=\sum_{\n\in \Omega_{G_\fg'}:\:\sn=A}w^{G}_{\beta,\mathsf{h},\mathsf{h}'-\mathsf{h}}(\n)$.
From these weights, we can can construct a backbone expansion of $\langle \sigma_v\rangle_{G,\beta,\mathsf{h}'}$ on $G_\fg'$ which involves weights that we denote by $\rho_{G_\fg',\beta,\mathsf{h},\mathsf{h}'-\mathsf{h}}$. Namely,
\begin{equation}
	\langle \sigma_v\rangle_{G,\beta,\mathsf{h}'}=\sum_{\gamma:v\rightarrow \fg\subset G_\fg'}\rho_{G_\fg',\beta,\mathsf{h},\mathsf{h}'-\mathsf{h}}(\gamma).
\end{equation}
Now, by  \eqref{eq:propbackb1} (which adapts to this slightly more general setting), if $\gamma:v\rightarrow \fg\subset G_\fg$ (i.e.\ $\gamma\cap E_0=\emptyset$) is a consistent path,  
\begin{equation}\label{claim:propbackbone}
	\rho_{G_\fg,\beta,\mathsf{h}}(\gamma)\geq \rho_{G_\fg',\beta,\mathsf{h},\mathsf{h}'-\mathsf{h}}(\gamma).
	\end{equation}
	With this result in hand, we write
\begin{align}
	\langle \sigma_v\rangle_{G,\beta,\mathsf{h}'}-\langle \sigma_v\rangle_{G,\beta,\mathsf{h}}&=\sum_{\gamma:v\rightarrow \fg\subset G_\fg'}\Big(\rho_{G_\fg',\beta,\mathsf{h},\mathsf{h}-\mathsf{h}'}(\gamma)-\rho_{G_\fg,\beta,\mathsf{h}}(\gamma)\mathds{1}_{\gamma\cap E_0=\emptyset}\Big)\notag
	\\&\leq \sum_{\gamma:v\rightarrow \fg\subset G_\fg'}\Big(\rho_{G_\fg',\beta,\mathsf{h},\mathsf{h}'-\mathsf{h}}(\gamma)-\rho_{G_\fg',\beta,\mathsf{h},\mathsf{h}'-\mathsf{h}}(\gamma)\mathds{1}_{\gamma \cap E_0=\emptyset}\Big)\notag
	\\&=\sum_{\gamma:v\rightarrow \fg\subset G_\fg'}\mathds{1}_{\gamma \cap E_0\neq \emptyset}\rho_{G_\fg',\beta,\mathsf{h},\mathsf{h}'-\mathsf{h}}(\gamma),\label{eq:bbproof1}
\end{align}
where in the second line we used \eqref{claim:propbackbone}. 
To conclude the proof, it remains to observe that
\begin{equation}
	\sum_{\gamma:v\rightarrow \fg\subset G_\fg'}\mathds{1}_{\gamma \cap E_0\neq \emptyset}\rho_{G_\fg',\beta,\mathsf{h},\mathsf{h}'-\mathsf{h}}(\gamma)\leq \sum_{\gamma:v \rightarrow \fg\subset G_\fg}\rho_{G_\fg,\beta,\mathsf{h}'}(\gamma)\mathds{1}_{\gamma \cap \{x\fg: \mathsf{h}_x'>\mathsf{h}_x\}\neq \emptyset}.
\end{equation}
For this, observe that $Z^\emptyset_{G,\beta,\mathsf{h},\mathsf{h}'-\mathsf{h}}=Z_{G,\beta,\mathsf{h}'}^\emptyset$, so that the above inequality is equivalent to
\begin{equation}
	Z^{v\fg}_{G,\beta,\mathsf{h},\mathsf{h}'-\mathsf{h}}[\Gamma(\n)\cap E_0\neq \emptyset]\leq Z^{v\fg}_{G,\beta,\mathsf{h}'}[\Gamma(\n)\cap \{x\fg: \mathsf{h}_x'>\mathsf{h}_x\}\neq \emptyset].
\end{equation}
To obtain the above, introduce the natural projection $\pi:\gamma\subset G_\fg'\mapsto \pi(\gamma)\subset G_\fg$, and $E_0':=\{e=x\fg\in E_0: \mathsf{h}'_x>\mathsf{h}_x\}$. Then, one has
\begin{multline}
	Z^{v\fg}_{G,\beta,\mathsf{h},\mathsf{h}'-\mathsf{h}}[\Gamma(\n)\cap E_0\neq \emptyset]=Z^{v\fg}_{G,\beta,\mathsf{h},\mathsf{h}'-\mathsf{h}}[\Gamma(\n)\cap E_0'\neq \emptyset]\\\leq Z^{v\fg}_{G,\beta,\mathsf{h},\mathsf{h}'-\mathsf{h}}[\pi(\Gamma(\n))\cap \{x\fg: \mathsf{h}_x'>\mathsf{h}_x\}\neq \emptyset]=Z^{v\fg}_{G,\beta,\mathsf{h}'}[\Gamma(\n)\cap \{x\fg: \mathsf{h}_x'>\mathsf{h}_x\}\neq \emptyset],
\end{multline} 
from which the proof follows readily.
\end{proof}

\subsection{Correlation inequalities}\label{sec:RCRcorineq}

In this section, we prove Lemmas \ref{Lem: bubble bound with C} and \ref{lem: random current proposition}. To ease the reading, we restate these results below. We begin with the proof of Lemma \ref{Lem: bubble bound with C}.

\begin{Lem} Let $d\geq 2$ and $G=(V,E)\subset H=(V',E')$ be two subgraphs of $\mathbb Z^d$. For every $\beta\geq 0$, and every $x,y \in V'$,
\begin{equation}
	\langle \sigma_x\sigma_y\rangle_{H,\beta}-\langle \sigma_x\sigma_y\rangle_{G,\beta}\leq \sum_{u\in \partial V}\langle \sigma_x\sigma_u\rangle_{G,\beta}\langle \sigma_u\sigma_y\rangle_{H,\beta},
\end{equation}
where $\p V:= \{ x\in V : \exists x'\in V' \setminus V \text{ with } xx' \in E' \}$.
\end{Lem}

\begin{proof} Using the switching lemma as in \eqref{eq:applicationswitch3} gives
	\begin{align}
		\langle \sigma_x\sigma_y\rangle_{H,\beta}-\langle \sigma_x\sigma_y\rangle_{G,\beta} &= \langle \sigma_x\sigma_y \rangle_{H,\beta}\mathbf{P}^{xy, \emptyset}_{H,G,\beta} [\{x \connect{(\n_1+\n_2)_{|G}\:}y\}^c]\notag
		\\&\leq \sum_{u \in \partial V} \langle \sigma_x \sigma_y \rangle_{H,\beta} \mathbf{P}^{xy, \emptyset}_{H,G,\beta} [x\connect{(\n_1+\n_2)_{|G}\:}u]\notag
		 \\
		&= \sum_{u \in \partial V} \langle \sigma_x \sigma_u \rangle_{G,\beta} \langle \sigma_u \sigma_y \rangle_{H,\beta},
	\end{align}
	where in the second line we used the fact that 
	\begin{equation}
		\{x \connect{(\n_1+\n_2)_{|G}\:}y\}^c\subset \bigcup_{u\in \partial V}\{x\connect{(\n_1+\n_2)_{|G}\:}u\}
	\end{equation}
	together with a union bound, and on the last line we used the switching lemma as in \eqref{eq:applicationswitch2}.
\end{proof}

Recall that for every $S\subset \mathbb Z^d$ finite containing $0$, and every $\beta\geq 0$, 
\begin{equation}
	\varphi_\beta(S)=\beta\sum_{\substack{u\in  S\\v\notin S\\ u\sim v}}\langle \sigma_0\sigma_u\rangle_{S,\beta}.
\end{equation}
Also, the external vertex boundary of $S$ is defined by $\partial^{\textup{ext}} S=\{y\in \mathbb Z^d\setminus S:\: \exists x\in S, \: x\sim y\}$.
\begin{Lem} Let $d\geq 2$, $\beta>0$, $h\geq 0$, and $n\geq 1$. Let $J\in (\mathbb R^+)^{E(\mathbb Z^d)}$ which satisfies $J_e=\beta$, for every $e\in E(\Lambda_n)$. Let $\Lambda$ be a finite subset of $\mathbb Z^d$ containing $\Lambda_{2n}$. Then, for every $S\subset \Lambda_n$, 
\begin{equation}
	\langle \sigma_0\rangle^+_{\Lambda,J,h}-\langle \sigma_0\rangle_{\Lambda,J,h}\leq \varphi_\beta(S)\cdot \max_{x\in \partial^{\textup{ext}} S}\langle \sigma_x\rangle^+_{\Lambda,J,h}.
\end{equation}
\end{Lem}
\begin{proof} Let $S\subset \Lambda_n$. Recall from Remark \ref{rem:more general ising models} that we may write
\begin{equation}
	\langle \cdot \rangle_{\Lambda,\beta,h}^+=\langle \cdot \rangle_{\Lambda,\beta,\mathsf{h}},
\end{equation}
where $\mathsf{h}_x=h+\sum_{y\sim x}\mathds{1}_{y\notin \Lambda}$. Using Lemma \ref{lem:monobackbone} and noticing that $\{x\in \Lambda: \mathsf{h}_x>h\}=\partial \Lambda$,  
\begin{equation}\label{eq:pcor1}
	\langle \sigma_0\rangle^+_{\Lambda,\beta,h}-\langle 
	\sigma_0\rangle_{\Lambda,\beta,h}\leq 	\sum_{\gamma:0\rightarrow \fg}
	\rho_{\Lambda,\beta,\mathsf{h}}(\gamma)\mathds{1}_{\gamma \text{ reaches $\fg$ through }\partial \Lambda}.
\end{equation}
The above inequality has two important consequences. First, since $\Lambda$ contains $\Lambda_{2n}$, any path $\gamma$ contributing to \eqref{eq:pcor1} has to reach $\partial S$. This is because the path $\gamma$ is killed when reaching the ghost, and can only do so at the boundary of $\Lambda$. As a consequence, letting $uv$ with $u\in S$ and $v\notin S$ be the first edge visited by $\gamma$ which exits $S$, we can write $\gamma$ as the concatenation of $\gamma_1,\gamma_2,\gamma_3$, where:
\begin{equation}
	\gamma_1:0 \rightarrow u\subset S, \qquad \gamma_2=(uv), \qquad \gamma_3:v\rightarrow \fg,
\end{equation}
where $(uv)$ denotes the path consisting of the edge $uv$. Plugging this observation in \eqref{eq:pcor1} and using \eqref{eq:propbackb2} gives,
\begin{equation}
	\langle \sigma_0\rangle^+_{\Lambda,\beta,h}-\langle 
	\sigma_0\rangle_{\Lambda,\beta,h}\leq \sum_{\substack{u\in S\\v\notin S\\u\sim v}}\sum_{\substack{\gamma_1:0\rightarrow u\subset S\\\gamma_3:v\rightarrow \fg}}\rho_{\Lambda,\beta,\mathsf{h}}(\gamma_1\circ (uv))\rho_{\Lambda \setminus \overline{\gamma_1\circ (uv)},\beta,\mathsf{h}}(\gamma_3).\label{eq:pcor2}
\end{equation}
For any fixed $\gamma_1$ as above,
\begin{equation}\label{eq:pcor3}
	\sum_{\gamma_3:v\rightarrow \fg}\rho_{\Lambda \setminus \overline{\gamma_1\circ (uv)},\beta,\mathsf{h}}(\gamma_3)=\langle \sigma_v\rangle^+_{\Lambda\setminus\overline{\gamma_1\circ (uv)} ,\beta,h}\leq \max_{x\in \partial^{\textup{ext}}S}\langle \sigma_x\rangle^+_{\Lambda,\beta,h},
\end{equation}
where we used \eqref{eq:monot Griffiths} and the fact that $v\in \partial^{\textup{ext}}S$ by definition. Now, using \eqref{eq:propbackb2} and \eqref{eq:propbackb3} yields, 
\begin{equation}\label{eq:pcor4}
	\rho_{\Lambda,\beta,\mathsf{h}}(\gamma_1\circ(uv))=\rho_{\Lambda,\beta,\mathsf{h}}(\gamma_1)\rho_{\Lambda\setminus \overline{\gamma_1},\beta,\mathsf{h}}(\gamma_2)\leq \rho_{\Lambda,\beta,\mathsf{h}}(\gamma_1)\beta.
\end{equation}
Since $\gamma_1\subset S$, one also have $\rho_{\Lambda,\beta,\mathsf{h}}(\gamma_1)\leq \rho_{S,\beta}(\gamma_1)$. Plugging this observation, \eqref{eq:pcor3}, and \eqref{eq:pcor4} in \eqref{eq:pcor2} gives
\begin{equation}
	\langle \sigma_0\rangle^+_{\Lambda,\beta,h}-\langle 
	\sigma_0\rangle_{\Lambda,\beta,h}\leq \beta\sum_{\substack{u\in S\\v\notin S\\u\sim v}}\sum_{\gamma_1:0\rightarrow u\subset S}\rho_{S,\beta}(\gamma_1)\max_{x\in \partial^{\textup{ext}} S}\langle \sigma_x\rangle^+_{\Lambda,\beta,h},
\end{equation}
from which the proof follows readily.
\end{proof}

\subsection{Random currents and the FK-Ising model: coupling and applications}\label{sec:RCRFK}

The following result is proved in \cite[Theorem~3.2]{AizenmanDuminilTassionWarzelEmergentPlanarity2019} (see also \cite{hansen2025uniform,hansen2025general}) and originates from an observation made in \cite{lupu2016note}. Recall that the event $\mathcal F_A=\mathcal{F}_A^G$ introduced in \eqref{eq: event F_A} can be viewed as a percolation event (see Remark \ref{rem: fs is a perco event}). This allows to make sense of $\phi^0_{G,\beta}[\:\cdot\:|\: \mathcal F_A]$ below.

\begin{Lem}\label{lem: coupling FK current} Let $G=(V,E)$ be a subgraph of $\mathbb Z^d$ and let $A\subset V$ be a finite even subset of $G$. Let $\beta>0$. Let $\n$ be distributed according to $\mathbf P_{G,\beta}^{A}$. Let $(\omega_e)_{e\in E(\mathbb Z^d)}$ be an independent Bernoulli percolation with parameter $1-\exp(-\beta)$. For each $e\in \mathbb Z^d$, we define
\begin{equation}
	\eta_e:=\max(\mathds{1}_{\n_e>0},\omega_e).
\end{equation}
Then, the law of $\eta$ is $\phi_{G,\beta}^0[\:\cdot\:|\: \mathcal F_A]$. As a consequence, if $\mathcal A$ is an increasing event, one has
\begin{equation}
	\mathbf P^{A}_{G,\beta}[\mathcal A]\leq \phi_{G,\beta}^0[\mathcal A\: | \: \mathcal F_A].
\end{equation}
\end{Lem}
We will apply the above statement to the special cases $A=\emptyset$ or $A=\{x,y\}$ for $x,y\in V$. In the former case the law of $\eta$ is $\phi^0_{G,\beta}$, and in the latter case it is $\phi^0_{G,\beta}[\:\cdot\:|\:x\connect{}y]$, i.e. the FK-Ising measure conditioned on the event that $x$ is connected to $y$. Recall that when $G=\mathbb Z^d$, there is a unique translation invariant measure that we denote by $\phi_\beta$.

\subsubsection{A BK-type inequality for the FK-Ising model}
We use the coupling of Lemma \ref{lem: coupling FK current} to prove a new correlation inequality: Proposition \ref{prop:BK type for FK}. In the context of Bernoulli percolation, this inequality is a standard consequence of BK inequality. However, as discussed in \cite[Chapter~3.9]{Grimmett2006RCM}, the BK inequality is not satisfied for general random cluster models. Nevertheless, the following application of it remains valid for the FK-Ising model.

\begin{Prop}\label{prop: BK type in the RCR SECTION} Let $d\geq 2$ and $\beta>0$. Let $G=(V,E)$ be a subgraph of $\mathbb Z^d$ with $0\in V$. Let $S\cap B=\emptyset$ be finite subsets of $V$, with $S$ containing $0$. Assume that $\partial S$ is separating $0$ and $B$ in the sense that every path from $0$ to $B$ has to visit $\partial S$. Then,
\begin{equation}
	\phi_{G, \beta}^0[0 \connect{} B] \leq \sum_{\substack{u\in S\\ v\notin S\\ u\sim v}} \phi_{S, \beta}^0[0 \connect{} u]\beta\phi_{G, \beta}^0[v \connect{} B]. 
\end{equation}
\end{Prop}
\begin{proof} It is sufficient to treat the case where $G=(V,E)$ is finite as the general case follows by approximation.
Let $\beta>0$ and write $\phi=\phi^0_{G,\beta}$. 
	
	Write $B=\{b_1,\ldots,b_{|B|}\}$. If $x\in V$, we let $L^x:=\max\{1\leq k \leq |B|: x\connect{} b_k\}$ (with the convention that $L^x=0$ if $x$ is not connected to $B$). Moreover, let $B_{> \ell }:=\{b_{\ell + 1},\ldots, b_{|B|}\}$. By definition, 
	\begin{equation} \label{eq: FK connect as point-point}
		\phi[0 \leftrightarrow B] = \sum_{\ell = 1}^{|B|} \phi[L^0 = \ell] = \sum_{\ell = 1}^{|B|} \phi[0 \connect{} b_\ell]\phi[0 {\large \:\nleftrightarrow\:} B_{> \ell} \mid 0 \connect{} b_\ell]. 
	\end{equation}
	By Lemma \ref{lem: coupling FK current}, we may rewrite the right-most factor of \eqref{eq: FK connect as point-point} as
\begin{equation}\label{eq:pBK1}
	\phi [0 {\large \:\nleftrightarrow \:} B_{> \ell} \mid 0 \leftrightarrow b_\ell] = (\mathbf{P}^{0b_\ell}_{G,\beta} \otimes \mathbb{P}_{p_\beta}) [0 {\large \:\nleftrightarrow \:} B_{> \ell} \textup{ in }\eta],
\end{equation}
where we recall that $\eta$ is the percolation configuration on $E$ defined by $\eta_e=\max(\mathds{1}_{\n_e>0},\omega_e)$, where $\omega$ is distributed according to the Bernoulli percolation measure $\mathbb P_{p_\beta}$ with $p_\beta=1-e^{-\beta}$.
	Any current $\n$ with sources $\sn=\{0,b_\ell\}$ contains an odd path---the backbone $\Gamma(\n)$---from $0$ to $b_\ell$. 
Denote by $\Gamma^{S}(\n)$ the sub-exploration of the backbone of $\n$ started from $0$ and stopped at the first time it takes a step $uv$ exiting $S$. Observe that this is well-defined since, by assumption, any path from $0$ to $b_\ell$ has to visit $\partial S$, and then exit $S$ (since $S\cap B=\emptyset$). Decomposing the probability in \eqref{eq:pBK1} according to the value of $\Gamma^S(\n)$ yields
	\begin{equation}\label{eq:pBK2}
	(\mathbf{P}^{0b_\ell}_{G,\beta} \otimes \mathbb{P}_{p_\beta}) [0 {\large \:\nleftrightarrow \:} B_{> \ell}] 
	= \sum_{\substack{u\in S\\ v\notin S\\ u\sim v}} \sum_{\substack{\gamma: 0 \to u \\ \gamma\subset S}} 
	(\mathbf{P}^{0\ell}_\beta \otimes \mathbb{P}_{p_\beta})[ \Gamma^S(\n) = \gamma\circ(uv), 0 \nleftrightarrow B_{> \ell}\text{ in }\eta], 
\end{equation}
where we recall that $(uv)$ denotes the one step path from $u$ to $v$.
	Fix $(uv)$ and $\gamma: 0 \to u$ as in \eqref{eq:pBK2}, and let $\gamma_{uv}=\gamma\circ(uv)$. Observe that
\begin{multline}\label{eq:pBK3}
	\{\Gamma^S(\n) = \gamma_{uv}\} \cap \{0 {\large \:\nleftrightarrow \:} B_{> \ell} \text{ in } \eta\} = \{\Gamma^S(\n) = \gamma_{uv}\}\cap \{v {\large \:\nleftrightarrow \:}B_{> \ell} \text{ in }\eta\} \\\subset \{\Gamma^S(\n) = \gamma_{uv}\}\cap \{v {\large \:\nleftrightarrow \:}B_{> \ell} \text{ in } \eta_{|\overline{\gamma_{uv}}^c}\},
\end{multline}
where $\eta_{|\overline{\gamma_{uv}}^c}$ denotes the restriction of $\eta$ to edges in $\overline{\gamma}^c$. Additionally, note that the two events on the right-hand side of \eqref{eq:pBK3} are measurable with respect to $\overline{\gamma}_{uv}$ and $\overline{\gamma_{uv}}^c(=E\setminus \overline{\gamma_{uv}})$ respectively. What follows is a simple modification of Lemma \ref{lem:decompbackbone}. If $\n\in \Omega_G$, we write $\n=\n_1+\n_2$ where $\n_1$ (resp. $\n_2$) is supported on $\overline{\gamma_{uv}}$ (resp. $\overline{\gamma_{uv}}^c$). Observe that 
\begin{equation}\label{eq:pBK4}
	w^G_\beta(\n)=w^{\overline{\gamma_{uv}}}_\beta(\n_1)w^{\overline{\gamma_{uv}}^c}_\beta(\n_2).
\end{equation}
With the above observation, the event $\{\Gamma^S(\n) = \gamma_{uv}\}$ is measurable in terms of $\n_1$, and the event $\{v {\large \:\nleftrightarrow \:}B_{> \ell} \text{ in } \eta_{|\overline{\gamma_{uv}}^c}\}$ is  measurable in terms of the pair $(\n_2,\omega)$. Moreover, notice that $\{\Gamma^S(\n) = \gamma_{uv}\}=\{\Gamma(\n_1) = \gamma_{uv}\}$. Therefore, 
\begin{align}\notag
	(\mathbf P^{0b_\ell}_{G,\beta}\otimes \mathbb P_{p_\beta})&[\Gamma^S(\n) = \gamma_{uv},\: v {\large \:\nleftrightarrow \:}B_{> \ell} \text{ in } \eta_{|\overline{\gamma_{uv}}^c}]
	\\&=\frac{1}{Z^{0 b_\ell}_{G,\beta}}\sum_{\substack{\n_1\in \Omega_{\overline{\gamma_{uv}}}: \: \sn_1=\{0,v\}\\\n_2\in \Omega_{\overline{\gamma_{uv}}^c}:\:\sn_2=\{v,b_\ell\}}}w^{\overline{\gamma_{uv}}}_\beta(\n_1)\mathds{1}_{\Gamma(\n_1)=\gamma}w^{\overline{\gamma_{uv}}^c}_\beta(\n_2)\mathbb P_{p_\beta}[v {\large \:\nleftrightarrow \:}B_{> \ell} \text{ in } \eta_{|\overline{\gamma_{uv}}^c}]\notag
	\\&=\frac{Z^{0v}_{\overline{\gamma_{uv}},\beta}[\Gamma(\n)=\gamma_{uv}]Z^{vb_\ell}_{\overline{\gamma_{uv}}^c,\beta}}{Z^{0 b_\ell}_{G,\beta}}(\mathbf P_{\overline{\gamma_{uv}}^c,\beta}^{vb_\ell}\otimes \mathbb P_{p_\beta})[v {\large \:\nleftrightarrow \:}B_{> \ell} \text{ in } \eta_{|\overline{\gamma_{uv}}^c}].\label{eq:pBK5}
\end{align}
Now, we observe that
\begin{equation}\label{eq:pBK6}
	Z^{0v}_{\overline{\gamma_{uv}},\beta}[\Gamma(\n)=\gamma_{uv}]Z^{\emptyset}_{\overline{\gamma_{uv}}^c,\beta}=Z^{0v}_{G,\beta}[\Gamma(\n)=\gamma_{uv}]=Z^{0v}_{G,\beta}\mathbf P^{0v}_{G,\beta}[\Gamma(\n)=\gamma_{uv}].
\end{equation}
Plugging \eqref{eq:pBK6} in \eqref{eq:pBK5} yields
\begin{multline}\label{eq:pBK7}
	(\mathbf P^{0b_\ell}_{G,\beta}\otimes \mathbb P_{p_\beta})[\Gamma^S(\n) = \gamma_{uv},\: v {\large \:\nleftrightarrow \:}B_{> \ell} \text{ in } \eta_{|\overline{\gamma_{uv}}^c}]
\\= \frac{\langle \sigma_0\sigma_v\rangle_{G,\beta}\langle\sigma_v\sigma_{b_\ell}\rangle_{\overline{\gamma_{uv}}^c,\beta}}{\langle \sigma_0\sigma_{b_\ell}\rangle_{G,\beta}}\mathbf P^{0v}_{G,\beta}[\Gamma(\n)=\gamma_{uv}](\mathbf P_{\overline{\gamma_{uv}}^c,\beta}^{xb_\ell}\otimes \mathbb P_{p_\beta})[v {\large \:\nleftrightarrow \:}B_{> \ell} \text{ in } \eta_{|\overline{\gamma_{uv}}^c}].
\end{multline}
Plugging \eqref{eq:pBK3} and \eqref{eq:pBK7} in \eqref{eq:pBK2} gives
\begin{multline}\label{eq:pBK8}
	\phi[0\connect{}b_\ell]\phi [0 {\large \:\nleftrightarrow \:} B_{> \ell} \mid 0 \leftrightarrow b_\ell]=\langle \sigma_0\sigma_{b_\ell}\rangle_{G,\beta}(\mathbf{P}^{0b_\ell}_{G,\beta} \otimes \mathbb{P}_{p_\beta}) [0 {\large \:\nleftrightarrow \:} B_{> \ell}] 
\\\leq \sum_{\substack{u\in S\\ v\notin S\\ u\sim v}}\sum_{\substack{\gamma: 0 \to u \\ \gamma\subset S}}\rho_{G,\beta}(\gamma\circ(uv))\phi_{\overline{\gamma_{uv}}^c,\beta}^0[v\connect{}b_\ell]\phi_{\overline{\gamma_{uv}}^c,\beta}^0[v {\large \:\nleftrightarrow \:}B_{> \ell} \text{ in } \eta_{|\overline{\gamma_{uv}}^c}\mid  v\connect{}b_\ell],
\end{multline}
where in the first line we used \eqref{eq: consequences of ESC} and \eqref{eq:pBK1}, and on the second line we used (again) \eqref{eq: consequences of ESC} and Lemma \ref{lem: coupling FK current}. Proceeding as in \eqref{eq: FK connect as point-point},
\begin{equation}\label{eq:pBK9}
	\sum_{\ell=1}^{|B|}\phi_{\overline{\gamma_{uv}}^c,\beta}^0[v\connect{}b_\ell]\phi_{\overline{\gamma_{uv}}^c,\beta}^0[v {\large \:\nleftrightarrow \:}B_{> \ell} \text{ in } \eta_{|\overline{\gamma_{uv}}^c}\mid v\connect{}b_\ell]=\phi_{\overline{\gamma_{uv}}^c,\beta}^0[v\connect{}B].
\end{equation}
By Proposition \ref{prop:stoch dom FK}, the measure $\phi$ stochastically dominates $\phi^0_{\overline{\gamma_{uv}}^c,\beta}$, so that $\phi_{\overline{\gamma_{uv}}^c,\beta}^0[v\connect{}B]\leq \phi[v\connect{}B]$. Combining this observation with \eqref{eq: FK connect as point-point}, \eqref{eq:pBK8}, and \eqref{eq:pBK9}, gives
\begin{equation}\label{eq:pBK10}
	\phi[0\connect{}B]\leq \sum_{\substack{u\in S\\ v\notin S\\ u\sim v}}\phi[v\connect{} B] \sum_{\substack{\gamma: 0 \to u \\ \gamma \subset S}}\rho_{G,\beta}(\gamma\circ(uv)).
\end{equation}
Finally, 
\begin{equation}
	\sum_{\substack{\gamma: 0 \to u\\ \gamma\subset S}}\rho_{G,\beta}(\gamma\circ(uv))\leq \sum_{\substack{\gamma: 0 \to u\\\gamma \subset S}}\rho_{S,\beta}(\gamma)\beta=\langle \sigma_0\sigma_u\rangle_{S,\beta}\beta=\phi_{S,\beta}^0[0\connect{}u]\beta,\label{eq:pBK11}
\end{equation}
where in the inequality we used \eqref{eq:propbackb2} and then \eqref{eq:propbackb1} and \eqref{eq:propbackb3} to write 
\begin{equation}
\rho_{G,\beta}(\gamma\circ(uv))=\rho_{G,\beta}(\gamma)\rho_{G\setminus \overline{\gamma},\beta}((uv))\leq \rho_{S,\beta}(\gamma)\beta,
\end{equation}
and in the last equality of \eqref{eq:pBK11} we used \eqref{eq: consequences of ESC}. Plugging \eqref{eq:pBK11} in \eqref{eq:pBK10} concludes the proof.
\end{proof}

\begin{Rem} In the preceding proof, we relied on the random current representation. This is not essential to derive Proposition \ref{prop: BK type in the RCR SECTION}. Indeed, there is an alternative approach which uses Markov's property for the FK-Ising measure and which takes as an input the Simon--Lieb inequality (i.e. the case $|B|=1$). The two routes are of comparable difficulty and, although our proof is slightly more technical, it has the advantage of providing an example of how the coupling of Lemma \ref{lem: coupling FK current} can be used to analyse the FK-Ising model.
\end{Rem}

\subsubsection{Tree graph inequality}
We now turn to the proof of the tree-graph inequality of Proposition \ref{prop:tree graph}. Recall that $\{x\connect{}y,z\}$ denotes the event that $x$ is connected to both $y$ and $z$. 
\begin{Prop}[Tree-graph inequality] Let $G=(V,E)$ be a subgraph of $\mathbb Z^d$ and $\beta\geq 0$. Then, for every $x,y,z \in V$, 
\begin{equation}
	\phi_{G,\beta}^0[x\connect{}y,z]\leq \sum_{u,u'\in V: \: u\sim u'}\phi_{G,\beta}^0[x\connect{}u']\phi_{G,\beta}^0[u'\connect{} z]\phi^0_{G,\beta}[u\connect{}y].
\end{equation}
\end{Prop}
\begin{proof} We let $\phi=\phi_{G,\beta}^0$. The idea is to write
	\begin{equation}\label{eq:ptree1}
		\phi[x\connect{\:}y,z]=\phi[x\connect{\:}y\:|\: x\connect{\:}z]\phi[x\connect{\:}z].
	\end{equation}
	Using the coupling of Lemma \ref{lem: coupling FK current},
	\begin{equation}\label{eq:ptree2}
		\phi[x\connect{\:}y\:|\: x\connect{\:}z]=(\mathbf P^{xz}_{G,\beta}\otimes \mathbb P_{p_\beta})[x\connect{}y \text{ in }\eta],
	\end{equation}
	where we recall that $\eta_e=\max(\mathds{1}_{\n_e>0},\omega_e)$ where $\omega\sim\mathbb P_{p_\beta}$ which is the Bernoulli percolation measure of parameter $p_\beta=1-e^{-\beta}$. Let $\gamma:x\rightarrow z$ be any possible realisation of $\Gamma(\n)$. Observe that
	\begin{equation}\label{eq:ptree3}
		\{\Gamma(\n)=\gamma\}\cap \{x\connect{}y\text{ in }\eta\}\subset \{\Gamma(\n)=\gamma\}\cap \{y\connect{}\overline{\gamma} \text{ in }\eta_{\overline{\gamma}^c}\}.
	\end{equation}
	Thus, if we proceed as in the proof of Proposition \ref{prop:BK type for FK} and write $\n=\n_1+\n_2$ with $\n_1\in \Omega_{\overline{\gamma}}$ and $\n_2\in \Omega_{\overline{\gamma}^c}$, one has that $\{\Gamma(\n)=\gamma\}=\{\Gamma(\n_1)=\gamma\}$ is measurable in terms of $\n_1$, and that $\{\overline{\gamma}\connect{}y \text{ in }\eta_{\overline{\gamma}^c}\}$ is measurable in terms of the pair $(\n_2,\omega)$. By Lemma \ref{lem:decompbackbone},	\begin{align}\notag
		(\mathbf P^{xz}_{G,\beta}\otimes \mathbb P_{p_\beta})[ \{\Gamma(\n)=\gamma\}\cap \{y&\connect{}\overline{\gamma} \text{ in }\eta_{\overline{\gamma}^c}\}]
		\\&=\mathbf P^{xz}_{G,\beta}[\Gamma(\n)=\gamma](\mathbf P^{\emptyset}_{\overline{\gamma}^c,\beta}\otimes \mathbb P_{p_\beta})[\{y\connect{}\overline{\gamma} \text{ in }\eta_{\overline{\gamma}^c}\}].\label{eq:ptree4}
	\end{align}
	Now, observe that by Lemma \ref{lem: coupling FK current} and the stochastic domination $\phi_{\overline{\gamma}^c,\beta}^0\leq \phi$ (see Proposition \ref{prop:stoch dom FK}), 
	\begin{equation}\label{eq:ptree5}
		(\mathbf P^{\emptyset}_{\overline{\gamma}^c,\beta}\otimes \mathbb P_{p_\beta})[\{y\connect{}\overline{\gamma} \text{ in }\eta_{\overline{\gamma}^c}\}]=\phi_{\overline{\gamma}^c,\beta}^0[y\connect{}\overline{\gamma}]\leq \phi[y\connect{}\overline{\gamma}]\leq \sum_{u\in V}\mathds{1}_{u\in \overline{\gamma}}\phi[y\connect{}u].
	\end{equation} 
	Combining \eqref{eq:ptree1}--\eqref{eq:ptree5} gives
	\begin{equation}\label{eq:ptree6}
		\phi[x\connect{\:} y,z]\leq \sum_{u\in V}\phi[u\connect{}y]\sum_{\gamma:x\rightarrow z}\mathds{1}_{u\in \overline{\gamma}}\rho_{G,\beta}(\gamma).
	\end{equation}
	Note that if $u\in \overline{\gamma}$, then there exists $u'\sim u$ such that $u'\in \gamma$. Hence, \eqref{eq:ptree6} rewrites
	\begin{equation}
		\phi[x\connect{\:} y,z]\leq \sum_{u,u'\in V:\: u\sim u'}\phi[u\connect{}y]\sum_{\gamma:x\rightarrow z}\mathds{1}_{u'\in \gamma}\rho_{G,\beta}(\gamma).
	\end{equation}
	The chain rule of \eqref{eq:chainrule} combined with \eqref{eq: consequences of ESC} gives that
	\begin{equation}
		\sum_{\gamma:x\rightarrow z}\mathds{1}_{u'\in \gamma}\rho_{G,\beta}(\gamma)\leq \langle \sigma_x\sigma_{u'}\rangle_{G,\beta}\langle \sigma_{u'}\sigma_y\rangle_{G,\beta}= \phi[x\connect{}u']\phi[u'\connect{}z],
	\end{equation}
	from which the proof follows readily.
\end{proof}

\subsection{Proof of Lemma \ref{lem:existence of S for drc}}\label{appendix:lemma on phi(S)}

In this section, we prove Lemma \ref{lem:existence of S for drc}. For sake of clarity, we import the appropriate notations from \cite{AizenmanDuminilTassionWarzelEmergentPlanarity2019,DuminilPanis2024newLB}. For each $n\geq 1$, we let $\mathbb H_n:=\{x\in \mathbb Z^d: x_1=n\}$, and denote by $\mathcal R_n$ the reflection with respect to $\mathbb H_n$. We also fix $d\geq 4$.

We consider \emph{folded} single random currents on $\mathbb Z^d$ (for a review of this ``trick'' see \cite{aizenman2025geometric}), and partition the edge-set $E(\mathbb Z^d)$ into three disjoint sets:
\begin{align*}
    E_{-}(\mathbb H_n)&:=\Big\{ uv\in E(\mathbb Z^d) :\textup{at least one endpoint is strictly on the left of $\mathbb H_n$}\Big\},\\
    E_{+}(\mathbb H_n)&:=\Big\{ uv\in E(\mathbb Z^d) :\textup{at least one endpoint is strictly on the right of $\mathbb H_n$}\Big\},\\
    E_{0}(\mathbb H_n)&:=E(\mathbb Z^d)\setminus\left(E_{-}(\mathbb H_n)\cup E_{+}(\mathbb H_n)\right).
\end{align*}
Decompose $\n\in \Omega_{\mathbb Z^d}$ into its restrictions $\n_{-}$, $\n_{+}$ and $\n_0$ to the above three subsets of $E(\mathbb Z^d)$. It is convenient to consider the \emph{multigraph} representation of these objects. Consider the multigraph $\mathcal{M}_n$
 obtained by taking the union of the multigraph $\mathcal{N}_{-}$ associated with $\n_{-}$ and the reflection $\mathcal{R}_n(\mathcal{N}_+)$ of the multigraph $\mathcal{N}_+$ associated with $\n_{+}$. 
The above definitions are all with respect to the direction $\mathbf e_1$. In order to generalise them, we write $\mathcal{M}_n(+\mathbf{e}_1):=\mathcal{M}_n$, $\mathcal R_n(+\mathbf{e}_1):=\mathcal R_n$, $\mathbb H_n(+\mathbf{e}_1):=\mathbb H_n$, and define
\begin{equation}
    \mathcal{S}_n(+\mathbf{e}_1):=\Big\lbrace x\in \Lambda_n: \: x \overset{\mathcal M_n(+\mathbf e_1)}{\centernot\longleftrightarrow} \mathbb H_n(+\mathbf{e}_1)\Big\rbrace.
\end{equation}
\noindent Similarly, define $\mathbb H_n(\pm \mathbf{e}_i)$, $\mathcal R_n(\pm \mathbf{e}_i)$, $\mathcal{M}_n(\pm \mathbf{e}_i)$, and $\mathcal{S}_n(\pm\mathbf{e}_i)$ for $1\leq i \leq d$ in the other $2d-1$ directions. 
Define
\begin{equation}
    \mathcal{S}_n:=\bigcap_{1\leq i \leq d}\left(\mathcal{S}_n(+\mathbf{e}_i)\cap \mathcal{S}_n(-\mathbf{e}_i)\right).
\end{equation}
In \cite{DuminilPanis2024newLB}, the authors proved the following result.
\begin{Lem}\label{lem:appendix1} There exists $K_1>0$ such that, for every $n\geq 1$,
\begin{equation}
	\mathbf E^{\emptyset}_{\beta_c}[\mathds{1}_{0\in \mathcal S_n}\varphi_{\beta_c}(\mathcal S_n)]\leq K_1 (\log n)^{\mathds{1}_{d=4}}.
\end{equation}
\end{Lem}
\begin{proof} Using \cite[Lemma~2.5]{DuminilPanis2024newLB}, one has
\begin{equation}\label{eq:pa1}
	\mathbf E^{\emptyset}_{\beta_c}[\mathds{1}_{0\in \mathcal S_n}\varphi_{\beta_c}(\mathcal S_n)]\leq \sum_{\substack{x,y\in \Lambda_n\\x\sim y}}\Big(\langle \sigma_0\sigma_x\rangle_{\beta_c}-\langle \sigma_0\sigma_{\mathcal R_n(x)}\rangle_{\beta_c}\Big)\langle \sigma_y\sigma_{\mathcal R_n(y)}\rangle_{\beta_c}.
\end{equation}
Moreover, thanks to \cite[(1.19)]{DuminilPanis2024newLB}, 
\begin{multline}\label{eq:pa2}
	\sum_{\substack{x,y\in \Lambda_n\\x\sim y}}\Big(\langle \sigma_0\sigma_x\rangle_{\beta_c}-\langle \sigma_0\sigma_{\mathcal R_n(x)}\rangle_{\beta_c}\Big)\langle \sigma_y\sigma_{\mathcal R_n(y)}\rangle_{\beta_c}\\\lesssim \langle \sigma_0\sigma_{\lfloor n/4\rfloor}\rangle_{\beta_c}\Big(\chi_n(\beta_c)+n^{d-2}\sum_{k=0}^n(k+2)\langle \sigma_0\sigma_{k\mathbf{e}_1}\rangle_{\beta_c}\Big),
\end{multline}
where $\chi_n(\beta_c)=\sum_{x\in \Lambda_n}\langle \sigma_0\sigma_x\rangle_{\beta_c}$. Plugging \eqref{eq:pa2} in \eqref{eq:pa1}, and using the infrared bound \eqref{eq:IRB} yields
\begin{equation}
	\mathbf E^{\emptyset}_{\beta_c}[\mathds{1}_{0\in \mathcal S_n}\varphi_{\beta_c}(\mathcal S_n)]\lesssim (\log n)^{\mathds{1}_{d=4}},
\end{equation}
which concludes the proof.
\end{proof}
A second result we will use is the following consequence of the switching lemma for these folded currents.
\begin{Lem}\label{lem:appendix2} Let $x\in \Lambda_n$ and $y\sim x$. One has
\begin{equation}\label{eq:switchappendix1}
	\mathbf P^{\emptyset}_{\beta_c}[y\connect{\mathcal M_n\:}\mathbb H_n]=\langle \sigma_y\sigma_{\mathcal R_n(y)}\rangle_{\beta_c},
\end{equation}
and
\begin{equation}\label{eq:switchappendix2}
    \mathbf E^{\emptyset}_{\beta_c}\Big[\mathds{1}\{0,x\in \mathcal{S}_n,\: y \xleftrightarrow[]{\:\mathcal M_n\:} \mathbb H_n \}\langle\sigma_0\sigma_x\rangle_{\mathcal{S}_n,\beta}\Big]
    \leq
   \Big(\langle \sigma_0\sigma_x\rangle_{\beta_c}-\langle \sigma_0\sigma_{\mathcal{R}_n(x)}\rangle_{\beta_c}\Big)\langle \sigma_y \sigma_{\mathcal{R}_n(y)}\rangle_{\beta_c}.
\end{equation}
\end{Lem}
\begin{proof} See \cite[(2.28)]{DuminilPanis2024newLB}.
\end{proof}
We are now in a position to prove Lemma \ref{lem:existence of S for drc}.
\begin{proof}[Proof of Lemma \textup{\ref{lem:existence of S for drc}}] We first prove the result for $n$ large enough. Let $K>0$ to be fixed. Introduce the event
\begin{equation}
	\mathcal A(K):=\{\varphi_{\beta_c}(\mathcal S_n)\leq K(\log n)^{\mathds{1}_{d=4}}\}\cap \{0\in \mathcal S_n\}.
\end{equation}
We will prove that for $n$ and $K$  large enough, one has $\mathbf P^{\emptyset}_{\beta_c}[\mathcal A(K)]\geq \tfrac{1}{2}$. Assume for a moment that it is the case. If $S\subset \Lambda_n$, we let
\begin{equation}
	F_n(S):=\beta_c\sum_{\substack{x\in S\cap \Lambda_{n/2}\\y\notin S\\x\sim y}}\langle \sigma_0\sigma_x\rangle_{S,\beta_c}\langle \sigma_0\sigma_y\rangle_{\beta_c}
\end{equation}
 Using Lemma \ref{lem:appendix2} above, and the infrared bound \eqref{eq:IRB},
 \begin{align}
 \begin{split}
 	\mathbf E^\emptyset_{\beta_c}[\mathds{1}_{\mathcal A(K)}F_n(\mathcal S_n)]&\leq \sum_{\substack{x\in \Lambda_{n/2}\\y\in \Lambda_n\\ x\sim y}}\langle \sigma_0\sigma_y\rangle_{\beta_c}\mathbf E_{\beta_c}^\emptyset\Big[\mathds{1}\{0,x\in \mathcal S_n, y\notin \mathcal S_n\}\langle \sigma_0\sigma_x\rangle_{\mathcal S_n,\beta_c}\Big]
 	\\&\leq \sum_{\substack{x\in \Lambda_{n/2}\\y\in \Lambda_n\\ x\sim y}}\langle \sigma_0\sigma_y\rangle_{\beta_c}\langle \sigma_0\sigma_x\rangle_{\beta_c}\langle \sigma_y\sigma_{\mathcal R_n(y)}\rangle_{\beta_c}
 	\lesssim \frac{(\log n)^{\mathds{1}_{d=4}}}{n^{d-2}}.
 	\end{split}
 \end{align}
As a consequence, there exists a realisation $S$ of $\mathcal S_n$ which lies in $\mathcal A(K)$ and satisfies
\begin{equation}
	\mathbf P^\emptyset_{\beta_c}[\mathcal A(K)]\cdot F_n(S)\lesssim \frac{(\log n)^{\mathds{1}_{d=4}}}{n^{d-2}}.
\end{equation}
Since $\mathbf P^{\emptyset}_{\beta_c}[\mathcal A(K)]\geq \tfrac{1}{2}$ by assumption, this concludes the proof that $S$ satisfies $(i)$ and $(ii)$.

It remains to prove that $\mathbf P^{\emptyset}_{\beta_c}[\mathcal A(K)]\geq \tfrac{1}{2}$ for $n,K$ large enough. First, by Markov's inequality and Lemma \ref{lem:appendix1}, 
\begin{align}\label{eq:pa3}
\begin{split}
	\mathbf P^{\emptyset}_{\beta_c}[\varphi_{\beta_c}(\mathcal S_n)\geq K(\log n)^{\mathds{1}_{d=4}}]&\leq\frac{1}{K(\log n)^{\mathds{1}_{d=4}}}\mathbf E_{\beta_c}^{\emptyset}[\varphi_{\beta_c}(\mathcal S_n)]\\&= \frac{1}{K(\log n)^{\mathds{1}_{d=4}}}\mathbf E_{\beta_c}^{\emptyset}[\mathds{1}_{0\in \mathcal S_n}\varphi_{\beta_c}(\mathcal S_n)]\leq \frac{K_1}{K}.
	\end{split}
\end{align}
Then,
\begin{equation}\label{eq:pa4}
	\mathbf P^{\emptyset}_{\beta_c}[\{0\in \mathcal S_n\}^c]= \mathbf P^{\emptyset}_{\beta_c}[0\connect{\mathcal M_n\:}\mathbb H_n]=\langle \sigma_0\sigma_{\mathcal R_n(0)}\rangle_{\beta_c}\lesssim n^{2-d},
	\end{equation}
	where in the first equality we used the definition of $\mathcal S_n$, in the second we used \eqref{eq:switchappendix1}, and in the last we used the infrared bound \eqref{eq:IRB}. It remains to choose $K$ and $n$ large enough so that both \eqref{eq:pa3} and \eqref{eq:pa4} are smaller than $\tfrac{1}{4}$. The proof follows from adjusting the value of the constant $C$ to accommodate the small values of $n$.
\end{proof}

\section{One-arm exponent of the double random current measure}\label{sec:DRC}
In this section, we prove Theorems \ref{thm:DRC} and \ref{thm:DRC2}. We begin with the (more interesting) proof of the upper bounds.
\subsection{Proof of the upper bounds}
The most naive approach to get an upper bound on the one-arm probability is to use a first moment method. In the case of the full-space double random current measure, such an approach gives\begin{equation}\label{eq:naive upper bound one arm DRC}
	\mathbf P_{\beta_c}^{\emptyset,\emptyset}[0\connect{\n_1+\n_2\:}\partial \Lambda_n]\stackrel{\phantom{\eqref{eq:applicationswitch1}}}\leq \sum_{x\in \partial \Lambda_n}\mathbf P^{\emptyset,\emptyset}_{\beta_c}[0\connect{\n_1+\n_2\:}x]\stackrel{\eqref{eq:applicationswitch1}}=\sum_{x\in \partial \Lambda_n}\langle \sigma_0\sigma_x\rangle_{\beta_c}^2\stackrel{\eqref{eq:IRB}}\leq \frac{C}{n^{d-3}},
\end{equation}
for some $C=C(d)>0$. Above, we used \eqref{eq:applicationswitch1} which follows from the switching lemma. In a way, it makes explicit the fact that the sourceless double random current model is a loop model: to connect $0$ and $x$, one has to ``pay'' a connection from $0$ to $x$---which ``costs'' $\langle \sigma_0\sigma_x\rangle_{\beta_c}$---and then a connection from $x$ to $0$---which costs an additional $\langle \sigma_x\sigma_0\rangle_{\beta_c}$. To gain an exponent in \eqref{eq:naive upper bound one arm DRC}, one may try to replace one of the full-space correlations functions by $\langle \sigma_0\sigma_x\rangle_{S,\beta_c}$ for some well-chosen $S\subset \Lambda_n$ (which will be given by Lemma \ref{lem:existence of S for drc}). This improvement would be very reasonable if we had a way to explore the double random current cluster of $0$, and to ``stop'' this exploration as soon as it reaches $\partial S$. The next lemma justifies this intuition. We believe the strategy of proof (which relies on some kind of exploration) to be of independent interest in the study of sourceless currents. 
 \begin{Lem}\label{lem:one arm backbone} Let $d\geq 2$, $\beta\geq 0$, and $n\geq 1$. Let $G=(V,E)$ be a subgraph of $\mathbb Z^d$ which contains $\Lambda_{2n}$. Let $S\subset \Lambda_{n-1}$ containing $0$. Then,
\begin{equation}
	\mathbf P^{\emptyset,\emptyset}_{G,\beta}[0\connect{\n_1+\n_2\:}\partial \Lambda_n]\leq \sum_{\substack{u\in S\\v\notin S\\u\sim v}}\langle \sigma_0\sigma_u\rangle_{S,\beta}\beta\langle \sigma_v\sigma_0\rangle_{G,\beta}. 
\end{equation}
\end{Lem}
\begin{proof}  It is sufficient to treat the case where $G=(V,E)$ is finite as the general case follows by approximation.
 We fix $\beta\geq 0$ and drop it from the notations. The following argument is very close to the one used in the proof of Proposition \ref{prop:BK type for FK} and we import some computations from there.
 
Writing $\partial \Lambda_n=\{x_1,\ldots,x_m\}$ (where $m=|\partial \Lambda_n|$) and decomposing according to the largest $\ell$ such that $0$ connects to $x_\ell$ in $\n_1+\n_2$ yields
\begin{equation}\label{eq:upper1}
	\mathbf P^{\emptyset,\emptyset}_{G}[0\connect{\n_1+\n_2\:}\partial \Lambda_n]=\sum_{\ell=1}^m\mathbf P^{\emptyset,\emptyset}_G[\{0\connect{\n_1+\n_2\:}x_\ell\}\cap \big\{0\connect{\n_1+\n_2\:}\partial \Lambda_n\setminus \{x_1,\ldots ,x_\ell\}\big\}^c].
\end{equation}
By the switching lemma (see Lemma \ref{lem:switching}),
\begin{equation}\label{eq:upper2}
	\mathbf P^{\emptyset,\emptyset}_G[0\connect{\n_1+\n_2\:}\partial \Lambda_n]=\sum_{\ell=1}^m \langle \sigma_0\sigma_{x_\ell}\rangle^2_G\mathbf P^{0x_\ell,0x_\ell}_G[\mathcal E(\ell;0)],
\end{equation}
where for $x\in V$, $\mathcal E(\ell;x):=\big\{x\connect{\n_1+\n_2\:}\partial \Lambda_n\setminus \{x_1,\ldots ,x_\ell\}\big\}^c$. The idea is to explore the portion of $\Gamma(\n_1)$ up to the first edge $uv$ with $u\in S$ and $v\notin S$ it visits. We denote it by $\Gamma^S(\n_1)$. 
Recall that $(uv)$ denotes the one-step path between $u$ and $v$. Observe that,
\begin{align}
	\mathbf P^{0x_\ell,0x_\ell}_{G}[\mathcal E(\ell)]&=
	\sum_{\substack{u\in S\\v\notin S\\u\sim v}}\sum_{\substack{\gamma:0\rightarrow u\subset S}}\mathbf P^{0x_\ell,0x_\ell}_G[\mathcal E(\ell;0)\cap \{\Gamma^S(\n_1)=\gamma\circ (uv)\}]\notag
	\\&\leq \sum_{\substack{u\in S\\v\notin S\\u\sim v}}\sum_{\substack{\gamma:0\rightarrow u\subset S}}\mathbf P^{0x_\ell,0x_\ell}_G[\mathcal E_{\overline{\gamma\circ(uv)}^c}(\ell;v)\cap \{\Gamma^S(\n_1)=\gamma\circ (uv)\}],\label{eq:upper3}
\end{align}
where for a set of edges $E'\subset E$ and $x\in V$, we have set $\mathcal E_{E'}(\ell;x):=\big\{x\connect{(\n_1+\n_2)_{|E'}\:}\partial \Lambda_n\setminus \{x_1,\ldots,x_\ell\}\big\}^c$, and where we used that 
\begin{equation}
\mathcal E(\ell;0)\cap \{\Gamma^S(\n_1)=\gamma\circ (uv)\}\subset \mathcal E_{\overline{\gamma\circ(uv)}^c}(\ell;v)\cap \{\Gamma^S(\n_1)=\gamma\circ (uv)\}. 
\end{equation}
Below, we let $\gamma_{uv}:=\gamma\circ(uv)$. 
Decomposing the probability in \eqref{eq:upper3} according to the value of $\n_2$ gives
\begin{align}\notag
	\mathbf P^{0x_\ell,0x_\ell}_G[&\mathcal E_{\overline{\gamma_{uv}}^c}(\ell;v)\cap \{\Gamma^S(\n_1)=\gamma_{uv}\}]
	\\&=\sum_{\sn_2=\{0,x_\ell\}}\mathbf P^{0x_{\ell}}_G[\n_2]\mathbf P^{0x_{\ell}}_G[\{\n+\n_2\in \mathcal E_{\overline{\gamma_{uv}}^c}(\ell;v)\}\cap \{\Gamma^S(\n)=\gamma_{uv}\}]\notag
	\\&=\sum_{\sn_2=\{0,x_\ell\}}\mathbf P^{0x_{\ell}}_G[\n_2]\frac{\langle \sigma_0\sigma_v\rangle_G\langle \sigma_v\sigma_{x_\ell}\rangle_{\overline{\gamma_{uv}}^c}}{\langle \sigma_0\sigma_{x_\ell}\rangle_G}\mathbf P^{0v}_{G}[\Gamma(\n)=\gamma_{uv}]\mathbf P^{vx_{\ell}}_{\overline{\gamma_{uv}}^c}[\n+\n_2\in \mathcal E_{\overline{\gamma_{uv}}^c}(\ell;v)]\notag
	\\&=\frac{\langle \sigma_0\sigma_v\rangle_G\langle \sigma_v\sigma_{x_\ell}\rangle_{\overline{\gamma_{uv}}^c}}{\langle \sigma_0\sigma_{x_\ell}\rangle_G}\mathbf P^{0v}_G[\Gamma(\n)=\gamma_{uv}]\mathbf P^{vx_\ell,0x_\ell}_{\overline{\gamma_{uv}}^c,G}[\mathcal E_{\overline{\gamma_{uv}}^c}(\ell;v)],\label{eq:upper4}
\end{align}
where we argued as in \eqref{eq:pBK7} (which is a simple modification of Lemma \ref{lem:decompbackbone}) in the second equality.
Combining \eqref{eq:upper2},\eqref{eq:upper3}, and \eqref{eq:upper4} gives
\begin{align}\notag
	\mathbf P^{\emptyset,\emptyset}_G[0\connect{\n_1+\n_2\:}\partial \Lambda_n]
	&\leq \sum_{\substack{u\in S\\v\notin S\\u\sim v}}\sum_{\substack{\gamma:0\rightarrow u\subset S}}\langle \sigma_0\sigma_v\rangle_G\mathbf P^{0v}_G[\Gamma(\n)=\gamma_{uv}]\notag
	\\&\qquad\times\sum_{\ell=1}^m\langle \sigma_v\sigma_{x_\ell}\rangle_{\overline{\gamma_{uv}}^c}\langle \sigma_0\sigma_{x_\ell}\rangle_G\mathbf P_{\overline{\gamma_{uv}}^c,G}^{vx_{\ell},0 x_{\ell}}[\mathcal E_{\overline{\gamma_{uv}}^c}(\ell;v)]\notag
	\\&= \sum_{\substack{u\in S\\v\notin S\\u\sim v}}\sum_{\substack{\gamma:0\rightarrow u\subset S}}\langle \sigma_0\sigma_v\rangle_G\mathbf P^{0v}_G[\Gamma(\n)=\gamma_{uv}]\notag
	\\&\qquad\times\sum_{\ell=1}^m\langle \sigma_0\sigma_{v}\rangle_{G}\mathbf P_{\overline{\gamma_{uv}}^c,G}^{\emptyset,0v}[\{v\connect{(\n_1+\n_2)_{|\overline{\gamma_{uv}}^c}\:}x_\ell\}\cap\mathcal E_{\overline{\gamma_{uv}}^c}(\ell;v)]\notag
	\\&=\sum_{\substack{u\in S\\v\notin S\\u\sim v}}\sum_{\substack{\gamma:0\rightarrow u\subset S}}\langle \sigma_0\sigma_v\rangle^2_G\mathbf P^{0v}_G[\Gamma(\n)=\gamma_{uv}],\label{eq:upper5}
\end{align}
where in the first equality we used the switching lemma (Lemma \ref{lem:switching}), and in the second one we used that, similarly to \eqref{eq:upper1}, 
\begin{equation}
	\sum_{\ell=1}^m\mathbf P_{\overline{\gamma_{uv}}^c,G}^{\emptyset,0v}[\{v\connect{(\n_1+\n_2)_{|\overline{\gamma_{uv}}^c}\:}x_\ell\}\cap\mathcal E_{\overline{\gamma_{uv}}^c}(\ell;v)]=\mathbf P^{\emptyset,0v}_{\overline{\gamma_{uv}}^c,\mathbb Z^d}[v\connect{(\n_1+\n_2)_{\overline{\gamma_{uv}}^c}\:}\partial \Lambda_n]\leq 1.\label{eq:upper6}
\end{equation}
Proceeding as in \eqref{eq:pBK11} gives
\begin{equation}\label{eq:upper6}
	\sum_{\substack{\gamma:0\rightarrow u\subset\Lambda_{n-1}}}\langle \sigma_0\sigma_v\rangle_G\mathbf P^{0v}_G[\Gamma(\n)=\gamma_{uv}]\leq \langle \sigma_0\sigma_u\rangle_{S}\beta.
	\end{equation}
Combining \eqref{eq:upper5} and \eqref{eq:upper6} concludes the proof.
\end{proof}

We are now in a position to prove the upper bound in Theorem \ref{thm:DRC}. Our other main input for the proof is Lemma \ref{lem:existence of S for drc} which gives the following: if $d\geq 4$, there exists $C_0>0$ such that, for every $n\geq 1$, there is $S\subset \Lambda_n$ which satisfies the following properties:
\begin{enumerate}
	\item[$(i)$] $\varphi_{\beta_c}(S)\leq C_0(\log n)^{\mathds{1}_{d=4}}$;
	\item[$(ii)$] one has
	\begin{equation}
		\beta_c\sum_{\substack{x\in S\cap \Lambda_{n/2}\\y\notin S\\ x\sim y}}\langle \sigma_0\sigma_x\rangle_{S,\beta_c}\langle \sigma_0\sigma_y\rangle_{\beta_c}\leq \frac{C_0(\log n)^{\mathds{1}_{d=4}}}{n^{d-2}}.
	\end{equation} 
\end{enumerate}
\begin{proof}[Proof of the upper bounds in Theorems \textup{\ref{thm:DRC}} and \textup{\ref{thm:DRC2}}] Let $d\geq 4$. We fix $\beta=\beta_c$ and drop it from the notations. Let $n\geq 1$.
Lemma \ref{lem:one arm backbone} applied to $G=\mathbb Z^d$ and the set $S$ from above provided by Lemma \ref{lem:existence of S for drc}  gives that
\begin{equation}\label{eq:pubdrc1}
	\mathbf P^{\emptyset,\emptyset}[0\connect{\n_1+\n_2\:}\partial \Lambda_n]\leq \sum_{\substack{u\in S\\v\notin S\\u\sim v}}\langle \sigma_0\sigma_u\rangle_{S}\beta_c \langle \sigma_v\sigma_0\rangle \leq \sum_{\substack{u\in S\cap \Lambda_{n/2}\\v\notin S\\ u\sim v}}(\ldots) + \sum_{\substack{u\in S\setminus\Lambda_{n/2}\\v\notin S\\ u\sim v}}(\ldots)=(\mathrm{I})+(\mathrm{II})
\end{equation}
First, by property $(ii)$ above, one has 
\begin{equation}\label{eq:pubdrc2}
	(\mathrm{I})\lesssim \frac{(\log n)^{\mathds{1}_{d=4}}}{n^{d-2}}.
\end{equation} 
We now handle the second term in \eqref{eq:pubdrc1}. By the infrared bound \eqref{eq:IRB},
\begin{equation}\label{eq:pubdrc3}
	(\mathrm{II})\lesssim \frac{1}{n^{d-2}}\varphi_{\beta_c}(S)\lesssim\frac{(\log n)^{\mathds{1}_{d=4}}}{n^{d-2}},
\end{equation}
where we used $(i)$ in the second inequality. Combining \eqref{eq:pubdrc2} and \eqref{eq:pubdrc3} in \eqref{eq:pubdrc1} concludes the proof.
\end{proof}

\begin{Rem}\label{rem:upper bound drc d=4} Let $d\geq 4$. Applying Lemma \ref{lem:one arm backbone} to $G=\mathbb Z^d$ and any $S$ such that $\partial S\subset \Lambda_n\setminus\Lambda_{n/2}$ yields, for every $n\geq 1$,
\begin{equation}\label{eq:rem what the bound give in terms of phi(S)}
	\mathbf P^{\emptyset,\emptyset}_{\beta_c}[0\connect{\n_1+\n_2\:}\partial \Lambda_n]\leq \frac{C\varphi_{\beta_c}(S)}{n^{d-2}}.
\end{equation}
In dimensions $d>4$, Lemma \ref{lem:existence of S for drc} proves in a ``weak sense'' that for every $n\geq 1$ there exists $S$ with $\partial S\subset \Lambda_n\setminus\Lambda_{n/2}$ such that $\varphi_{\beta_c}(S)\lesssim (\log n)^{\mathds{1}_{d=4}}$. 
%
\end{Rem}

\subsection{Proof of the lower bounds} 
We now turn to the lower bounds on the one-arm probability. We rely on a second moment method, as for the case of the FK-Ising measure with free boundary conditions. We will use the following correlation inequality which already appeared in \cite[Proposition~A.3]{AizenmanDuminilTriviality2021}.

\begin{Lem}\label{lem:2pt connect} Let $d\geq 2$. For every $\beta\geq 0$, and every $x,y\in \mathbb Z^d$,
\begin{equation}
	\mathbf P_{\beta}^{\emptyset,\emptyset}[x,y\in \mathbf{C}_{\n_1+\n_2}(0)]\leq \langle \sigma_0\sigma_x\rangle_{\beta}\langle \sigma_x\sigma_y\rangle_{\beta}\langle \sigma_y\sigma_0\rangle_{\beta}+(x\Leftrightarrow y),
\end{equation}
where $(x\Leftrightarrow y)$ denotes the same term with $x$ and $y$ inverted.
\end{Lem}

\begin{Rem} Lemma \ref{lem:2pt connect} is the random current counterpart of Proposition \ref{prop:tree graph}. The former suggests a path-like structure for the clusters of the double random current model, while the latter suggests a tree-like structure for the clusters of the FK-Ising model. 
\end{Rem}

We begin with the case $d>4$.

\begin{Prop}\label{prop:smfull space DRC} Let $d>4$. There exists $c>0$ such that, for every $n\geq 1$,
\begin{equation}
	\mathbf P^{\emptyset,\emptyset}_{\beta_c}[0\connect{\n_1+\n_2\:}\partial \Lambda_n]\geq \frac{c}{n^{d-2}}.
\end{equation}
\end{Prop}
\begin{proof} We fix $\beta=\beta_c$ and drop it from the notations. Let $n\geq 1$ and 
\begin{equation}
\mathcal N:=\sum_{x\in \partial \Lambda_n}\mathds{1}\{0\connect{\n_1+\n_2\:}x\}.
\end{equation} By the Cauchy--Schwarz inequality,
\begin{equation}\label{eq:second moment method}
	\mathbf P^{\emptyset,\emptyset}_{}[0\connect{\n_1+\n_2\:}\partial \Lambda_n]=\mathbf P_{}^{\emptyset,\emptyset}[\mathcal N>0]\geq \frac{\mathbf E^{\emptyset,\emptyset}_{}[\mathcal N]^2}{\mathbf E^{\emptyset,\emptyset}_{}[\mathcal N^2]}.
\end{equation}
On the one hand, applying the switching lemma as in \eqref{eq:applicationswitch1} yields
\begin{equation}
	\mathbf E_{}^{\emptyset,\emptyset}[\mathcal N]\stackrel{\phantom{\eqref{eq: lower bound full-space sec 2}}}=\sum_{x\in \partial \Lambda_n}\langle \sigma_0\sigma_x\rangle_{}^2\stackrel{\eqref{eq: lower bound full-space sec 2}}\gtrsim \frac{|\partial \Lambda_n|}{(n^{d-2})^2}\stackrel{\phantom{\eqref{eq: lower bound full-space sec 2}}}\gtrsim\frac{1}{n^{d-3}}.\label{eq:bound first moment}
\end{equation}
On the other hand, 
\begin{align}
	\mathbf E^{\emptyset,\emptyset}[\mathcal N^2]&=\sum_{x,y\in \partial\Lambda_n}\mathbf P^{\emptyset,\emptyset}[x,y \in \mathbf C_{\n_1+\n_2}(0)]\notag
	\\&\leq \sum_{x,y\in \partial \Lambda_n}\Big(\langle \sigma_0\sigma_x\rangle\langle \sigma_x\sigma_y\rangle\langle\sigma_y\sigma_0\rangle+(x\Leftrightarrow y)\Big)\notag
	\\&\lesssim (n^{2-d})^2\sum_{x,y \in \partial \Lambda_n}\frac{1}{(1\vee |x-y|)^{d-2}}\notag
	\\&\lesssim (n^{2-d})^2\sum_{x\in \partial \Lambda_n}\sum_{k=0}^n\frac{k^{d-2}}{(1\vee k)^{d-2}}\lesssim\frac{1}{n^{d-4}},\label{eq:bound second moment}
\end{align}
where we used Lemma \ref{lem:2pt connect} on the second line, and the infrared bound \eqref{eq:IRB} in the third line. Plugging \eqref{eq:bound first moment} and \eqref{eq:bound second moment} in \eqref{eq:second moment method} yields
\begin{equation}
	\mathbf P^{\emptyset,\emptyset}_{}[0\connect{\n_1+\n_2\:}\partial \Lambda_n]\gtrsim \frac{(n^{3-d})^2}{n^{4-d}}=\frac{1}{n^{d-2}},
\end{equation}
which concludes the proof.
\end{proof}

\begin{Prop} Let $d=4$. There exists $c>0$ such that for every $n\geq 1$,
\begin{equation}
	\mathbf P^{\emptyset,\emptyset}_{\beta_c}[0\connect{\n_1+\n_2\:}\partial \Lambda_n]\geq \frac{c}{n^2(\log n)^4}.
\end{equation}
\end{Prop}
\begin{proof} The proof follows the exacts same lines as the one of Proposition \ref{prop:smfull space DRC} except that the lower bound on $\mathbf E_{\beta_c}^{\emptyset,\emptyset}[\mathcal N]$ is different. Indeed, replacing the use of \eqref{eq: lower bound full-space sec 2} by that of \eqref{eq: lb d=4 sec 2} in \eqref{eq:bound first moment} yields
\begin{equation}
	\mathbf E^{\emptyset,\emptyset}_{\beta_c}[\mathcal N]\gtrsim \frac{1}{n(\log n)^2}.
\end{equation} 
The same computation as in \eqref{eq:bound second moment} gives $\mathbf E^{\emptyset,\emptyset}_{\beta_c}[\mathcal N^2]\lesssim C$. By \eqref{eq:second moment method}, this yields
\begin{equation}
	\mathbf P^{\emptyset,\emptyset}_{\beta_c}[0\connect{\n_1+\n_2\:}\partial \Lambda_n]\gtrsim \frac{1}{n^2(\log n)^4},
\end{equation}
which concludes the proof.
\end{proof}

\bibliographystyle{alpha}
\bibliography{biblio}
\end{document}